\documentclass[11pt,letterpaper]{amsart}
\usepackage{amssymb, latexsym, graphicx, mathrsfs, enumerate, amsmath, amsthm, textcomp, url, tikz-cd, bbm, mathtools, multirow, multicol, bbm, etoolbox}
\usepackage[T1]{fontenc} 
\usepackage{mathdots}
\usepackage{comment}
\usepackage[toc,page]{appendix}
\usepackage[breaklinks=true]{hyperref}
\hypersetup{
    colorlinks=true,
    linkcolor=blue,
    filecolor=magenta,      
    urlcolor=cyan,
}

\tikzset{
    labl/.style={anchor=south, rotate=90, inner sep=.5mm}
}

\usepackage{color}

\input xy
\xyoption{all}

\allowdisplaybreaks
\allowdisplaybreaks

\numberwithin{equation}{section}
\newtheorem{theorem}{Theorem}[section]
\newtheorem{proposition}[theorem]{Proposition}
\newtheorem{corollary}[theorem]{Corollary}
\newtheorem{lemma}[theorem]{Lemma}

\newtheorem*{remark*}{Remark}

\theoremstyle{definition}
\newtheorem{definition}[theorem]{Definition}

\newtheorem{remark}[theorem]{Remark}

\newcommand{\Q}{\mathbb{Q}}

\newcommand{\R}{\mathbb{R}}
\newcommand{\Z}{\mathbb{Z}}
\newcommand{\Br}{\mathrm{Br\ }}
\newcommand{\Pic}{\mathrm{Pic}\ }
\newcommand{\mfo}{\mathfrak{o}}
\newcommand{\diag}{\mathrm{diag}}

\makeatletter
\def\ops@declare#1{\expandafter\DeclareMathOperator\csname #1\endcsname{#1}}
\def\ops@scan#1,{\ifx#1\relax\let\ops@next\relax\else\ops@declare{#1}\let\ops@next\ops@scan\fi\ops@next}
\newcommand{\DeclareMathOperators}[1]{\ops@scan#1,\relax,}
\makeatother
\DeclareMathOperators{Hom,Nat,Tor,Ann,Aff,Aut,GL,SL,SO,Spec,lt,lp,lc,im,lcm,id,Set,Ab,Grp,Ring,Mod,Top,Vect,Ind,res,Gal,coker,sh,Supp, diag,Stab,inv,Nm,ord,End}
\def\makebb#1{\expandafter\def\csname b#1\endcsname{{\mathbb{#1}}}\ignorespaces}
\def\makebf#1{\expandafter\def\csname bf#1\endcsname{{\mathbf{#1}}}\ignorespaces}
\def\makegr#1{\expandafter\def\csname f#1\endcsname{{\mathfrak{#1}}}\ignorespaces}
\def\makescr#1{\expandafter\def\csname s#1\endcsname{{\mathscr{#1}}}\ignorespaces}
\def\makec#1{\expandafter\def\csname c#1\endcsname{{\mathcal{#1}}}\ignorespaces}
\def\makecal#1{\expandafter\def\csname cal#1\endcsname{{\mathcal{#1}}}\ignorespaces}
\def\doLetters#1{#1A #1B #1C #1D #1E #1F #1G #1H #1I #1J #1K #1L #1M
	#1N #1O #1P #1Q #1R #1S #1T #1U #1V #1W #1X #1Y #1Z}
\def\doletters#1{#1a #1b #1c #1d #1e #1f #1g #1h #1j #1k #1l #1m
	#1n #1o #1p #1q #1r #1s #1t #1u #1v #1w #1x #1y #1z}
\doLetters\makebb
\doLetters\makecal
\doLetters\makec
\doLetters\makescr
\doLetters\makebf
\doletters\makebf
\doletters\makegr

\def\abs#1{\lvert#1\rvert}
\def\norm#1{\lVert#1\rVert}

\begin{document}

\title[Diophantine Analysis and Arthur's Trace formula]{Diophantine Analysis and Arthur's Trace formula}



\author[Yuchan Lee]{Yuchan Lee}
\thanks{The author was supported by Samsung Science and Technology Foundation under Project Number SSTF-BA2001-04.}

\address{Yuchan Lee \\  Department of Mathematics, POSTECH, 77, Cheongam-ro, Nam-gu, Pohang-si, Gyeongsangbuk-do, 37673, KOREA}

\email{yuchanlee329@gmail.com}

\begin{abstract}



Let $X$ be a $G$-homogeneous space over a number field $k$ such that $X\cong G_\gamma\backslash G$. Here, $G$ is a simply connected semisimple group over $k$ and $\gamma\in G(k)$ whose centralizer $G_\gamma$ is a maximal torus in $G$ which is anisotropic over $k$.
We formulate the asymptotic for the number of integral points on $X$ bounded by a fixed norm $T>0$ as $T\rightarrow \infty$ in terms of $\kappa$-orbital integrals, which play a role in the stabilization of Arthur’s trace formula.
This formula coincides with the contribution of the stable conjugacy class of $\gamma$ over $k$ to the geometric side of the trace formula.

As an application, we obtain an asymptotic formula for the number of $n \times n$ matrices over the ring of integers $\mathcal{O}_k$ whose characteristic polynomial equals a fixed irreducible polynomial $\chi(x)$ of degree $n$. This result generalizes a case studied by Eskin–Mozes–Shah (1996).
\end{abstract}
\maketitle
\tableofcontents

\section{Introduction}
Let $G$ be a connected reductive group over a number field $k$ and let $X$ be a right homogeneous space of $G$ over $k$.
Assume that $X$ has a $k$-point $x_0$ and that its stabilizer $H:=G_{x_0}$ is connected and reductive. 
Then it follows from \cite[Theorem 4.17]{VP} that $X$ is affine. 
We fix a closed immersion \[X\hookrightarrow \mathrm{Spec}(k[t_1,\cdots,t_n]).\]
Via this embedding, any point of $X(A)$ can be regarded as a point in $A^n$ for any $k$-algebra $A$.
For a point $x=(x_1,\cdots,x_n)\in X(k)\subset k^n$, set
\begin{equation}\label{intro:eq:norm}
\norm{x}_{\infty}:=\max_{v\in\infty_k}\norm{x}_v
~~~~~~\text{ where }~~
\norm{x}_v:=\max_{1\leq i \leq n}\{|x_i|_v\},
\end{equation}for the set of infinite places $\infty_k$ of $k$.
 
Let $\mathbf{X}$ be a separated scheme of finite type over $\mathcal{O}_k$ such that $\mathbf{X}\otimes _{\mathcal{O}_k}k=X$. 
We refer to such $\mathbf{X}$ as a model of $X$.
Define a function $f_{\mathbf{X},T}$ on $X(\mathbb{A}_k)$ by
\begin{equation}\label{def:alg_ftn}
f_{\mathbf{X},T}:=\mathbbm{1}_{X(k_\infty,T)\times \prod_{v<\infty_k}\mathbf{X}(\mathcal{O}_{k_v})},
\end{equation}
where $X(k_\infty,T):= \{x\in X(k_\infty)\mid \norm{x}_\infty \leq T\}$ and $k_\infty:=\prod_{v\in\infty_k}k_v$.
Set
\[
N(X,f_{\mathbf{X},T})
:=|\{x\in \mathbf{X}(\mathcal{O}_k)\mid \norm{x}_{\infty}\leq T \}|.
 \]
In this paper, we study the asymptotic behavior of $ N(X;f_{\mathbf{X},T})$ in the case where $X\cong G_\gamma \backslash G$.
Here, $G$ is semisimple and simply connected, and $\gamma$ is an element in $G(k)$ whose stabilizer $G_\gamma$ is a maximal torus in $G$ which is anisotropic over $k$, i.e. has no nontrivial $k$-split subtorus. 

Our main result establishes an asymptotic formula of $N(X;f_{\mathbf{X},T})$ in terms of $\kappa$-orbital integrals, which arise as a key ingredient in the stabilization of Arthur’s trace formula (see Theorem \ref{intro:eq:mainthm}).
This asymptotic formula suggests a relation between the asymptotic behavior of $N(X;f_{\mathbf{X},T})$ and Arthur's trace formula (see (\ref{intro:eq:corr_with_tf})).
Thus, it allows us to apply the fundamental lemma and the endoscopic transfer of functions, two powerful tools in the stabilization of the trace formula, thereby rewriting $N(X;f_{\mathbf{X},T})$ in terms of stable orbital integrals of smaller groups, called "endoscopic groups" (see Corollary \ref{intro:cor:main_cor}). 

\subsection{Backgrounds}
Firstly, we review the historical background on the asymptotic analysis of $N(X;f_{\mathbf{X},T})$ for a $G$-homogeneous space $X$.
In order to better explain our main result, we then provide a brief introduction to Arthur’s trace formula and to the stabilization of the relevant part of its geometric side.

\subsubsection{\textbf{Asymptotic formula of $ N(X;f_{\mathbf{X},T})$}}
In the case that $k=\mathbb{Q}$, the study of the asymptotic behavior of $ N(X;f_{\mathbf{X},T})$ reduces to a classical problem in Diophantine analysis - namely, the distribution of integer solutions to a Diophantine system.
This problem has been extensively studied via the Hardy-Littlewood circle method (see \cite{Sch}, \cite{Vau}) or automorphic methods (see \cite{FMT}, \cite{DRS}).

In \cite{Sch}, for a function $A(T)$ such that $\lim_{T\rightarrow\infty} \frac{ N(X;f_{\mathbf{X},T})}{A(T)}=1$, Schmidt anticipated that it takes the form of an Euler product whose local factors are given by local densities of $\mathbf{X}$. 
This allows the global asymptotic problem, over $\mathbb{Q}$, to be reduced to local computations, over $\mathbb{Q}_p$ or $\mathbb{R}$.
We refer to this expectation as the Hardy-Littlewood expectation. 
Borovoi and Rudnick \cite{BR} reformulated this Hardy-Littlewood expectation, using Tamagawa measure $m^X$ on $X(\mathbb{A}_\mathbb{Q})$ (see Section \ref{sec:tamagawa_measure}).
They mainly studied a $G$-homogeneous space $X\cong H\backslash G$ under the following assumptions:
\begin{enumerate}
    \item $G$ is a simply connected semisimple group over $\mathbb{Q}$ such that $G'(\mathbb{R})$ is not compact for any non-trivial simple factor $G'$ of $G$.
    \item $H$ is connected and reductive without a non-trivial $\mathbb{Q}$-character.
    \item $X$ satisfies the equidistribution property (\ref{eq:equidistribution_property}) over $\mathbb{Q}$.
\end{enumerate}
For $X$ satisfying (1) and (2), they proved that the Hardy-Littlewood expectation holds if $H$ is simply connected.
When $H$ is not simply connected, this expectation fails in general.
Nevertheless, the local-global principle remains valid for such $X$, in the sense that the global asymptotic problem still reduces to local computations (see \cite[Theorem 0.6]{BR}). 



For a general number field $k$, Wei and Xu generalized this local-global principle
in \cite{WX}.
By using strong approximation with Brauer-Manin obstruction developed in \cite{BD} and \cite{CtX}, they analyzed how the asymptotic behavior of $ N(X;f_{\mathbf{X},T})$ differs from the expected Euler product form.
Based on this observation, they formulated the asymptotic behavior of $ N(X;f_{\mathbf{X},T})$ in terms of an integral on $X(\mathbb{A}_k)$. 
Their result is stated in Theorem \ref{thm:wx}.

\subsubsection{\textbf{Arthur's trace formula and its stabilization}}
Arthur's trace formula provides a description of the characters of representations of $G(\mathbb{A}_k)$ which appear in the discrete automorphic spectrum of $G$ in terms of the geometric data, known as the orbital integral - the integral over the $G(\mathbb{A}_k)$-conjugacy class of an element in $G(k)$.
The trace formula is the identity (see (\ref{eq:Arthurs_trace_formula})):
\begin{equation}\label{intro:eq:Arthurs_trace_formula}
J_{spec}(f)=J_{geom}(f),
\end{equation}
between two distributions taking a test function $f\in \mathcal{C}_c^\infty(G(\mathbb{A}_k))$, where $J_{spec}(f)$ is called the spectral side, which consists of the contributions from representations of $G(\mathbb{A}_k)$, while $J_{geom}(f)$ is called the geometric side and is given by a sum indexed over semisimple conjugacy classes of $G(\mathbb{A}_k)$.

The trace formula has played a central role in the Langlands program, by means of comparing the spectral data of different groups or directly computing the dimension of the space of automorphic forms. 
In various applications (see \cite{JacLang}, \cite{LL}, \cite{Lang63}, \cite{CT20}, \cite{Art13}), the trace formula needs to be described more explicitly, or to be invariant under conjugacy, or even invariant under stable conjugacy (see Definition \ref{def:stable_conjugacy}).
Therefore, 
its stabilization - a further refinement toward the stable trace formula (see \cite[Corollary 29.9.(b)]{Art}) - are required.

For a simply connected semisimple group $G$, an element $\gamma\in G(k)$ is said to be elliptic regular if its stabilizer $G_\gamma$ is a maximal torus in $G$ which is anisotropic over $k$.
The contribution from conjugacy classes of elliptic regular elements to the geometric side is a weighted sum of usual orbital integrals (see (\ref{eq:specail_part_of_trace_formula})). 
Its stabilization involves roughly two steps.
The first step, called pre-stabilization, rewrites the geometric side into a sum of $\kappa$-orbital integrals (see Section \ref{sec:pre_stabilization}). 
The second step transfers these modified terms into stable orbital integrals of endoscopic groups (see \ref{sec:endoscopic_transfer}).



\subsection{Main Result}\label{sec:intro:mainresult}
Suppose that $G$ is semisimple and simply connected, and that $X\cong G_\gamma \backslash G$ where $\gamma$ is an elliptic regular element of $G(k)$. 
Then $X$ is identified with the conjugacy class of $\gamma$ in $G$ over $k$, which is known to be a closed subscheme of $G$ over $k$ (see Remark \ref{rmk:conjugacy_and_pointsofX}).
Accordingly, $X(\mathbb{A}_k)$ corresponds to the stable conjugacy class of $\gamma$ in $G(\mathbb{A}_k)$. 
Let $\tilde{f}_{\mathbf{X}, T}$ be a function on $G(\mathbb{A}_k)$ such that
\begin{equation}\label{eq:test_function_on_G}
\tilde{f}_{\mathbf{X}, T}(x)=f_{\mathbf{X}, T}(x) \text{ for $x\in X(\mathbb{A}_{k})$},
\end{equation}
for a model $\mathbf{X}$ of $X$ over $\mathcal{O}_k$ and $T>0$.
\subsubsection{\textbf{An asymptotic formula in terms of $\kappa$-orbital integrals}}
Our main theorem is as follows. 
\begin{theorem}[Theorem \ref{thm:main}]\label{intro:thm:main} 
    
    If $G'(k_\infty)$ is not compact for any non-trivial simple factor $G'$ of $G$ and $X$ satisfies the euqidistribution property (\ref{eq:equidistribution_property}), 
    then we have 
    \begin{equation}\label{intro:eq:mainthm}
     N(X;f_{\mathbf{X},T})\sim \sum_{\kappa\in H^1(k,G_\gamma(\mathbb{A}_{\bar{k}})/G_\gamma(\bar{k}))^*}\mathcal{O}_{\gamma}^\kappa(\tilde{f}_{\mathbf{X}, T}),\end{equation} 
    where 
    $\mathcal{O}_{\gamma}^\kappa(\tilde{f}_{\mathbf{X}, T})=
    \sum\limits_{\substack{\gamma'\in G(\mathbb{A}_k)/\sim_{\mathbb{A}}\\ \gamma'\sim_{st,\mathbb{A}}\gamma}}\kappa(inv(\gamma',\gamma))\mathcal{O}_{\gamma'}(\tilde{f}_{\mathbf{X}, T})$.
    Here, the explanation of these terms is as follows:
    \begin{itemize}
    \item For two functions $A(T)$ and $B(T)$ for $T$, we write $A(T)\sim B(T)$ if 
$\lim\limits_{T\rightarrow \infty}\frac{A(T)}{B(T)}=1$.
     
        \item $\sim_\mathbb{A}$ denotes the conjugacy relation in $G(\mathbb{A}_k)$ and $ \sim_{st, \mathbb{A}}$ denotes the stable conjugacy relation in $G(\mathbb{A}_k)$.
    \item For $\gamma'\in G(\mathbb{A}_k)$ such that $\gamma'\sim_{st,\mathbb{A}} \gamma$, the invariant $inv(\gamma',\gamma)$ is defined in Section \ref{subsec:orbit}.
    \item $\mathcal{O}_{\gamma'}(\tilde{f}_{\mathbf{X}, T})$ is the orbital integral for $\tilde{f}_{\mathbf{X}, T}$ at $\gamma'$ with respect to the Tamagawa measure on $X(\mathbb{A}_k)$ (see (\ref{eq:orbital_integral})).
    \end{itemize}
\end{theorem}

\begin{remark}\label{rmk:equidistributionwhenGgamma}
    By Remark \ref{rmk:equidistribution_listup}, in order to verify the equidistribution property \eqref{eq:equidistribution_property} for $X$, it suffices to show that the sets $x\cdot G(k_\infty)\cap X(k_\infty,T)$ satisfy the nonfocusing condition in the sense of \cite[Definition 1.14]{EMS}.
    In particular, when $X$ is an $\mathrm{SL}_n$-homogeneous space under the additional assumptions specified in Section \ref{intro:application}, it satisfies the equidistribution property (see Remark \ref{rmk:equidistribution_for_SLn}).
\end{remark}

    In particular, for a suitable choice of $\mathbf{X}$ defining $f_{\mathbf{X},T}$, we establish a relation between the right-hand side of (\ref{intro:eq:mainthm}) and the geometric side of the trace formula.
    Let $\mathbf{G}$ be an affine model of $G$ over $\mathcal{O}_k$ 
    and let $\mathbf{X}$ be the schematic closure of $X$ in $\mathbf{G}$, appearing in \cite[Lemma 2.4.3]{GH19}.
    
    We take the function $f_{\mathbf{G},T}\in \mathcal{C}_c^\infty(G(\mathbb{A}_k))$ as defined in (\ref{eq:choice_of_G_integral}).
    Then the asymptotic equivalence (\ref{intro:eq:mainthm}) still holds when we replace   $\tilde{f}_{\mathbf{X},T}$ with $f_{\mathbf{G},T}$.
    In this case, the right-hand side of (\ref{intro:eq:mainthm}) coincides with the contribution to the trace formula of the stable conjugacy class of $\gamma$, after pre-stabilization (see (\ref{eq:pre_stabilization})).
    By reversing the steps for the pre-stabilization, we obtain the following asymptotic equivalence (see Remark \ref{rmk:main}):
\begin{equation}\label{intro:eq:corr_with_tf}
 N(X;f_{\mathbf{X},T})\sim \sum_{\substack{\gamma'\in G(k)/\sim_k\\ \gamma'\sim_{k,st}\gamma}}\tau(G_{\gamma'})\mathcal{O}_{\gamma'}(f_{\mathbf{G},T})=\sum_{\substack{\mfo'\in\mathcal{O}\\ \mfo'\sim_{k,st}\mfo}}J_{\mfo'}(f_{\mathbf{G},T}),
\end{equation}where $\sim_k$ denotes the conjugacy relation in $G(k)$ and $\sim_{k,st}$ denotes the stable conjugacy relation in $G(k)$.
\begin{remark}
The asymptotic equivalence (\ref{intro:eq:corr_with_tf}) suggests that, as $T\rightarrow \infty$, the main term of $ N(X;f_{\mathbf{X},T})$ should reflect automorphic data from the spectral side of the trace formula.
However, determining precisely which automorphic data are encoded in this term is difficult, due to the reasons mentioned in Remark \ref{rmk:main}.

On the other hand, by considering the spectral expansion of a function on $G(k)\backslash G(\mathbb{A}_k)$ associated with $ N(X;f_{\mathbf{X},T})$, one may obtain another spectral interpretation of the asymptotic behavior of $ N(X;f_{\mathbf{X},T})$.
According to \cite{Get18}, if such a spectral interpretation exists, then as $T\rightarrow \infty$, the main term of $ N(X;f_{\mathbf{X},T})$ is expected to correspond to the contribution of the trivial representation.
In Appendix \ref{app:spectral_expansion}, we show that our main results provide supporting evidence for this expectation.
\end{remark}

\subsubsection{\textbf{An asymptotic formula in terms of stable orbital integrals}}
Furthermore, using the endoscopic transfer (see Section \ref{sec:endoscopic_transfer}), the right-hand side of (\ref{intro:eq:corr_with_tf}) is rewritten in terms of stable orbital integrals of endoscopic groups, yielding a form that is especially suitable for the application to be carried out in Section \ref{intro:application}.
\begin{corollary}(Corollary \ref{cor:main_cor})\label{intro:cor:main_cor}
Under the same assumptions with Theorem \ref{intro:thm:main},
let $\mathbf{G}$ be an affine model of $G$ and $\mathbf{X}$ be the schematic closure of $X$ in $\mathbf{G}$. 
We have
\[
     N(X;f_{\mathbf{X},T})\sim \sum_{\substack{\kappa\in H^1(k,G_{0,\gamma_0}(\mathbb{A}_{\bar{k}})/G_{0,\gamma_0}(\bar{k}))^*\\(\gamma_0, \kappa)\mapsto((H,s,\eta),\ \gamma_{0,H})\\}}
    \mathcal{SO}_{\gamma_{0,H}}(f^H),
\]
where $G_0$ is the quasi-split inner form of $G$ and $G_{0,\gamma_0}$ is the centralizer of $\gamma_0$ in $G_0$.
Here, the explanation of these terms is as follows:
\begin{itemize}
    \item 
$\gamma_0\in G_0(k)$ lies in the stable conjugacy class in $G_0(k)$ corresponding to that of $\gamma \in G(k)$ under the map in \cite[Section 6]{Kot82}.
\item $(\gamma_0, \kappa)\mapsto ((H,s,\eta),\gamma_{0,H})$ denotes the correspondence in \cite[Lemma 9.7]{Kot86}.
\item
$\mathcal{SO}_{\gamma_{0,H}}(f^H)$ is the stable orbital integral at $\gamma_{0,H}$ for a function $f^H\in\mathcal{C}_c^\infty(H(\mathbb{A}_k))$ (see Definition \ref{def:stable_conjugacy}). 
Here, $f^H$ matches the function $f_{\mathbf{G},T}$ defined in (\ref{eq:choice_of_G_integral}),
in the sense of Theorem \ref{thm:smooth_transfer_and_fundamental_lemma}.
\end{itemize}
\end{corollary}
\subsubsection{\textbf{Sketch of proof}}
We briefly sketch the proof of Theorem \ref{intro:thm:main}.
Our strategy is based on Theorem \ref{thm:wx}, the main result of Wei and Xu in \cite{WX}, which allows us to reduce the problem into the integration of $f_{\mathbf{X},T}$ over $X(\mathbb{A}_k)^{\mathrm{Br}}$.
Here, $X(\mathbb{A}_k)^{\mathrm{Br}}$ denotes the subset of $X(\mathbb{A}_k)$ which is annihilated by Brauer evaluations attached to elements of $\Br X:=H^2_{\acute{e}t}(X,\mathbb{G}_m)$, defined in Definition \ref{def:brauer_evaluation}.

Our key observation, based on \cite[Theorem 3.2]{CtX} and the Hasse principle for $G$ (see \cite[Definition 5.3]{Kal}), is that $X(\mathbb{A}_k)^{\mathrm{Br}}$ coincides with the subset of $X(\mathbb{A}_k)$ consisting of points whose $G(\mathbb{A}_k)$-orbit contains a $k$-point.
This allows us to reinterpret the integral over $X(\mathbb{A}_k)^{\mathrm{Br}}$ as the sum of integrals over such $G(\mathbb{A}_k)$-orbits.
By identifying $X(\mathbb{A}_k)$ with the stable conjugacy class of $\gamma$ in $G(\mathbb{A}_k)$, the integral of $f_{\mathbf{X},T}$ over an orbit $\gamma'\cdot G(\mathbb{A}_k)$ in $X(\mathbb{A}_k)$ turns out to be an orbital integral for $\tilde{f}_{\mathbf{X},T}$ at $\gamma'$.
This identification also allows us to employ tools from the pre-stabilization of the trace formula.
One such tool is the following function from the set of $G(\mathbb{A}_k)$-orbits in $X(\mathbb{A}_k)$ to $\mathbb{C}^\times$:
\[
\gamma'\cdot G(\mathbb{A}_k)\mapsto 
|H^1(k,G_\gamma(\mathbb{A}_{\bar{k}})/G_\gamma(\bar{k}))|^{-1}\sum\limits_{\kappa\in H^1(k,G_\gamma(\mathbb{A}_{\bar{k}})/G_\gamma(\bar{k}))^*}\kappa(inv(\gamma',\gamma)),\]
which vanishes if the orbit $\gamma'\cdot G(\mathbb{A}_k)$ does not contain a $k$-point, and equals 1 otherwise. 
Via these observations, we arrive at the equality (\ref{intro:eq:mainthm}). 

\subsection{Application}\label{intro:application}
Let $X\cong S_\gamma \backslash \mathrm{SL}_n$, where $S_\gamma$ is the centralizer of an elliptic regular element $\gamma \in \mathrm{SL}_n(k)$.
Let $\chi(x)$ be the characteristic polynomial of $\gamma$.
We define \[N(X,T):= N(X;f_{\mathbf{X},T}),\]where $f_{\mathbf{X},T}$ is defined in (\ref{def:alg_ftn}) for the schematic closure $\mathbf{X}$ of $X$ in $\mathrm{SL}_{n,\mathcal{O}_k}$.
Then $X$ is identified with the variety over $k$ which represents the set of $n\times n$ matrices whose characteristic polynomial equals $\chi(x)$.

Note that $\gamma \in \mathrm{SL}_n(k)$ is an elliptic regular element if and only if its characteristic polynomial $\chi(x)$ is irreducible.
For a modified norm $\norm{\cdot}_\infty$ defined by
\begin{equation}\label{intro:eq:modified_norm}
\norm{x}_{\infty}:=\max_{v\in\infty_k}\norm{x}_v
~~~~~~\text{ where }~~
\norm{x}_v = \sqrt{\sum_{1\leq i,j\leq n} |x_{ij}|_v^2},
\end{equation}
Eskin-Mozes-Shah \cite[Theorem 1.1 and Theorem 1.16]{EMS} estabilished the asymptotic formula for $N(X,T)$:
\begin{equation}\label{intro:eq:EMS_result}
            N(X,T)\sim \frac{2^{n-1}h_{K} R_{K} w_n}{\sqrt{\Delta_{\chi}}\cdot
            \prod_{k=2}^n\Lambda(k/2)}T^{\frac{n(n-1)}{2}}, 
\end{equation}
under the following restrictions
\begin{equation}\label{intro:eq:conditionon_EMS}\left\{
\begin{array}{l}
\textit{(1) $k=\Q$}; \\
\textit{(2) $\mathbb{Z}[x]/(\chi(x))$ is the ring of integers of $K:=\mathbb{Q}[x]/(\chi(x))$};\\ 
\textit{(3) $\mathbb{Q}[x]/(\chi(x))$ is totally real, equivalently $\chi(x)$ splits completely over $\R$}.
\end{array}\right.\end{equation}
Here, we use the following notations 
\begin{itemize}
    \item $h_{K}$ is the class number of the ring of integers $\mathcal{O}_K$ of $K$ and $R_{K}$ is the regulator of $K$.
    \item $\Delta_{\chi}$ is the discriminant of the polynomial $\chi(x)$, $w_n$ is the volume of the unit ball in $\mathbb{R}^{\frac{n(n-1)}{2}}$, and $\Lambda(s)=\pi^{-s}\Gamma(s)\zeta(2s)$.
\end{itemize}

In particular, the condition (2) simplifies the computation of local integrals (see \cite[Corollary 6.5]{WX}); without it, such computations become highly nontrivial.
By applying Theorem \ref{intro:thm:main}, we generalize the first two restrictions (1) and (2) in (\ref{intro:eq:conditionon_EMS}), thereby extending the previous result to a broader setting.
Significantly, our approach allows the aforementioned local computations to be formulated in terms of orbital integrals without assuming the condition (2).

More precisely, under the following conditions:
\begin{equation}\label{i:intro:condition_ours}
\left\{
\begin{array}{l}
     \textit{(1) $k$ and $K\left(:=k[x]/(\chi(x))\right)$ are totally real number fields};  \\
     \textit{(2) if $k\neq \mathbb{Q}$, then  $\chi(x)$ is of prime degree $n$}
     ;\\
     \textit{(3) if $k=\mathbb{Q}$, then there is \textbf{no restriction} on $\chi(x)$},
\end{array}
\right.
\end{equation}
we describe the asymptotic formula for $N(X, T)$ using orbital integrals of $\mathrm{GL}_n$.
Here, the condition (2) for $k\neq \mathbb{Q}$ is imposed for technical reasons: it eases the computation of endoscopic groups and corresponding stable orbital integrals.
When $k=\mathbb{Q}$, this condition is not required. 
Hence, our result extends the work of Eskin–Mozes–Shah under the assumptions in (\ref{intro:eq:conditionon_EMS}).

\begin{theorem}[Theorem \ref{thm:application}]\label{intro:thm:application}
Let $\chi(x)\in\mathcal{O}_k[x]$ be an irreducible polynomial such that $\chi(0)=1$.
Let $\mathbf{X}$ be an $\mathcal{O}_k$-scheme representing the set of $n\times n$ integral matrices whose characteristic polynomial is $\chi(x)$ and $X:=\mathbf{X}\otimes_{\mathcal{O}_k}k$.
We define \[N(X, T)=|\{x\in \mathbf{X}(\mathcal{O}_k)\mid \norm{x}_\infty\leq T\}|,\]for $T>0$, where the norm $\norm{\cdot}_\infty$ is defined in (\ref{intro:eq:modified_norm}).
Under the condition (\ref{i:intro:condition_ours}), we have the following asymptotic formulas.
\begin{enumerate}
        \item If $K/k$ is not Galois or ramified Galois, then
        \[ N(X, T) \sim C_T\prod_{v<\infty_k}\frac{\mathcal{O}_{\gamma,d\mu_v}^{\mathrm{GL}_n}(\mathbbm{1}_{\mathrm{GL}_n(\mathcal{O}_{k_v})})}{q_v^{S_v(\gamma)}}.
        \]
    
        \item If $K/k$ is unramified Galois, then
        \[N(X, T) \sim C_T \left(\prod_{v<\infty_k}\frac{\mathcal{O}_{\gamma,d\mu_v}^{\mathrm{GL}_n}(\mathbbm{1}_{\mathrm{GL}_n(\mathcal{O}_{k_v})})}{q_v^{S_v(\gamma)}}+n-1\right).\]
\end{enumerate}\noindent
Here, $v<\infty_k$ denotes the finite places of $k$, $\mathcal{O}_{\gamma,d\mu_v}^{\mathrm{GL}_n}(\mathbbm{1}_{\mathrm{GL}_n(\mathcal{O}_{k_v})})$ denotes the orbital integral of $\mathrm{GL}_n$ for $\mathbbm{1}_{\mathrm{GL}_n(\mathcal{O}_{k_v})}$ at $\gamma$ with respect to the measure $d\mu_v$ defined in (\ref{eq:quotient_measure}),
    \[C_T:= |\Delta_k|^{\frac{-n^2+n}{2}} \frac{R_K h_K \sqrt{\abs{\Delta_K}}^{-1}}{R_k h_k \sqrt{\abs{\Delta_k}}^{-1}} \left(\prod_{i=2}^n \zeta_k(i)^{-1}\right)\left( \frac{2^{n-1}w_n\pi^{\frac{n(n+1)}{4}} }{\prod_{i=1}^n \Gamma(\frac{i}{2})} T^{\frac{n(n-1)}{2}}  \right)^{[k:\Q]},\]
    and we use the following notations:
    \begin{itemize}
        \item  $R_F$ is the regulator of $F$, $h_F$ is the class number of the ring of integers $\mathcal{O}_F$ of $F$, and $\Delta_F$ is the discriminant of $F/\mathbb{Q}$ for $F=k$ or $K$. 
        \item
    $w_n$ is the volume of the unit ball in $\mathbb{R}^{\frac{n(n-1)}{2}}$, and $\zeta_k$ is the Dedekind zeta function of $k$. 
    \item $q_v$ is the cardinality of the residue field of $k_v$, and the notation $S_v(\gamma)$ is defined in Section \ref{sec:application}.
\end{itemize}
\end{theorem}
To provide context for this result, we recall existing results on  $\mathcal{O}_{\gamma,d\mu_v}^{\mathrm{GL}_n}(\mathbbm{1}_{\mathrm{GL}_n(\mathcal{O}_{k_v})})$, which are summarized in Remark \ref{rmk:about_orbitalintegral_of_GLn}.
In brief, this orbital integral is an integer given by a $q_v$-polynomial whose first leading term is expected to be $q_v^{S_v(\gamma)}$, and it equals $1$ for almost all $v<\infty_k$ by \cite{Yun13}.
In particular cases, including $n=2$ and $n=3$, explicit closed formulas were  obtained in joint works of the author (see \cite{CKL} and \cite{CHL}).

\begin{remark}
    In \cite{EMS}, the asymptotic formula for $N(X,T)$ is established under (\ref{intro:eq:conditionon_EMS}) without assuming  $\chi(0)=1$, so that $\gamma$ is not required to lie in $\mathrm{SL}_n(\mathcal{O}_k)$.
    In contrast, in our setting, we impose the additional assumption $\chi(0)=1$ together with (\ref{i:intro:condition_ours}) in order to use the $\mathrm{SL}_n$-homogeneous space structure of $X$.
    
    However, in another joint paper of the author \cite{JL}, we compute the asymptotic formula for $N(X,T)$ under (\ref{i:intro:condition_ours}) without assuming $\chi(0)=1$, by directly computing the Brauer evaluation. The resulting formula still takes the form stated in Theorem \ref{intro:thm:application}, thereby providing a full generalization of the result of \cite{EMS}. 
\end{remark}

\textbf{Organization.} We outline the structure of the article as follows. 
In Section \ref{sec:orbits_and_measures}, we review the descriptions of the $G(F)$-orbits in $X(F)$ for $F=k,k_v,$ and $\mathbb{A}_k$, and recall the Tamagawa measure on a connected reductive group and its homogeneous space.
Section \ref{sec:refined_HL_expectation} presents the result in \cite{WX} and introduce the related concepts of the Brauer evaluation. 
Arthur's trace formula and the stabilization of its elliptic regular part are discussed in Section \ref{sec:arthur's_trace_formula}.
Our main theorem is stated in Section \ref{sec:main_result}, along with an explanation of a connection between the asymptotic formula of $ N(X;f_{\mathbf{X},T})$ and Arthur's trace formula.
Finally, Section \ref{sec:application} provides an application to an $\mathrm{SL}_n$-homogeneous space. 
The supplementary explanations for the computations in Section \ref{sec:application} are provided in Appendix \ref{sec:appendix_A}.
In Appendix \ref{app:spectral_expansion}, we discuss the spectral interpretation of $ N(X;f_{\mathbf{X},T})$ using the $L^2$-expansion.
\vspace{1em}

\textbf{Acknowledgments.}
The author would like to acknowledge Sungmun Cho for assisting in initiating this work and providing valuable comments, and Fei Xu for sharing information on the Brauer–Manin obstruction and engaging in discussions. 
The author is also grateful to Jayce Getz for discussions on the description of $L^2$-expansion, to Sug Woo Shin for discussions on Arthur’s trace formula, and to Seongsu Jeon and Jungtaek Hong for various collaborative discussions.
\subsection{Notations}
\begin{itemize}
\item Let $k$ be a number field with the ring of integers $\mathcal{O}_k$, $\bar{k}$ be an algebraic closure of $k$, and $\Gamma$ be the Galois group $\mathrm{Gal}(\bar{k}/k)$.
\item Let $\Omega_k$ be the set of places of $k$ and let $\infty_k$ be the set of infinite places of $k$. We use $v<\infty_k$ to denote the finite places of $k$. 
\item For $v\in \Omega_k$, we denote the normalized absolute value associated with $v$ by $|\cdot|_v$ and the completion with respect to $|\cdot|_v$ by $k_v$.

\item Let $\mathbb{A}_k$ be the adeles of $k$ and $k_\infty:=\prod_{v\in\infty_k}k_v$.

\item 
For $v \in \Omega_k \setminus \infty_k$, we define the following notations
\[\left\{\begin{array}{l}
\mathcal{O}_{k_v}:\textit{ the ring of integers of $k_v$};\\
\pi_v:\textit{ a uniformizer of $\mathcal{O}_{k_v}$};\\
\kappa_v:\textit{ the residue field of $\mathcal{O}_{k_v}$};\\
q_v:\textit{ the cardinality of the residue field $\kappa_v$}.
\end{array}
\right.\]
For an element $x\in k_v$, the exponential order of $x$ with respect to the maximal ideal in $\mathcal{O}_{k_v}$ is written by $\ord_v(x)$. We then have that $|x|_v=q_v^{-\ord_v(x)}$.

\item 
For two functions $A(T)$ and $B(T)$ for $T$, we say that $A(T)$ and $B(T)$ is asymptotically equivalent if the following relation holds 
\begin{equation}\label{eq:asymptotically_equivalent}
\lim\limits_{T\rightarrow \infty}\frac{A(T)}{B(T)}=1 \text{, denoted by $A(T)\sim B(T)$}.
\end{equation} 

\item
For a finite field extension $E/F$ where $F$ is $\Q$ or $k_v$, let $\Delta_{E/F}$ be the discriminant ideal of $E/F$.
\begin{itemize}
    \item 
If $F = \Q$, then we denote $\Delta_{E/\Q}$ by $\Delta_E$, and $\abs{\Delta_{E}}$ by the absolute value of a generator of $\Delta_{E}$ as an ideal in $\mathbb{Z}$.
\item
If $F = k_v$, then $\ord_v(\Delta_{E/k_v})$ (resp. $|\Delta_{E/k_v}|_v$) denotes the exponential order (resp. the normalized absolute value) of a generator of $\Delta_{E/k_v}$ as an ideal in $\mathcal{O}_{k_v}$.
\end{itemize}

\item For a finite set $A$, $|A|$ denotes the cardinality of $A$.

\item$\pi_0(-)$ denotes the group of connected components, $(-)^*$ denotes the Pontryagin dual, and $(-)_{tors}$ denotes the torsion subgroup.
\end{itemize}

\section{Homogeneous space and Measure}\label{sec:orbits_and_measures}
Let $G$ be a connected reductive group over $k$ and let $X$ be a right homogeneous space of $G$ with  
\begin{equation}\label{eq:generalGXmap}
\varphi_G:G\rightarrow X,\ g\mapsto x_0 \cdot g,
\end{equation}
for a fixed point $x_0 \in X(k)$. 
Let $H$ be the stabilizer of $x_0$ in $G$, so that the map $\varphi_G$ induces an isomorphism $H\backslash G \cong X$.
We assume that $H$ is a connected and reductive.


\subsection{Orbits}\label{subsec:orbit}
For a field $F=k$ or $k_v$ with $v\in \Omega_k$, the exact sequence of $F$-varieties
\[
1\rightarrow H_F \rightarrow G_F \xrightarrow{\varphi_G} X_F \rightarrow 1 
\]
induces the long exact sequence of pointed sets
\begin{equation}\label{eq:k-long_exact_sequence}
\cdots \rightarrow G(F)\rightarrow X(F)\rightarrow H^1(F,H_F)\rightarrow H^1(F,G_F)\rightarrow \cdots.
\end{equation}
Then the set of $G(F)$-orbits in $X(F)$ is in canonical bijection with $\ker(H^1(F,H)\rightarrow H^1(F,G))$.
According to \cite[Proposition 36 in Chapter 1]{Ser2}, the connecting homomorphism in (\ref{eq:k-long_exact_sequence}) is given by
\[
X(F)\rightarrow H^1(F,H),\ x\mapsto (\tau\in \mathrm{Gal}(\bar{F}/F)\mapsto \tau(g)g^{-1}),
\]
where $g\in G(\bar{k})$ such that $x_0\cdot g=x$.
In particular, when $F=k_v$, we denote the class of cocycles in $H^1(k_v,H)$ corresponding to $x\in X(k_v)$ by $inv_v(x,x_0)$.

Replacing $F$ by the adele ring $\mathbb{A}_k$, we have the long exact sequence
\begin{equation}\label{eq:Ak-long_exact_sequence}
\cdots \rightarrow G(\mathbb{A}_k)\rightarrow X(\mathbb{A}_k)\rightarrow H^1(k,H(\mathbb{A}_{\bar{k}}))\rightarrow H^1(k,G(\mathbb{A}_{\bar{k}}))\rightarrow \cdots,
\end{equation}
where $\mathbb{A}_{\bar{k}}$ denotes the adele ring of $\bar{k}$, which is identified with a direct limit of the adele ring $\mathbb{A}_K$ of $K$ over all finite extensions $K/k$.
In this case, the set of $G(\mathbb{A}_k)$-orbits in $X(\mathbb{A}_k)$ corresponds to $\ker(H^1(k,H(\mathbb{A}_{\bar{k}}))\rightarrow H^1(k,G(\mathbb{A}_{\bar{k}})))$. 
We denote by $inv(x,x_0)$ the element in $H^1(k,H(\mathbb{A}_{\bar{k}}))$ corresponding to $x\in X(\mathbb{A}_k)$ via the connecting homomorphism $X(\mathbb{A}_k)\rightarrow H^1(k,H(\mathbb{A}_{\bar{k}}))$ in (\ref{eq:Ak-long_exact_sequence}). 

\begin{definition}
    For $F=k$ or $k_v$ with $v\in \Omega_k$, we denote $\ker(H^1(F,H)\rightarrow H^1(F,G))$ by $\ker_F(X)$.
    For the adele ring $\mathbb{A}_k$, we denote $\ker(H^1(k,H(\mathbb{A}_{\bar{k}}))\rightarrow H^1(k,G(\mathbb{A}_{\bar{k}})))$ by $\ker_\mathbb{A}(X)$.
    If there is no ambiguity, we abbreviate $\ker_F$ and $\ker_\mathbb{A}$ as $\ker_F(X)$ and $\ker_\mathbb{A}(X)$, respectively.
\end{definition}

\begin{remark}\label{rmk:ker_of_alpha}
From the embedding $k\hookrightarrow \mathbb{A}_k$, we have a natural morphism from (\ref{eq:k-long_exact_sequence}) to (\ref{eq:Ak-long_exact_sequence}).
This induces the map
\begin{equation}\label{eq:parametrizing_orbits_compare}
\alpha: \ker_k\rightarrow \ker_\mathbb{A}.
\end{equation}
The kernel of $\alpha$ parametrizes distinct $G(k)$-orbits inside of $X(k)$ which lie within a single $G(\mathbb{A}_k)$-orbit inside of $X(\mathbb{A}_k)$, via the natural embedding $X(k)\hookrightarrow X(\mathbb{A}_k)$.
\end{remark}


\subsection{Measures}\label{sec:tamagawa_measure}
On the adelic points $X(\mathbb{A}_k)$, we use the canonical measure, called the Tamagawa measure, which is induced by a top-degree volume form on $X$.
In this section, based on the $G$-homogeneous space structure $X\cong H\backslash G$, we define the Tamagawa measure on $X(\mathbb{A}_k)$ which is compatible with those on $G(\mathbb{A}_k)$ and $H(\mathbb{A}_k)$.
Our construction mainly follows \cite[Section 2]{Wei82} and \cite[Section 1]{BR}.

\begin{definition}\label{def:gauge_form}
For a geometrically irreducible non-singular algebraic variety $V$ over $k$,
a gauge form on $V$ is a nowhere-zero and regular differential form in $\bigwedge^{\dim V}\Omega_{V/k}(V)$, where $\Omega_{V/k}$ denotes a cotangent sheaf of $V$.
We say a gauge form is invariant if it is translation-invariant.
\end{definition}

Let $\omega_G$, $\omega_H$, and $\omega_X$ be an invariant gauge form on $G$, an invariant gauge form on $H$, and a $G$-invariant gauge form on $X$, respectively. 
    Let $\varphi_G^*\omega_X$ be the pullback of $\omega_X$ along the map $\varphi_G$ in (\ref{eq:generalGXmap}), and let $\widetilde{\omega}_H$ be a lifting of $\omega_H$ on $G$ such that, for every $g\in G$, $\widetilde{\omega}_H(gh)$ induces on $H$ the form $\omega_H(h)$ for $h\in H$.
    Then the form $\varphi_G^*\omega_X\wedge \widetilde{\omega}_H$ is a gauge form on $G$ which is independent of the choice of a lifting $\widetilde{\omega}_H$.
\begin{definition}[{\cite[Section 2.4, p24]{Wei82}}]\label{deg:algmatch}
    We say that $\omega_G$, $\omega_H$, and $\omega_X$ match together algebraically, if $\omega_G=\varphi_G^*\omega_X\wedge \tilde{\omega}_H$ and denote it by $\omega_G=\omega_X\cdot \omega_H$.
\end{definition}
Since $G$ and $H$ are connected reductive groups over $k$, they are unimodular.
As shown in \cite[Corollary of Theorem 2.2.2]{Wei82}, $G$ and $H$ admit invariant gauge forms.
Moreover, $X$ admits a $G$-invariant gauge form $\omega_X$ by \cite[Section 1.4]{BR}.
For any choice of invariant gauge forms $\omega_G$ on $G$ and $\omega_H$ on $H$, we may multiply $\omega_X$ by a suitable constant in $k^\times$ so that $\omega_G = \omega_X \cdot \omega_H$, as explained in \cite[p.~24]{Wei82}.

For a gauge form  $\omega_V$ on a variety $V$ over $k$ considered in Definition \ref{def:gauge_form}, we denote by $|\omega_V|_v$ the measure on $V(k_v)$ induced from the canonical Haar measure on $k_v$ for each $v \in \Omega_k$ (see \cite[Section 2.2.1]{Wei82}).
We then have the following proposition, presented right after \cite[Lemma 1.6.4]{BR}.
\begin{proposition}\label{prop:alg-top_match}
    Let $\omega_G,\omega_H$, and $\omega_X$ be an invariant gauge form on $G$, an invariant gauge form on $H$, and a $G$-invariant gauge form on $X$, respectively, such that $\omega_G=\omega_X\cdot \omega_H$. 
    Then we have
    \begin{equation}\label{eq:topologicallymatch}
    \int_{G(k_v)}f(g)\ |\omega_G|_v=\int_{H(k_v)\backslash G(k_v)}\int_{H(k_v)}f(hg) \ |\omega_H|_v  |\omega_X|_v.
    \end{equation}
    Here we note that $H(k_v)\backslash G(k_v)$ is identified with an open subset in $X(k_v)$.
    For a triple of measures $(|\omega_G|_v,|\omega_H|_v,|\omega_X|_v)$ satisfying (\ref{eq:topologicallymatch}), we say that $|\omega_G|_v,|\omega_H|_v,$ and $|\omega_X|_v$ match together topologically.
\end{proposition}


For a connected reductive group $K$, let $\rho_K$ denote the representation of $\mathrm{Gal}(\bar{k}/k)$ in the space $X^*(K)\otimes\mathbb{Q}$ and let $t_K$ be the rank of $X^*(K)$, where $X^*(K)$ denotes the group of $k$-characters of $K$.
We define the following factors:
\[\left\{
\begin{array}{l}
     \lambda_v^K=L_v(1,\rho_K)^{-1}\text{ and}; \\
     r_K=\lim_{s\rightarrow 1}(s-1)^{t_K}L(s,\rho_K),
\end{array}
\right.\] where $L(s,\rho_K)=\prod_{v<\infty_k}L_v(s,\rho_K)$ is the Artin $L$-function associated with $\rho_K$.

\begin{definition}\label{def:tamagawameasure}
\begin{enumerate}
    \item For a connected reductive group $K$, let $\omega_K$ be an invariant gauge form on $K$.
The Tamagawa measure $m^K$ on $K(\mathbb{A}_k)$ is defined as follows
\[
\left\{
\begin{array}{l}
     m_{fin}^K=r_K^{-1}|\Delta_k|^{-\frac{1}{2}\dim K}\prod_{v<\infty_k}(\lambda_v^K)^{-1}|\omega_K|_v; \\
     m_\infty^K=\prod_{v\in \infty_k}|\omega_K|_v;\\
     m^K=m_\infty^K m_{fin}^K.
\end{array}
\right.
\]
\item
Let $\omega_X$ be a $G$-invariant gauge form on $X\cong H\backslash G$.
The Tamagawa measure $m^X$ on $X(\mathbb{A}_k)$ is defined as follows
\[
\left\{
\begin{array}{l}
     m_{fin}^X=r_X^{-1}|\Delta_k|^{-\frac{1}{2}\dim X}\prod_{v<\infty_k}(\lambda_v^X)^{-1}|\omega_X|_v; \\
     m_\infty^X=\prod_{v\in \infty_k}|\omega_X|_v;\\
     m^X=m_\infty^X m_{fin}^X.
\end{array}
\right.
\]
where $r_X=r_G/r_H$ and $\lambda_v^X=\lambda_v^G/\lambda_v^H$.
\end{enumerate}
\end{definition}
We fix an invariant gauge form $\omega_G$ on $G$, an invariant gauge form $\omega_H$ on $H$, and a $G$-invariant gauge form $\omega_X$ on $X$ such that $\omega_G=\omega_X\cdot \omega_H$.
Then the local measures $m^G_v$ on $G(k_v)$, $m^H_v$ on $H(k_v)$, and $m^X_v$ on $X(k_v)$ match together topologically.

\section{Refined Hardy-Littlewood expectation over a number field $k$}\label{sec:refined_HL_expectation}
In this section, we recall the result of \cite{WX} on the generalization of the Hardy-Littlewood expectation for $ N(X;f_{\mathbf{X},T})$ to a number field.
By using strong approximation with Brauer-Manin obstruction, Wei and Xu proved that $ N(X;f_{\mathbf{X},T})$ is asymptotically equivalent to a sum of Euler products, indexed by the finite group $\Br X/ \Br k$ (see Remark \ref{rem:normalized_eval}.(2)).
We present a slightly modified version of this result in Theorem \ref{thm:wx}, which is more suitable for comparison with Arthur's Trace formula in Section \ref{sec:arthur's_trace_formula}.
\subsection{Brauer group and evaluation}
Firstly, we introduce some notions related to the Brauer-Manin obstruction.
We mainly follow \cite[Section 8]{Poonen} for the definitions of the Brauer group and the Brauer evaluation.

\begin{definition}\label{def:brauer}{\cite[Definition 6.6.4]{Poonen}}
    For any scheme $X$ over $k$, the Brauer group is defined by
    \[\Br X = \mathrm{H}_{\acute{e}t}^2(X, \mathbb{G}_{m}),\]
    where $\mathbb{G}_{m}$ stands for the \'etale sheaf associated with the multiplicative group over $X$ (see \cite[Proposition 6.3.19]{Poonen}).
    For an affine scheme $X=\Spec A$, we denote $\Br (\Spec A)$ by $\Br A$ to ease the notation.
\end{definition}
Note that $ \mathrm{H}_{\acute{e}t}^2(-, \mathbb{G}_{m})$ gives a contravariant functor from $\mathrm{Sch}_k$ to $\mathrm{Ab}$.
Via this fact, each element of $\Br X$ induces an evaluation on the set of $A$-points $X(A)$, for any $k$-algebra $A$.
\begin{definition}\label{def:brauer_evaluation}
Let $X$ be a $k$-scheme and $A$ be a $k$-algebra.    
    \begin{enumerate}
        \item For any $x \in X(A)$ and $\xi \in \Br X$, we define the Brauer evaluation $\xi(x)$ to be the image of $\xi$ under the morphism $\mathrm{Br}(x):\Br X\rightarrow \Br A$.

        \item In particular, for a local field $k_v$ where $v\in \Omega_k$, \cite[Theorem 1.5.34]{Poonen} yields that the invariant map $\inv_v : \Br k_v \to \Q/\Z$ is injective.
        In this case,
        for any $x \in X(k_v)$ and $\xi \in \Br X$, we define the Brauer evaluation $\xi_v(x)$ to be the image of $\xi$ under the following composite map
        \[\Br X \xrightarrow{\mathrm{Br} (x)} \Br k_v \xhookrightarrow{\inv_v} \Q / \Z.  
        \]
    \end{enumerate}
\end{definition}
\begin{remark}\label{rem:normalized_eval}
    For a $k$-scheme $X$, the structure morphism induces a map $\Br k \to \Br X$ via the contravariant functor $\mathrm{Br}(-):=\mathrm{H}_{\acute{e}t}^2(-, \mathbb{G}_{m})$.
    We denote the cokernel of this map by $\Br X /\Br k$. 
    \begin{enumerate}
         \item
    For $\xi \in im(\Br k\rightarrow \Br X)$ , the evaluation $\xi_v$ is constant on $X(k_v)$ and the summation of these constant values over all $v\in \Omega_k$ is trivial.
    Indeed, for a fixed embedding $k \hookrightarrow k_v$, the composition $\Br k\rightarrow \Br X \xrightarrow{\mathrm{Br}(x)} \Br k_v$ is independent of $x \in X(k_v)$.
    Then it follows from Definition \ref{def:brauer_evaluation}.(2) that $\xi_v(x)$ is constant on $X(k_v)$, and hence \cite[Proposition 8.2.2]{Poonen} yields that $\sum_{v\in \Omega_k}\xi_v\equiv 0$.
        \item 
    If $X\cong H\backslash G$ where $G$ is a simply connected semisimple group over $k$ and $H\subset G$ is a connected subgroup over $k$, then by Proposition 2.2 and Proposition 2.10 in \cite{CtX} one has an isomorphism \[\Br X/\Br k\cong\Pic H.\] Here, the Picard group $\Pic H$ is a finite abelian group.
    \end{enumerate}
\end{remark}

\begin{definition}
   For $\xi \in \Br X$, we define
   \[
   X(\mathbb{A}_k)^\xi:=\{(x_v)_{v\in\Omega_k}\in X(\mathbb{A}_k)\mid \sum_{v\in\Omega_k}\xi_v(x_v)=0 \}.
   \]
   Also we define
   \[
   X(\mathbb{A}_k)^{\Br}:= \bigcap_{\xi\in\Br X}X(\mathbb{A}_k)^\xi.
   \]
\end{definition}
\noindent By Remark \ref{rem:normalized_eval}.(1), one can deduce that  $X(k)\subset X(\mathbb{A}_k)^\mathrm{Br}$ (see {\cite[Corollary 8.2.6]{Poonen}}).

In the remainder of this section, we set $X\cong H\backslash G$ where $G$ is a simply connected semisimple group and $H$ is a connected reductive subgroup of $G$ over $k$ without non-trivial $k$-character.
By \cite[Theorem 3.1]{CtX}, we have an exact sequence of pointed sets
\[
H^1(k,H)\rightarrow \bigoplus_{v\in \Omega_k}H^1(k_v,H)\rightarrow \mathrm{Hom}(\Pic H,\mathbb{Q}/\mathbb{Z}).
\]
Here, the second morphism is induced by the pairing $H^1(k_v,H)\times \Pic H_{k_v} \rightarrow \Br k_v$ in \cite[(2.4)]{CtX}, which is compatible with the Brauer evaluation on $X(k_v)$.
This exact sequence and \cite[Proposition 2.9]{CtX} yield the commutative diagram (3.1) in loc. cit.:
\begin{equation}\label{eq:diagram_for_H1}
\begin{tikzcd}
G(k) \arrow[r] \arrow[d] & G(\mathbb{A}_k) \arrow[d] \\
X(k) \arrow[r] \arrow[d] & X(\mathbb{A}_k) \arrow[r] \arrow[d] & \mathrm{Hom}(\mathrm{Br}\, X / \mathrm{Br}\, k, \mathbb{Q}/\mathbb{Z}) \arrow[d] \\
H^1(k, H) \arrow[r] \arrow[d] & \bigoplus_{v\in \Omega_k} H^1(k_v, H) \arrow[r] \arrow[d] & \mathrm{Hom}(\mathrm{Pic}\, H, \mathbb{Q}/\mathbb{Z}) \\
H^1(k, G) \arrow[r] & \bigoplus_{v\in\Omega_k} H^1(k_v, G).
\end{tikzcd}
\end{equation}
In this diagram, the two vertical sequences are from the long exact sequences (\ref{eq:k-long_exact_sequence}) and (\ref{eq:Ak-long_exact_sequence}), respectively.
The map $X(\mathbb{A}_k)\rightarrow \mathrm{Hom}(\Br X/\Br k,\mathbb{Q}/\mathbb{Z})$ is induced by the Brauer evaluation, while
the map $\Hom(\Br X/\Br k,\mathbb{Q}/\mathbb{Z})\rightarrow \Hom(\Pic H,\mathbb{Q}/\mathbb{Z})$ comes from the isomorphism $\Pic H\cong \Br X/\Br k$ given in \cite[Proposition 2.10]{CtX}.


\subsection{Result of \cite{WX}}
We continue to assume that $X\cong H\backslash G $ where $G$ is a simply connected semisimple group and $H$ is a connected reductive subgroup of $G$ over $k$ without non-trivial character.
To state the main result of \cite{WX}, we introduce the following equidistribution property for a congruence subgroup $\Gamma \subset G(k)$ and any point $x\in \mathbf{X}(\mathcal{O}_k)$:
\begin{equation}\label{eq:equidistribution_property}
|\{y\in x\cdot \Gamma\mid \norm{y}_{\infty}\leq T \}|\sim
\frac{m^{H_x}_{\infty}(\Gamma\cap H_x(k)\backslash H_x(k_\infty) )}{m^G_{\infty}(\Gamma \backslash G(k_\infty))}
m_{\infty}^X(x\cdot G(k_\infty)\cap X(k_\infty,T)),
\end{equation} where $H_x$ is the stabilizer of $x$ in $G$. 
Here, $m_\infty^G$, $m_\infty^{H_x}$, and $m_\infty^X$ denote the infinite components of Tamagawa measures on $G(\mathbb{A}_k)$, $H(\mathbb{A}_k)$, and $X(\mathbb{A}_k)$, respectively. 
By \cite{BorHar}, the measures $m^G_\infty$ and $m^{H_x}_\infty$ induce the well-defined finite Haar measures on $\Gamma\cap H_x(k)\backslash H_x(k_\infty) $ and $\Gamma \backslash G(k_\infty)$, respectively.
\begin{remark}\label{rmk:equidistribution_listup}
    The equidistribution property (\ref{eq:equidistribution_property}) is verified for the following $G$-homogeneous space $X\cong H\backslash G$:
    \begin{enumerate}
        \item $G$ is almost $k$-simple, $G(k_\infty)$ is not compact, and $X$ is symmetric.
        This case was proven by \cite{DRS} and \cite{EM} over $\mathbb{Q}$, and extended to a general number field $k$ by \cite{BO}.
        \item 
        $H$ is not contained in any proper $\mathbb{Q}$-parabolic subgroup of $G$, and the sequence $vol(m^X_\infty, x\cdot G(\mathbb{R})\cap X(\mathbb{R},T_n))$, for $x \in \mathbf{X}(\mathbb{Z})$, is not focused for some sequence $T_n\rightarrow \infty$ (see \cite[Definition 1.14]{EMS}). 
        Here, $X(\mathbb{R},T_n):=\{x\in X(\mathbb{R})\mid \norm{x}_\infty \leq T_n\}$.
        This case was proven by \cite{EMS}.
    \end{enumerate}
    In the case that $X\cong G_\gamma \backslash G$ where $\gamma$ is an elliptic regular element of $G(k)$, as in Section \ref{sec:intro:mainresult}, we consider the restrictions of scalars
    \[
    X':=\mathrm{Res}_{k/\mathbb{Q}}X\text{ and } G':=\mathrm{Res}_{k/\mathbb{Q}}G.
    \]
    Then, by \cite[Corollary A.5.4]{CGP15} and \cite[Proposition 17.2.2]{GH19}, we have $X'\cong H'\backslash G'$ where $H':=\mathrm{Res}_{k/\Q}G_\gamma$.
    Since $\mathrm{Res}_{k/\Q}G_\gamma$ is anisotropic over $\mathbb{Q}$, 
    to show that $X$ satisfies the equidistribution property, it suffices to verify the nonfocusing of the sets $x\cdot G(k_\infty)\cap X(k_\infty,T)$ in the sense of \cite[Definition 1.14]{EMS}.
\end{remark}

Recall that a separated scheme $\mathbf{X}$ of finite type over $\mathcal{O}_k$ such that $\mathbf{X}\otimes_{\mathcal{O}_k}k=X$ is called a model of $X$. The main result of \cite{WX} is as follows.
\begin{theorem}[{\cite[Theorem 4.3]{WX}}]\label{thm:wx}
    Suppose that $X\cong H\backslash G$ where $G$ is semisimple and simply connected, and $H$ is connected reductive subgroup of $G$ over $k$ without non-trivial $k$-character. 
    For a model $\mathbf{X}$ of $X$ over $\mathcal{O}_{k}$ and $T>0$,
    let $f_{\mathbf{X},T}$ be the function on $X(\mathbb{A}_k)$ defined in (\ref{def:alg_ftn}).
    If $G'(k_\infty)$ is not compact for any non-trivial simple factor $G'$ of $G$ and $X$ satisfies the euqidistribution property (\ref{eq:equidistribution_property}), 
    then we have 
    \[
     N(X;f_{\mathbf{X},T})\sim |\Br X/\Br k| \int_{{X(\mathbb{A}_k)}^{\mathrm{Br}}}f_{\mathbf{X},T} \ m^X.
    \]
\end{theorem}

\begin{remark}\label{rmk:identification_of_cohomology}
By \cite[Proposition 2.10.(ii)]{CtX}, we have
\[
\mathrm{Pic}\ H\cong \Br X/\Br k.\]
On the other hand, the finite abelian groups $\Pic H$ and $(\pi_1(H)_{\Gamma})_{tors}$ are dual to each other by \cite[Proposition 6.3]{Col}, where $\pi_1(H)$ is the algebraic fundamental group of $H$ (see \cite[Section 3.3]{BR}).
Hence, we have an isomorphism $(\Br X /\Br k)^*\cong (\pi_1(H)_{\Gamma})_{tors}$

Suppose further that $H$ is a maximal torus in $G$. 
In this case, we compare $\Br X / \Br k$ with the corresponding objects introduced in Section \ref{sec:arthur's_trace_formula}.
By \cite[Section 3.4]{BR}, we can identify $(\pi_1(H)_{\Gamma})_{tors}$ with $\pi_0(\hat{H}^{\Gamma})^*$ where $\hat{H}$ denotes the connected Langlands dual group of $H$. 
Moreover, combining the isomorphism $\pi_0(\hat{H}^{\Gamma})\cong H^1(k, X_*(\hat{H}))$ from \cite[Section 2.1]{Kot86} with the Tate-Nakayama duality $H^1(k, H(\mathbb{A}_{\bar{k}})/H(\bar{k}))\cong H^1(k, X^*(H))^*$, we have
\[
\pi_0(\hat{H}^{\Gamma})^*\cong H^1(k,H(\mathbb{A}_{\bar{k}})/H(\bar{k})).
\]
It then follows that $\Br X/\ \Br k$ and $H^1(k, H(\mathbb{A}_{\bar{k}})/H(\bar{k}))$ are dual to each other.
\end{remark}

\begin{remark}\label{rmk:comparision_bet_Q_k}
    In the case that $k=\mathbb{Q}$, Theorem \ref{thm:wx} is compatible with the result of \cite[Theorem 5.3]{BR}.
Indeed, we have $|(\pi_1(H)_\Gamma)_{\mathrm{tors}}|=|(\Br X/\Br \mathbb{Q})|$ by Remark \ref{rmk:identification_of_cohomology}, and the support of the density function $\delta$ in \cite[Theorem 5.3]{BR} coincides with $X(\mathbb{A}_\mathbb{Q})^{\mathrm{Br}}$.
In fact, by the proof of \cite[Theorem 3.6]{BR}, the support of the density function $\delta$ consists of the points $x\in X(\mathbb{A}_\mathbb{Q})$ such that $inv(x,x_0)$ lies in the image of the following map in the diagram (\ref{eq:diagram_for_H1})
\[
H^1(\mathbb{Q},H)\longrightarrow \bigoplus_{v\in\Omega_\mathbb{Q}}H^1(\mathbb{Q}_v,H).
\]
We will show that the set of these points is exactly $X(\mathbb{A}_\mathbb{Q})^{\mathrm{Br}}$ in the proof of Theorem \ref{thm:main}.
    
\end{remark}

\section{Arthur's Trace formula}\label{sec:arthur's_trace_formula}
In this section, we introduce Arthur's trace formula and briefly review the stabilization of the elliptic regular part of its geometric side.
Throughout this section, we assume that $G$ is a simply connected semisimple group.
For the locally compact  group $G(\mathbb{A}_k)$, the group $G(k)$ embeds as a discrete subgroup of $G(\mathbb{A}_k)$.
We denote by $[G]$ the adelic quotient $G(k)\backslash G(\mathbb{A}_k)$ of $G$, and we consider the Hilbert space $L^2([G])$ which is defined in \cite[Section 3.7]{GH19}.
For $f\in \mathcal{C}_c^\infty(G(\mathbb{A}_k))$, Arthur's trace formula is the identity
\begin{equation}\label{eq:Arthurs_trace_formula}
\sum_{\chi\in \mathcal{X}}J_\chi(f)=\sum_{\mfo \in \mathcal{O}}J_\mfo(f),
\end{equation}
where $\mathcal{X}$ is the set of conjugacy classes of pairs $(M,\sigma)$, with $M$ a Levi subgroup of $G$ and $\sigma$ a nonzero cuspidal automorphic representation of $M(\mathbb{A}_k)$ in the $L^2$-sense (see \cite[Definition 6.10]{GH19}), while $\mathcal{O}$ is the set of semisimple conjugacy classes in $G(k)$.
In the trace formula (\ref{eq:Arthurs_trace_formula}), the left-hand side is called the spectral side while the right-hand side is called the geometric side.
For a detailed explanation of the terms $J_\chi(f)$ and $J_{\mfo}(f)$, we refer to \cite[Sections 10, 14, 19, and 21]{Art}.

In particular, let $\mathcal{X}_{\mathrm{cusp}}$ denote the set of equivalence classes of cuspidal automorphic representations of $G(\mathbb{A}_k)$ in the $L^2$-sense.  
Let $\mathcal{O}_{\mathrm{ani}}$ denote the set of anisotropic semisimple conjugacy classes in $G(k)$, where a class is called anisotropic if it does not intersect $P(k)$ for any proper parabolic subgroup $P \subset G$.
Then we have
\begin{equation}\label{eq:specail_part_of_trace_formula}
\left\{
\begin{array}{l l}
     J_\chi(f)=m_{cusp}(\pi)\mathrm{tr}\ \pi(f) &\text{for }\chi=(G,\pi)\in\mathcal{X}_{cusp};\\
J_\mfo(f)=\tau(G_\gamma)\mathcal{O}_{\gamma}(f) &\text{for }\mfo\in\mathcal{O}_{ani}\text{ and }\gamma \in \mfo.
\end{array}
\right.
\end{equation}
The explanation of these terms is as follows:
\begin{itemize}
    \item $m_{cusp}(\pi)$ is the multiplicity of $\pi$ in the cuspidal subspace $L^2_{cusp}([G])$ of $L^2([G])$ (see \cite[(6.14)]{GH19})
    and $\pi(f)$ is the representation of $G(\mathbb{A}_k)$, induced from the representation $\pi=(\pi,V_\pi)$, defined by
    \[
    \pi(f)v= \int_{G(\mathbb{A}_k)}f(x)\pi(x)v\ m^G(x) ~~~~\text{ for }v\in V_\pi.
    \]
    \item $\tau(G_\gamma):=vol(m^{G_\gamma}, G_\gamma(k)\backslash G_\gamma(\mathbb{A}_k))$ is the Tamagawa number of $G_\gamma$. 
    We note that a semisimple element $\gamma \in G(k)$ represents an anisotropic class if and only if $G_\gamma$ is anisotropic.
    In this case, the quotient $G_\gamma(k)\backslash G_\gamma(\mathbb{A}_k)$ is compact and thus $\tau(G_\gamma)$ is well-defined.
    \item $\mathcal{O}_\gamma(f)$ is the orbital integral of $G$ for $f$ at $\gamma$, defined by
    \begin{equation}\label{eq:orbital_integral}
    \mathcal{O}_\gamma(f)=\int_{G_\gamma(\mathbb{A}_k)\backslash G(\mathbb{A}_k)}f(x^{-1}\gamma x)\frac{m^G}{m^{G_\gamma}}.
    \end{equation}
\end{itemize}
If $G(k)\backslash G(\mathbb{A}_k)$ is compact, equivalently if $G$ is anisotropic, then $\mathcal{X}=\mathcal{X}_{cusp}$ and $\mathcal{O}=\mathcal{O}_{ani}$.
In this case, (\ref{eq:Arthurs_trace_formula}) reduces to Selberg's trace formula for compact quotients.

In various applications of the trace formula, the test function $f$ is not given directly but rather indirectly through its orbital integral or stable orbital integral.
Thus the stabilization of the trace formula - roughly, a modification which makes the trace formula invariant under stable conjugacy over $\mathbb{A}_k$ (see Definition \ref{def:stable_conjugacy}) - is required.
\begin{definition}\label{def:stable_conjugacy}
Let $F=k,k_v,$ or $\mathbb{A}_k$, and let $\bar{F}$ the corresponding $\bar{k},\bar{k}_v$ or $\mathbb{A}_{\bar{k}}$, respectively.
\begin{enumerate}
    \item  Two semisimple elements $\gamma,\gamma'\in G(F)$ are called stably conjugate if there exists $g\in G(\bar{F})$ 
    such that $g^{-1}\gamma g=\gamma'$.
In order to distinguish between the usual notion of conjugacy under $G(F)$ and stably conjugacy, we say two semisimple elements of $G(F)$ are rationally conjugacy if they are conjugate in $G(F)$.
\item For $F=k_v$ or $\mathbb{A}_k$, the stable orbital integral for $f\in \mathcal{C}_c^\infty(G(F))$ at $\gamma\in G(F)$ is defined as follows
\[
\mathcal{SO}_\gamma(f)=\sum_{\substack{\gamma'\in G(F)/\sim \\ \gamma' \sim_{st} \gamma}}\mathcal{O}_{\gamma'}(f),
\]where $\sim$ denotes the rational conjugacy relation and $\sim_{st}$ denotes the stable conjugacy relation.
Here, the local orbital integral $\mathcal{O}_{\gamma'}(f)$ for $f\in \mathcal{C}_c^\infty(G(k_v))$ at $\gamma'\in G(k_v)$ is defined to be \[\mathcal{O}_{\gamma'}(f)=\int_{G_\gamma(k_v)\backslash G(k_v)}f(x^{-1}\gamma x)\frac{m^G_v}{m_v^{G_\gamma}}.
\]
\end{enumerate}
\end{definition}\noindent

A semisimple element $\gamma \in G(k)$ is said to be elliptic if it is contained in a maximal torus $T\subset G$ which is elliptic over $k$, meaning that $T$ is anisotropic modulo the center of $G$.
In our setting, where $G$ is semisimple, a maximal torus is elliptic over $k$ if and only if it is anisotropic over $k$.

For a $G$-homogeneous space $X\cong G_\gamma\backslash G$, where $\gamma \in G(k)$ is elliptic regular, we establish in Section \ref{sec:main_result} a relation between the asymptotic behavior of $ N(X;f_{\mathbf{X},T})$ and the contribution of the geometric side coming from elements $\gamma'\in G(k)$ stably conjugate to $\gamma$.
To show that, we use some technical tools that arise in the stabilization of the trace formula.
Therefore, in the remainder of this section, we recall the stabilization of the elliptic regular part of the trace formula.

\begin{remark}\label{rmk:conjugacy_and_pointsofX}
For a semisimple element $\gamma \in G(k)$, let $X$ be a $G$-homogeneous space over $k$ isomorphic to $G_\gamma \backslash G$.
In this case, we may take the map $\varphi_G$ in (\ref{eq:generalGXmap}) to be
\[\varphi_G:G \rightarrow X\cong G_\gamma\backslash G, \ g\mapsto g^{-1}\gamma g,\] thereby identifying $X$ with the conjugacy class of $\gamma$ in $G$ defined over $k$, which is known to be a closed subscheme of $G$ over $k$ (see \cite[Corollary 6.13]{St2} and \cite[Example 17.2]{GH19}).
Then, for $F=k,k_v,$ and $\mathbb{A}_k$, the set of $\bar{F}$-points $X(\bar{F})$ is identified with the $G(\bar{F})$-conjugacy class of $\gamma$ in $G(\bar{F})$.
Accordingly, $X(F)$ corresponds to the stable conjugacy class of $\gamma$ in $G(F)$, and each $G(F)$-orbit in $X(F)$ corresponds to a rational conjugacy class in this stable class.
By Section \ref{subsec:orbit}, the set of rational conjugacy classes in the stable class of $\gamma$ over $F$, which is in bijection with $X(F)/G(F)$, is parametrized by $\ker_F$.
\end{remark}

\subsection{Stabilization of the elliptic regular part of the trace formula}\label{sec:stab_ell.reg}
For the stabilization of the elliptic regular part,
we follow the approach of \cite[Section 5]{Kal}.
More specifically, \cite[Section 5]{Kal} deals with the elliptic strongly regular part, where a semisimple element is said to be strongly regular if its centralizer is a maximal torus.
In our setting, the assumption that $G$ is semisimple and simply connected implies that the centralizer of any semisimple element is connected (see \cite[Theorem 8.1]{St}).
Therefore, every elliptic regular element is strongly regular.

Since the centralizer of an elliptic regular element $\gamma$ is a maximal torus which is anisotropic over $k$, the conjugacy class of $\gamma$ belongs to $\mathcal{O}_{ani}$. 
Then the elliptic regular part of the trace formula is given by
\begin{equation}\label{eq:ell.reg.tf}
\mathrm{TF}_{ell.reg}(f)=\sum_{\gamma\in G(k)_{ell.reg}/\sim}\tau(G_\gamma)\mathcal{O}_\gamma(f),
\end{equation}
where  $G(k)_{ell.reg}/\sim$ denotes the set of conjugacy classes of elliptic regular elements in $G(k)$. 

To simplify the notation, we denote the rational conjugate relation over $F$ by $\sim_F$ and the stable conjugate relation over $F$ by $\sim_{st,F}$ for $F=k,k_v$, and $\mathbb{A}_k$. 
When the base field is clear from the context, we sometimes abbreviate the notations to $\sim$ and $\sim_{st}$, respectively.







\subsubsection{Pre-stabilization}\label{sec:pre_stabilization}
In this subsection, we rewrite the right-hand side of (\ref{eq:ell.reg.tf}) in terms of $\kappa$-orbital integrals, to be transformed into stable orbital integrals via the endoscopic transfer described in Section \ref{sec:endoscopic_transfer}.

If two regular semisimple elements $\gamma$ and $\gamma'$ of $G(k)$ are stably conjugate, then their centralizers $G_\gamma$ and $G_{\gamma'}$ are isomorphic over $k$. 
Thus we get
\begin{equation}\label{eq:ell.reg.part}
\mathrm{TF}_{ell.reg}(f)=\sum_{\gamma\in G(k)_{ell.reg}/\sim_{st}}\tau(G_{\gamma})\sum_{\substack{\gamma'\in G(k)/\sim\\ \gamma'\sim_{st}\gamma}}\mathcal{O}_{\gamma'}(f).
\end{equation}
From Remark \ref{rmk:ker_of_alpha} and Remark \ref{rmk:conjugacy_and_pointsofX}, the kernel of the map $\alpha:\ker_k\rightarrow \ker_\mathbb{A}$ for $\gamma\in G(k)_{ell.reg}$ parametrizes distinct $G(k)$-classes inside of a $G(\bar{k})$-class of $\gamma$ which lies in a single $G(\mathbb{A}_k)$-class inside of a $G(\mathbb{A}_{\bar{k}})$-class of $\gamma$.
We then have
\begin{equation}\label{eq:pre_stablization_1}
\mathrm{TF}_{ell.reg}(f)=\sum_{\gamma\in G(k)_{ell.reg}/\sim_{st}}\tau(G_{\gamma}) |\ker(\alpha)|\sum_{a\in im(\alpha)}\mathcal{O}_{\gamma_a}(f),
\end{equation}
where $\gamma_a \in G(\mathbb{A}_k)$ represents the $G(\mathbb{A}_k)$-class contained in the $G(\mathbb{A}_{\bar{k}})$-class of $\gamma$ and satisfying $inv(\gamma_a,\gamma)=a$.
On the other hand, we have the following exact sequence of pointed sets in \cite[Exercise 5.4]{Kal}
\begin{equation}\label{eq:mainthm_ses}
\ker_k\xrightarrow{\alpha} \ker_\mathbb{A}\rightarrow H^1(k,G_{\gamma}(\mathbb{A}_{\bar{k}})/G_{\gamma}(\bar{k})),
\end{equation}
for $\gamma\in G(k)_{ell.reg}/\sim_{st}$ since a semisimple and simply connected group $G$ satisfies the Hasse principle.
To ease the notation, we denote   $H^1(k,G_\gamma(\mathbb{A}_{\bar{k}})/G_\gamma(\bar{k}))$ by $\mathfrak{K}(G_\gamma/k)$.

By the finiteness of $\mathfrak{K}(G_{\gamma}/k)$ (see Remark \ref{rmk:identification_of_cohomology}), a function $|\mathfrak{K}(G_\gamma/k)|^{-1}\sum\limits_{\kappa\in \mathfrak{K}(G_\gamma/k)^*}\kappa$ on $\mathfrak{K}(G_{\gamma}/k)$ is the characteristic function of the identity element. 
Pulling this back to $\ker_\mathbb{A}$, we obtain the characteristic function of $im (\alpha)$.
Then by the fact that
$\tau(G_{\gamma}) |\ker(\alpha)|\cdot |\mathfrak{K}(G_{\gamma}/k)|^{-1}=\tau(G)=1$ (see \cite[(5.1.1)]{Kot84}), we have
\begin{equation}\label{eq:pre_stabilization}
    \mathrm{TF}_{ell.reg}(f)=\sum_{\gamma\in G(k)_{ell.reg}/\sim_{st}} ~~
    \sum_{\kappa\in \mathfrak{K}(G_{\gamma}/k)^*}\mathcal{O}^{\kappa}_{\gamma}(f)
\end{equation}
where $\mathcal{O}^{\kappa}_{\gamma}(f)$, the $\kappa$-orbital integral for $f$ at $\gamma$, is defined by
\begin{equation}\label{eq:kappa_orbital_integral}
\mathcal{O}_{\gamma}^\kappa(f)=\sum_{\substack{\gamma'\in G(\mathbb{A}_k)/\sim\\ \gamma'\sim_{st}\gamma}}\kappa(inv(\gamma',\gamma))\mathcal{O}_{\gamma'}(f).
\end{equation}
The formula (\ref{eq:pre_stabilization}) is called the pre-stabilization of the elliptic regular part of the trace formula. 
When $\kappa=1$, one can easily check that the $\kappa$-orbital integral $\mathcal{O}_{\gamma}^\kappa(f)$ is exactly an adelic stable orbital integrals.
\subsubsection{Endoscopic transfer}\label{sec:endoscopic_transfer}
The theory of endoscopy addresses the problem associated with the failure of stability, which arises from the discrepancy between the rational conjugacy and the stable conjugacy.
Through endoscopy, a pair $(\gamma,\kappa)$ in the summation index of (\ref{eq:pre_stabilization}) corresponds to an endoscopic triple.
For a quasi-split group $G$ over $k$, an endoscopic triple $(H,s,\eta)$ consists of 
\begin{enumerate}
        \item a quasi-split connected reductive group $H$;
         \item an embedding $\eta:\hat{H}\rightarrow \hat{G}$ of complex algebraic groups;
         \item an element $s\in [Z(\hat{H})/Z(\hat{G})]^\Gamma$ - we identify $Z(\hat{G})$ with a subgroup of $Z(\hat{H})$ via $\eta$,
     \end{enumerate}     
subject to the conditions in \cite[Definition 3.11]{Kal}, where $\hat{G}$ denotes the connected Langlands dual group of $G$.
We refer to \cite[Section 4]{Kal} for the definition of this notion for a general connected reductive group.

For an endoscopic triple $(H,s,\eta)$ of $G$, the following theorem allows us to transform $\kappa$-orbital integrals to stable orbital integrals of an endoscopic group $H$.
\begin{theorem}\label{thm:smooth_transfer_and_fundamental_lemma}
Suppose that $(H,s,\eta)$ is an endoscopic triple of $G$.
    \begin{enumerate}
        \item (Endoscopic transfer for functions) For every $f\in\mathcal{C}_c^{\infty}(G(k_v))$, there exists a matching function $f^H\in\mathcal{C}_c^{\infty}(H(k_v))$ such that
    \[
    \mathcal{SO}_{\gamma_H}(f^H)=\sum_{\substack{
    \gamma'\in G(k_v)/\sim\\
    \gamma'\sim_{st} \gamma}}\Delta_v(\gamma_H,\gamma')\mathcal{O}_{\gamma'}(f)
    \]
    for every strongly $G$-regular semisimple element $\gamma_H\in H(k_v)$ (see \cite[Remark 3.30]{Kal}) and $\gamma\in G(k_v)$ related to $\gamma_H$ (see \cite[Definition 3.17]{Kal}).
    Here, $\Delta_v(\gamma_H,\gamma')$ is the Langlands-Shelstad local transfer factor defined in \cite{LS}.
        \item (Fundamental lemma) If $G$ and $H$ are unramified, then there exists a constant $c\in \mathbb{C}^\times$ such that
        \[
        \mathcal{SO}_{\gamma_H}(\mathbbm{1}_{K_H})=c\sum_{\substack{\gamma'\in G(k_v)/\sim \\ \gamma' \sim_{st}\gamma}} \Delta_v(\gamma_H,\gamma')\mathcal{O}_{\gamma'}(\mathbbm{1}_{K_G}),
        \]
        for every strongly $G$-regular semisimple element $\gamma_H\in H(k_v)$ and $\gamma\in G(k_v)$ related to $\gamma_H$.
        Here, $K_G$ and $K_H$ are hyperspecial maximal compact subgroups of $G(k_v)$ and $H(k_v)$, respectively.
    \end{enumerate}
\end{theorem}
\begin{remark}\label{rmk:transfer_factor}
By \cite[Theorem 5.15]{Kal}, the Langlnads-Shelstad local factor $\Delta_v$ satisfies that 
\[
\prod_{v\in\Omega_k}\Delta_v(\gamma_H,\gamma_v')=\kappa(inv(\gamma',\gamma)),\]
where $\gamma'=(\gamma_v')_{v\in \Omega_k}\in G(\mathbb{A}_k)$ is stably conjugate to $\gamma\in G(k)_{ell.reg}$ and $\gamma_H\in H(k)$ is related to $\gamma$.
\end{remark}

We carry out the remaining steps in the stabilization by applying the correspondence given in \cite[Lemma 9.7]{Kot86}. 
In order to use this correspondence, one must switch from the sum over stable conjugacy classes in $G(k)$ in (\ref{eq:ell.reg.part}) to the sum over stable conjugacy classes in its quasi-split inner form $G_0(k)$. 
For clarity, we first assume that $G$ is quasi-split, and after that we briefly address the general case in Remark \ref{rmk:stabilization_innerform}.

By \cite[Lemma 9.7]{Kot86}, for each pair $(\gamma,\kappa)$ in the summation index of (\ref{eq:pre_stabilization}),
there exist an endoscopic triple $(H,s,\eta)$ of $G$ and an element $\gamma_H\in H(k)$ related to $\gamma$ that together correspond to $(\gamma,\kappa)$. 
Then Remark \ref{rmk:transfer_factor} yields that, for $\gamma\in G(k)_{ell.reg}/\sim_{st}$,
\[
\mathcal{O}_{\gamma}^\kappa(f)=\sum_{\substack{\gamma'\in G(\mathbb{A}_k)/\sim\\ \gamma'\sim_{st}\gamma}}\left(\prod_{v\in\Omega_k}\Delta_v(\gamma_{H},\gamma_v')\right)\mathcal{O}_{\gamma'}(f)=\prod_{v\in\Omega_k}\sum_{\substack{\gamma_v'\in G(k_v)/\sim \\ \gamma_v'\sim_{st}\gamma}}\Delta_v(\gamma_H,\gamma_v')\mathcal{O}_{\gamma_v'}(f_v).
\]
Theorem \ref{thm:smooth_transfer_and_fundamental_lemma} provides a function $f_v^H$ matching with $f_v$ for each $v\in\Omega_k$, and hence for $f^H=(f^H_v)_{v\in\Omega_k}$ we have
\begin{equation}\label{eq:summtion_after_use_transfer}
\sum_{\kappa\in \mathfrak{K}(G_{\gamma}/k)^*}\mathcal{O}_{\gamma}^\kappa(f)=\sum_{\substack{\kappa\in \mathfrak{K}(G_{\gamma}/k)^*\\(\gamma,\kappa)\mapsto((H,s,\eta),\ \gamma_H)\\}}\mathcal{SO}_{\gamma_{H}}(f^H),~~\text{ where } 
\mathcal{SO}_{\gamma_H}(f^H)=\prod_{v\in \Omega_k}\mathcal{SO}_{\gamma_H}(f^H_v).
\end{equation}

In order to complete the stabilization, we collect the terms in (\ref{eq:pre_stabilization}) associated with a single endoscopic triple $(H,s,\eta)$. 
This yields the final formula
\begin{equation}\label{eq:stabilization}
\mathrm{TF}_{ell.reg}(f)=\sum_{(H,s,\eta)/\cong}\iota(G,H)\cdot \mathrm{STF}_{G.ell.reg}(f^H),    
\end{equation}
where the sum is over isomorphism classes of elliptic endoscopic triples of $G$ (see \cite[Definition 3.12-3.13]{Kal}). 
The terms in (\ref{eq:stabilization}) are defined by
\[\iota(G,H)=\tau(G)\tau(H)^{-1}\lambda^{-1}~~~~\text{ and }~~~~\mathrm{STF}_{G.ell.reg}(f^H)=\tau(H)\sum_{\gamma_H}\mathcal{SO}_{\gamma_H}(f^H),\] where
$\lambda$ is defined in \cite[Section 9.4]{Kot86}, and 
the summation in $\mathrm{STF}_{G.ell.reg}(f^H)$ is taken over the set of stable conjugacy classes of strongly $G$-regular elliptic elements of $H(k)$.

\begin{remark}\label{rmk:root_system_of_endoscopic_group}
In this remark, we describe the endoscopic group $H$ corresponding to a given pair $(\gamma,\kappa)$ for a quasi-split group $G$.
By Remark \ref{rmk:identification_of_cohomology}, we identify  $\kappa$ as an element of $\pi_0(\hat{T}^{\Gamma})$, where $T:=G_{\gamma}$. 
Since $T$ is an anisotropic torus, $\hat{T}^{\Gamma}$ is finite and so we have $\pi_0(\hat{T}^\Gamma)=\hat{T}^\Gamma$.
Via the identification $\hat{T}\cong \mathrm{Hom}(X_*(T),\mathbb{G}_m)$, we define 
\[
R^{\vee}(T,H):=\{\alpha^\vee \in R^\vee(T,G)\mid \langle \kappa,\alpha^\vee\rangle=1\},
\]where $R^\vee(T,G)$ is the set of coroots of $T$ in $G$. 
We denote by $R(T,H)$ the dual of $R^{\vee}(T,H)$ and fix a basis $\Delta^H$ of $R(T,H)$.
Let $\omega_\sigma^H$ be the unique element of the Weyl group $\Omega(R(T,H))$ of $R(T,H)$ such that $\omega_\sigma^H\sigma_T$ preserves $\Delta^H$, where $\sigma_T$ denotes the $\Gamma$-action on $X^*(T)$ induced from that on $T$.
We define a root datum
\begin{equation}\label{eq:root_datum_of_H}
(X^*(T),R(T,H),X_*(T), R^\vee(T,H))\curvearrowleft\omega_\sigma^H \sigma_T ~~~~~ \text{   for }\sigma\in\Gamma,
\end{equation} 
equipped with a twisted $\Gamma$-action by $\{\omega_\sigma^H\}_{\sigma\in\Gamma}$.
Then there is a unique quasi-split connected reductive group $H$ over $k$, up to isomorphism, with a minimal Levi subgroup whose root datum in $H$ coincides with (\ref{eq:root_datum_of_H}).

By \cite[p. 101]{Kal}, $\{\omega_\sigma^{H}\}_{\sigma\in\Gamma}$ induces the canonical stable class of embeddings $i:T\rightarrow H$.
We then have $\gamma_H=i(\gamma)$ and $T^H=i(T)$, up to stable conjugacy in $H$.
Moreover, $i$ induces the isomorphism from $T^H$ to $T$, in \cite[Lemma 3.3.B]{KS}, which is called an admissible embedding.
Clearly, $\gamma_H$ is strongly $G$-regular, and hence its centralizer is $T^H$ by \cite[Lemma 3.3.C]{KS}.

\end{remark}

\begin{remark}\label{rmk:fundamental_lemma_for_quasi_split}
    With the notations in Remark \ref{rmk:root_system_of_endoscopic_group}, the equality in Theorem \ref{thm:smooth_transfer_and_fundamental_lemma}.(2) is given more explicitly in \cite[Section 6.1]{Tom} as follows:
    \begin{align*}
    \Bigg(\prod_{\alpha\in R(T^H,H)}&|\alpha(\gamma_H)-1|_v^{1/2} \Bigg)\frac{vol(m^{T^H}_v,K_{T^H})}{vol(m^{H}_v,K_H)}\mathcal{SO}_{\gamma_H}(\mathbbm{1}_{K_H})\\
    &=\Bigg(\prod_{\alpha\in R(T,G)}|\alpha(\gamma)-1|_v^{1/2}\Bigg)\frac{vol(m^{T}_v,K_{T})}{vol(m^{G}_{v},K_G)}\sum_{\substack{\gamma_v'\in G(k_v)/\sim \\ \gamma_v' \sim_{st}\gamma}}
    \kappa_v(inv_v(\gamma_v', \gamma))
    \mathcal{O}_{\gamma_v'}(\mathbbm{1}_{K_G}),
    \end{align*}
    where 
     $K_T$ and $K_{T^H}$ are the maximal compact subgroups of $T(k_v)$ and $T^H(k_v)$, respectively, and $\kappa_v$ is the image of $\kappa$ under $\hat{T}^\Gamma\rightarrow \hat{T}^{\Gamma_v}$.
\end{remark}

\begin{remark}\label{rmk:stabilization_innerform}
    For a general connected reductive group $G$, we fix an inner twist $\psi:G_0\rightarrow G$ as in \cite[Section 6]{Kot82} where $G_0$ is the quasi-split inner form of $G$.
    For $\gamma_0\in G_0(k)$ and $\gamma'\in G(\mathbb{A}_k)$ conjugate to $\gamma_0$ under $G_0(\mathbb{A}_{\bar{k}}) (=G(\mathbb{A}_{\bar{k}}))$,
    the invariant $\mathrm{obs}(\gamma')\in \mathfrak{K}(G_{0,\gamma_0}/k)$, defined in \cite[Section 5.4]{Kal}, is used to determine whether the $G(\mathbb{A}_k)$-conjugacy class of $\gamma'$ contains a $k$-point.
    This is the analogue of $inv(\gamma',\gamma)$ in (\ref{eq:kappa_orbital_integral}) for $\gamma \in G(k)$. 
    Via this invariant, we have
\begin{equation}\label{eq:pre_stabilization_innerform}
    \mathrm{TF}_{ell.reg}(f)=\sum_{\gamma_0\in G_0(k)_{ell.reg}/\sim_{st}} ~~
    \sum_{\kappa\in \mathfrak{K}(G_{0,\gamma_0}/k)^*}\mathcal{O}^{\kappa}_{\gamma_0}(f)
~~~~~\text{ where }\mathcal{O}_{\gamma_0}^\kappa(f)=\sum_{\substack{\gamma'\in G(\mathbb{A}_k)/\sim\\ \gamma'\sim_{st}\psi(\gamma_0)}}\kappa(\mathrm{obs}(\gamma'))\mathcal{O}_{\gamma'}(f).
\end{equation}
As in the quasi-split case, for each pair $(\gamma_0,\kappa)$, we obtain an endoscopic triple $(H,s,\eta)$ of $G$ and an element $\gamma_H\in H(k)$ that together correspond to $(\gamma_0,\kappa)$ by \cite[Lemma 9.7]{Kot86}.
By applying Theorem \ref{thm:smooth_transfer_and_fundamental_lemma}, we obtain the final formula, which agrees with (\ref{eq:stabilization}).
\end{remark}
\section{Main Result}\label{sec:main_result}
Suppose that $G$ is semisimple and simply connected, and that $X\cong G_\gamma \backslash G$ where $\gamma$ is an elliptic regular element of $G(k)$. 
By Remark \ref{rmk:conjugacy_and_pointsofX}, $X(\mathbb{A}_k)$ is identified with the stable conjugacy class of $\gamma$ in $G(\mathbb{A}_k)$. 
For a model $\mathbf{X}$ of $X$ over $\mathcal{O}_k$ and $T>0$, we recall the function $f_{\mathbf{X},T}$ on $X(\mathbb{A}_k)$, defined by
\[
f_{\mathbf{X},T}= \mathbbm{1}_{X(k_\infty,T)\times \prod_{v<\infty_k}\mathbf{X}(\mathcal{O}_k)},
\]where $X(k_\infty,T)= \{x\in X(k_\infty)\mid \norm{x}_\infty\leq T\}$.
Let $\tilde{f}_{\mathbf{X}, T}$ be a function on $G(\mathbb{A}_k)$ such that
\begin{equation}\label{eq:tilde_f}
\tilde{f}_{\mathbf{X}, T}(x)=f_{\mathbf{X}, T}(x) \text{ for $x\in X(\mathbb{A}_{k})$}.
\end{equation}
Following the notations of Section \ref{sec:arthur's_trace_formula}, $\sim_F$ and  $\sim_{st,F}$  denote the rational and stable conjugate relations over $F$, respectively, for $F=k,k_v$, and $\mathbb{A}_k$. 
If the base field is clear from the context, then we sometimes abbreviate the notations to $\sim$ and $\sim_{st}$, respectively.

Our main theorem establishes an asymptotic formula for $ N(X;f_{\mathbf{X},T})$ in terms of $\kappa$-orbital integrals, arising in the pre-stabilization (\ref{eq:pre_stabilization}) of the elliptic regular part.
This suggests a link between the asymptotic behavior of $ N(X;f_{\mathbf{X},T})$ and a part of the geometric side of the trace formula associated with the stable class of $\gamma$.

\begin{theorem}\label{thm:main}

   If $G'(k_\infty)$ is not compact for any non-trivial simple factor $G'$ of $G$ and $X$ satisfies the euqidistribution property (\ref{eq:equidistribution_property}), 
    then we have 
    \begin{equation}\label{eq:mainthm}
     N(X;f_{\mathbf{X},T})\sim \sum_{\kappa\in H^1(k,G_\gamma(\mathbb{A}_{\bar{k}})/G_\gamma(\bar{k}))^*}\mathcal{O}_{\gamma}^\kappa(\tilde{f}_{\mathbf{X}, T}),\end{equation}
     where 
    $\mathcal{O}_{\gamma}^\kappa(\tilde{f}_{\mathbf{X}, T})=
    \sum\limits_{\substack{\gamma'\in G(\mathbb{A}_k)/\sim_{\mathbb{A}}\\ \gamma'\sim_{st,\mathbb{A}}\gamma}}\kappa(inv(\gamma',\gamma))\mathcal{O}_{\gamma'}(\tilde{f}_{\mathbf{X}, T})$. Here, $\mathcal{O}_{\gamma'}(\tilde{f}_{\mathbf{X},T})$ is the orbital integral for $\tilde{f}_{\mathbf{X},T}$ at $\gamma'$ with respect to the Tamagawa measure on $X(\mathbb{A}_k)$ (see (\ref{eq:orbital_integral})).
\end{theorem}
\begin{proof}
    By Theorem \ref{thm:wx}, we have
    \begin{equation}\label{eq:Br_integral}
     N(X;f_{\mathbf{X},T})\sim |\Br X/\Br k|\int_{X(\mathbb{A}_k)^{\mathrm{Br}}} f_{\mathbf{X},T}\ m^X.
    \end{equation}
    According to \cite[Theorem 3.2]{CtX}, the set $X(\mathbb{A}_k)^{\mathrm{Br}}$ is identified with the kernel of the composite map in the diagram (\ref{eq:diagram_for_H1}):
    \[
    X(\mathbb{A}_k)\rightarrow \bigoplus_{v\in\Omega_k}H^1(k_v,G_\gamma)\rightarrow \mathrm{Hom}(\Pic G_\gamma,\mathbb{Q}/\mathbb{Z}).
    \] 
    From the exact sequence of pointed sets in the diagram (\ref{eq:diagram_for_H1})
    \[H^1(k,G_\gamma)\xrightarrow{\alpha^1} \bigoplus_{v\in\Omega_k}H^1(k_v,G_\gamma)\rightarrow \mathrm{Hom}(\Pic G_\gamma,\mathbb{Q}/\mathbb{Z}),\]
    we have 
    \begin{equation}\label{eq:Br_fixed_description}
    X(\mathbb{A}_k)^{\mathrm{Br}}=\{\gamma'\in X(\mathbb{A}_k)\mid inv(\gamma',\gamma)\in im(\alpha^1)\}.
    \end{equation}    
    On the other hand, consider the following diagram of pointed sets:
    \[
    \begin{tikzcd}
	1 & {\ker_k} & {H^1(k,G_\gamma)} & {H^1(k,G)} \\
	1 & {\ker_\mathbb{A}} & {\bigoplus_{v\in\Omega_k}H^1(k_v,G_\gamma)} & {\bigoplus_{v\in\Omega_k}H^1(k_v,G)}.
	\arrow[from=1-1, to=1-2]
	\arrow["\subset",from=1-2, to=1-3]
	\arrow["\alpha", from=1-2, to=2-2]
	\arrow[from=1-3, to=1-4]
	\arrow["\alpha^1", from=1-3, to=2-3]
	\arrow[from=1-4, to=2-4]
	\arrow[from=2-1, to=2-2]
	\arrow["\subset" ,from=2-2, to=2-3]
	\arrow[from=2-3, to=2-4]
    \end{tikzcd}
    \]
By the Hasse principle for simply connected groups (see \cite[Theorem 6.6]{PR94}),  the map $H^1(k,G)\rightarrow \bigoplus_{v\in\Omega_k}H^1(k_v,G)$ is injective.
Since the image of $inv(-,\gamma):X(\mathbb{A}_k)\rightarrow \bigoplus_{v\in\Omega_k}H^1(k_v,G_\gamma)$ is exactly $\ker _\mathbb{A}$, the equation (\ref{eq:Br_fixed_description}) turns to be
\begin{equation}\label{eq:Br=H1_description}
    X(\mathbb{A}_k)^{\mathrm{Br}}=\{\gamma'\in X(\mathbb{A}_k)\mid inv(\gamma',\gamma)\in im(\alpha)\}=\bigsqcup_{\substack{\gamma'\in X(\mathbb{A}_k)/G(\mathbb{A}_k)\\ inv(\gamma',\gamma)\in im(\alpha)}}\gamma'\cdot G(\mathbb{A}_k).    
\end{equation}
Therefore, the right-hand side of the equation (\ref{eq:Br_integral}) is reformulated as follows
    \begin{equation}\label{eq:pre_last_formula}
     N(X;f_{\mathbf{X},T})\sim |\Br X/\Br k|\sum_{a\in im(\alpha)}\int_{\gamma_a \cdot G(\mathbb{A}_k)}f_{\mathbf{X},T}\ m^X,
    \end{equation}
where $\gamma_a\in X(\mathbb{A}_k)$ represents the $G(\mathbb{A}_k)$-orbit in $X(\mathbb{A}_k)$ such that $inv(\gamma_a,\gamma)=a$ for each $a\in im(\alpha)$.

Recall that the function 
\[|\mathfrak{K}(G_\gamma/k)|^{-1}\sum\limits_{\kappa\in \mathfrak{K}(G_\gamma/k)^*}\kappa\]
on $\mathfrak{K}(G_{\gamma}/k):=H^1(k,G_\gamma(\mathbb{A}_{\bar{k}})/G_\gamma(\bar{k}))$, introduced in Section \ref{sec:pre_stabilization},
is the characteristic function of the identity element.
Recall also that its pullback to
$\ker_\mathbb{A}$ along the map $\ker_\mathbb{A}\rightarrow\mathfrak{K}(G_\gamma/k)$ in (\ref{eq:mainthm_ses}) is the characteristic function of $im (\alpha)$.
It then follows that
    \[
     N(X;f_{\mathbf{X},T})\sim |\Br X/ \Br k| \cdot |\mathfrak{K}(G_\gamma/k)|^{-1}\sum\limits_{\kappa\in \mathfrak{K}(G_\gamma/k)^*}\sum_{\gamma'\in X(\mathbb{A}_k)/G(\mathbb{A}_k)}\kappa(inv(\gamma',\gamma)) \int_{\gamma'\cdot G(\mathbb{A}_k)}f_{\mathbf{X},T} \ m^X.
    \]
By Remark \ref{rmk:identification_of_cohomology}, the factor $|\Br X/ \Br k| \cdot |\mathfrak{K}(G_\gamma/k)|^{-1}$ equals 1. 
Furthermore, by Remark \ref{rmk:conjugacy_and_pointsofX} $\gamma'\cdot G(\mathbb{A}_k)$ is identified with the $G(\mathbb{A}_k)$-conjugacy class of $\gamma$ in $G(\mathbb{A}_k)$. Therefore, for a function $\tilde{f}_{\mathbf{X},T}$ on $G(\mathbb{A}_k)$ satisfying (\ref{eq:tilde_f}), we have
    \[
    \sum_{\gamma'\in X(\mathbb{A}_k)/G(\mathbb{A}_k)}\kappa(inv(\gamma',\gamma)) \int_{\gamma'\cdot G(\mathbb{A}_k)}f_{\mathbf{X},T} \ m^X=\mathcal{O}_{\gamma}^\kappa(\tilde{f}_{\mathbf{X},T}).
    \]
    This completes the proof of the theorem.
\end{proof}
\begin{remark}\label{rmk:main}
Let $\mathbf{G}$ be an affine model of $G$ over $\mathcal{O}_k$ and let $\mathbf{X}$ be the schematic closure of $X$ in $\mathbf{G}$.
Since $\mathcal{O}_{k_v}$ is a flat $\mathcal{O}_{k}$-module, 
$\mathbf{X}_{\mathcal{O}_{k_v}}$ is the schematic closure of $X_{k_v}$ in $\mathbf{G}_{\mathcal{O}_{k_v}}$ and one has
$\mathbf{X}(\mathcal{O}_{k_v})=X(k_v)\cap \mathbf{G}(\mathcal{O}_{k_v})$ by \cite[Lemma 2.4.3]{GH19} where the intersection is taken in $G(k_v)$ and $X(k_v)$ is identified with the stable conjugacy class of $\gamma$ in $G(k_v)$.
We define a function $f_{\mathbf{G},T}$ on $G(\mathbb{A}_k)$ as follows
\begin{equation}\label{eq:choice_of_G_integral}
    f_{\mathbf{G},T}=\prod_{v\in \Omega_k}f_{\mathbf{G},T,v} \text{ where }f_{\mathbf{G},T,v} =\left\{
    \begin{array}{l l}
        \mathbbm{1}_{\mathbf{G}(\mathcal{O}_{k_v})} & v<\infty_k; \\
        \sum_{\gamma_v'\sim_{st,k_v}\gamma}f_{\mathbf{G},T,\gamma_v'} & v\in \infty_k.
    \end{array}
    \right.
    \end{equation}
    Here, for $v \in \infty_k$ and $\gamma_v'\sim_{st,k_v}\gamma$, the function $f_{\mathbf{G},T,\gamma_v'}$ is taken to be a smooth function supported on an open neighborhood of $\gamma_v'\cdot G(k_v)\cap X(k_v,T)$ in $G(k_v)$, that does not intersect any $G(k_v)$-conjugacy class in $X(k_v)$ other than $\gamma_v'\cdot G(k_v)$ (for $G(k_v)$-conjugacy classes in $X(k_v)$, such neighborhoods exist since $G(k_v)$ is a Lie group, hence Hausdorff, and conjugacy classes of semisimple elements are closed). 
    It is normalized so that
    \[
    \int_{\gamma_v'\cdot G(k_v)}f_{\mathbf{G},T,\gamma_v'}\ m_v^X=vol(m_v^X, \gamma_v' \cdot G(k_v)\cap X(k_v,T)).
    \]
    For almost all $v<\infty_k$, by \cite[Proposition 2.4.5]{GH19}, $f_{\mathbf{G},T,v}$ coincides with the characteristic function of a hyperspecial subgroup of $G(k_v)$.
    Moreover, it follows that $f_{\mathbf{G},T}\in \mathcal{C}_c^\infty(G(\mathbb{A}_k))$.

The asymptotic equivalence (\ref{eq:mainthm}) remains valid when we replace   $\tilde{f}_{\mathbf{X},T}$ with $f_{\mathbf{G},T}$.
In this case, the right-hand side of (\ref{eq:mainthm}) coincides with that of the equation (\ref{eq:pre_stabilization}) putting $f=f_{\mathbf{G},T}$.
By applying the reverse of derivation for pre-stabilization described in \ref{sec:pre_stabilization}, we obtain that
\begin{equation}\label{eq:compare_with_trace_formula}
  N(X;f_{\mathbf{X},T})\sim\sum_{\substack{\gamma'\in G(k)/\sim\\ \gamma'\sim_{st}\gamma}} \tau(G_{\gamma'})\mathcal{O}_{\gamma'}(f_{\mathbf{G},T})=\sum\limits_{\substack{\mfo'\in\mathcal{O}\\ \mfo'\sim_{st}\mfo}}J_{\mfo'}(f_{\mathbf{G},T}),\end{equation}
where the right-hand side is the part of the geometric side of the trace formula for $f_{\mathbf{G},T}$.

This observation suggests that the asymptotic behavior of $ N(X;f_{\mathbf{X},T})$ should reflect contributions from the spectral side of the trace formula.
To make this expectation precise, we should consider the contribution of $f_{\mathbf{G},T}$ to the geometric terms associated with semisimple conjugacy classes in $G(k)$ which is not stably conjugate to $\gamma$.
Moreover, it is necessary to analyze the multiplicity and the character of any automorphic representation $\pi$ for which $\mathrm{tr}(\pi(f_{\mathbf{G},T}))$ is nonzero.
However, this would be challenging.
Indeed, when the effect of a test function on the geometric side is straightforward to determine
— for instance, as in our case, where we analyze the contribution from a single stable conjugacy class — 
it often becomes complicated on the spectral side since it may influence a vast range of representations.
Such a phenomenon already appears in the simplest example of the trace formula, namely the Poisson summation formula.

On the other hand, using the $L^2$-expansion for a function on $[G]$ associated with $ N(X;f_{\mathbf{X},T})$, one may obtain another spectral interpretation of the asymptotic behavior of $ N(X;f_{\mathbf{X},T})$. 
We discuss this observation in Appendix \ref{app:spectral_expansion}.

\end{remark}

\begin{corollary}\label{cor:main_cor}
Under the same assumptions with Theorem \ref{thm:main},
let $\mathbf{G}$ be an affine model of $G$ and $\mathbf{X}$ be the schematic closure of $X$ in $\mathbf{G}$.
We have
\[
     N(X;f_{\mathbf{X},T})\sim \sum_{\substack{\kappa\in H^1(k,G_{0,\gamma_0}(\mathbb{A}_{\bar{k}})/G_{0,\gamma_0}(\bar{k}))^*\\(\gamma_0, \kappa)\mapsto((H,s,\eta),\ \gamma_{0,H})\\}}
    \mathcal{SO}_{\gamma_{0,H}}(f^H),
\]
where $G_0$ is the quasi-split inner form of $G$ and $G_{0,\gamma_0}$ is the centralizer of $\gamma_0$ in $G_0$.
Here, the explanation of these terms is as follows:
\begin{itemize}
    \item $\gamma_0\in G_0(k)$ lies in the stable conjugacy class in $G_0(k)$ corresponding to that of $\gamma \in G(k)$ under the map in \cite[Section 6]{Kot82}.
    \item $(\gamma_0, \kappa)\mapsto ((H,s,\eta),\gamma_{0,H})$ denotes the correspondence in \cite[Lemma 9.7]{Kot86}. 
    \item $\mathcal{SO}_{\gamma_{0,H}}(f^H)$ is the stable orbital integral at $\gamma_{0,H}$ for a function $f^H\in\mathcal{C}_c^\infty(H(\mathbb{A}_k))$ (see Definition \ref{def:stable_conjugacy}). 
Here, $f^H$ matches the function $f_{\mathbf{G},T}$ defined in (\ref{eq:choice_of_G_integral}),
in the sense of Theorem \ref{thm:smooth_transfer_and_fundamental_lemma}.
\end{itemize}
\end{corollary}
\begin{proof}
    By Remark \ref{rmk:stabilization_innerform}, for a fixed inner twisting $\psi: G_0 \rightarrow G$ in \cite[Section 6]{Kot82}, as in Remark \ref{rmk:stabilization_innerform}, the asymptotic equivalence (\ref{eq:mainthm}) turns to be 
    \begin{equation}\label{eq: N(X;f_{mathbf{X},T})_and_traceformula_inner}
     N(X;f_{\mathbf{X},T})\sim 
    \sum_{\kappa\in \mathfrak{K}(G_{0,\gamma_0}/k)^*}\mathcal{O}^{\kappa}_{\gamma_0}(f_{\mathbf{G},T})
~~~~~\text{ where }\mathcal{O}_{\gamma_0}^\kappa(f_{\mathbf{G},T})=\sum_{\substack{\gamma'\in G(\mathbb{A}_k)/\sim\\ \gamma'\sim_{st}\psi(\gamma_0)}}\kappa(\mathrm{obs}(\gamma'))\mathcal{O}_{\gamma'}(f_{\mathbf{G},T}),
    \end{equation}
    where $\mathfrak{K}(G_{0,\gamma_0}/k):=H^1(k,G_{0,\gamma_0}(\mathbb{A}_{\bar{k}})/G_{0,\gamma_0}(\bar{k}))$.
    Applying \cite[Lemma 9.7]{Kot86} and Theorem \ref{thm:smooth_transfer_and_fundamental_lemma} to the right-hand side of (\ref{eq: N(X;f_{mathbf{X},T})_and_traceformula_inner}), we have the desired result.
\end{proof}

\section{Application: $\mathrm{SL}_n$-homogeneous space}\label{sec:application}
Let $X\cong S_\gamma \backslash \mathrm{SL}_n$, where $S_\gamma$ is the centralizer of an elliptic regular element $\gamma \in \mathrm{SL}_n(k)$. 
We denote by $\chi(x)$ the characteristic polynomial of $\gamma$.
Since $\gamma$ is elliptic regular in $\mathrm{SL}_n(\mathcal{O}_k)$, the polynomial $\chi(x)$ is irreducible in $\mathcal{O}_k[x]$ and $\chi(0)=1$.
We define \[N(X,T):= N(X;f_{\mathbf{X},T})=|\{x\in\mathbf{X}(\mathcal{O}_k)\mid \norm{x}_\infty \leq T\}|,
\] where $f_{\mathbf{X},T}=\mathbbm{1}_{X(k_\infty,T)\times \prod_{v<\infty_k}\mathbf{X}(\mathcal{O}_k)}$ for the schematic closure $\mathbf{X}$ of $X$ in $\mathrm{SL}_{n,\mathcal{O}_k}$, and $\norm{\cdot}_\infty$ is the modified norm defined in (\ref{intro:eq:modified_norm}).

The main goal of this section is to formulate the asymptotic of $N(X,T)$ in terms of orbital integrals of $\mathrm{GL}_n$ under the following assumptions:
\begin{equation}\label{intro:condition_ours}
\left\{
\begin{array}{l}
     \textit{(1) $k$ and $K\left(:=k[x]/(\chi(x))\right)$ are totally real number fields};  \\
     \textit{(2) if $k\neq \mathbb{Q}$, then  $\chi(x)$ is of prime degree $n$};\\
    \textit{(3) if $k=\mathbb{Q}$, then there is \textbf{no restriction} on $\chi(x)$}.
\end{array}
\right.
\end{equation}
\begin{remark}\label{rmk:equidistribution_for_SLn}
    As recalled in Remark \ref{rmk:equidistribution_listup}, to prove that $X$ satisfies the equidistribution property, it suffices to verify that the sets $x\cdot G(k_\infty)\cap X(k_\infty,T)$ are not focused in the sense of \cite[Definition 1.14]{EMS}.
    
    Over $\mathbb{Q}$, this was verified by Eskin, Mozes, and Shah  \cite[Theorem 1.16]{EMS}.
    They worked with a slightly different group rather than $\mathrm{SL}_n$: They regarded $X$ as a homogeneous space under \[G:=\{g\in\mathrm{GL}_n\mid \det g=\pm 1 \},\text{ (see {\cite[p. 255]{EMS})}},\] and proved on p. 282 of loc. cit. that the sets $ x\cdot G(\mathbb{R})\cap X(\mathbb{R},T_n)$ are not focused.

    For a number field $k$, we apply the same to the restriction of scalars
    \[X':=\mathrm{Res}_{k/\mathbb{Q}}X\text{ and } G':=\mathrm{Res}_{k/\mathbb{Q}}G.\]
    For the stabilizer $H_x$ of $x$ in $\mathbf{X}(\mathcal{O}_k)$, we have $X'\cong H'\backslash G'$ where $H':= \mathrm{Res}_{k/\Q}H_x$ (see Remark \ref{rmk:equidistribution_listup}).
    Since $H_x$ is an anisotropic maximal torus over $k$, $H_x'$ is anisotropic over $\Q$.
    Moreover, since $K$ and $k$ are totally real, we have $G(k_v)=G(\R)$ and $X(k_v)=X(\mathbb{R})$ for all $v\in \infty_k$ and thus the nonfocusing statement follows from the arguments in \cite[p. 282]{EMS}.
    Consequently, \cite[Theorem 1.16]{EMS} applies to $X'$ and therefore $X$ over $k$ also satisfies the equidistribution property \cite[(4.2)]{WX}.
\end{remark}

In this section, we use the following notations:
\begin{itemize}
\item
Let $K= k[x]/(\chi(x))$ with the ring of integers $\mathcal{O}_K$.
For each $v \in \Omega_k$, define $K_v = k_v[x]/(\chi(x))$ and denote its ring of integers by $\mathcal{O}_{K_v}$.

\item A field extension $K/k$ is said to be unramified if no prime ideal of $\mathcal{O}_k$ ramifies in $\mathcal{O}_K$, and ramified otherwise.

\item
For each $v \in \Omega_k$, let $B_v(\chi)$ be an index set in bijection with the irreducible factors $\chi_{v,i}$ of $\chi$ over $k_v$.
For $i \in B_v(\chi)$, we define the following notations
\[
\left\{
\begin{array}{l}
    \textit{$K_{v,i} = k_v[x]/(\chi_{v,i}(x))$ with the ring of integers $\mathcal{O}_{K_{v,i}}$};\\
    \kappa_{K_{v,i}}:\textit{ the residue field of $K_{v,i}$}.
\end{array}
\right.
\]
Then we have $K_v \cong \prod_{i \in B_v(\chi)} K_{v,i}$ and $\mathcal{O}_{K_v} \cong \prod_{i \in B_v(\chi)}\mathcal{O}_{K_{v,i}}$.
\item For a polynomial $f(x)\in F[x]$ and a field $F$, we denote by $\Delta_f \in F$ a discriminant of a polynomial. 

\item The Tamagawa measure on $X(\mathbb{A}_k)$ is given by
    \[
    m^X:=|\Delta_k|^{-\frac{1}{2}\dim X}\prod_{v\in\Omega_k}|\omega_X|_v,
    \]for an $\mathrm{SL}_n$-invariant gauge form $\omega_X$ on $X$ such that $\omega_{\mathrm{SL}_n}=\omega_X\cdot \omega_{S_\gamma}$, where $\omega_{\mathrm{SL}_n}$ and $\omega_{S_\gamma}$ are invariant gauge forms on $\mathrm{SL}_n$ and $S_\gamma$, respectively.

\item For $F= k_v$ or $\mathbb{A}_k$, let $\mathcal{O}_\delta(f)$ (resp. $\mathcal{SO}_\delta(f)$) denote the orbital integral (resp. the stable orbital integral) of $\mathrm{SL}_n$ for 
$f\in \mathcal{C}_c^\infty(\mathrm{SL}_n(F))$ at $\delta \in \mathrm{SL}_n(F)$.
\item We denote by $\mathfrak{K}(S_\gamma/k)$ the finite abelian group $H^1(k,S_\gamma(\mathbb{A}_{\bar{k}})/S_\gamma(\bar{k}))$.
\end{itemize}
\begin{remark}
    In general, the volume of $\prod_{v<\infty_k} \mathbf{X}(\mathcal{O}_{k_v})$ with respect to $\prod_{v <\infty_k} \abs{\omega_X}_v$ does not converge for a model $\mathbf{X}$ of $X$.
    To resolve this problem, Ono \cite{Ono65} introduced the convergence factors using the Artin $L$-function, presented in Definition \ref{def:tamagawameasure}. 
    However, in the case that $X\cong S_\gamma \backslash \mathrm{SL}_n$, this already converges. 
    Thus we use the notion of the Tamagawa measure without convergence factors.
\end{remark}

\subsection{Measures}\label{meas:gln}
Since $X\cong S_\gamma\backslash \mathrm{SL}_n\cong T_\gamma \backslash \mathrm{GL}_n$ where $T_\gamma$ is the centralizer of $\gamma$ in $\mathrm{GL}_n$, we also have a $\mathrm{GL}_n$-homogeneous space structure on $X$.
Here, $T_\gamma$ is isomorphic to the Weil restriction $\mathrm{R}_{K/k}\mathbb{G}_{m,K}$ of $\mathbb{G}_{m.K}$.
By Hilbert's Theorem 90, we have $H^1(k,T_{\gamma})= H^1(k_v,T_{\gamma,k_v})=1$. 
Then it follows that $H^1(k,T_{\gamma}(\mathbb{A}_{\bar{k}}))\cong \bigoplus_{v\in\Omega_k}H^1(k_v,T_{\gamma,k_v})$ is trivial.
By the arguments in Section \ref{subsec:orbit}, for $F=k, k_v$ or $\mathbb{A}_k$, we then have \[X(F) = \gamma \cdot \mathrm{GL}_n(F)\cong T_\gamma(F)\backslash \mathrm{GL}_n(F).\]
As in Remark \ref{rmk:conjugacy_and_pointsofX}, $X(F)$ is identified with the stable conjugacy class of $\gamma$ in $\mathrm{GL}_n(F)$. By the preceding discussion, this stable conjugacy class in fact coincides with the rational conjugacy class.

Let $\omega_X^{\mathrm{GL}_n}$ be a $\mathrm{GL}_n$-invariant gauge form such that $\omega_{\mathrm{GL}_n}=\omega_X^{\mathrm{GL}_n}\cdot \omega_{T_\gamma}$ where $\omega_{\mathrm{GL}_n}$ and $\omega_{T_\gamma}$ are invariant gauge forms on $\mathrm{GL}_n$ and $T_\gamma$, respectively.
Then the following lemma yields that
\begin{equation}\label{eq:tamagawa_measure_in_SLn}
m^X=|\Delta_k|^{-\frac{1}{2}\dim X}\prod_{v\in\Omega_k}|\omega_X^{\mathrm{GL}_n}|_v.
\end{equation}
\begin{lemma}\label{lem:uniquetamagawa}
    The Tamagawa measure on $X$ is independent of the choice of a gauge form $\omega_X$.
\end{lemma}
\begin{proof}
    It suffices to show that any gauge form on $X$ is unique up to a constant in $k^\times$.
    A gauge form $\omega \in \Omega_{X/k}^{\dim X}(X)$ defines an isomorphism $\mathcal{O}_X \cong \Omega_{X/k}^{\dim X}(X), 1 \mapsto \omega$ of locally free $\mathcal{O}_X$-modules.
    For another gauge form $\omega'$ on $X$, $\omega \circ \omega'^{-1}$ defines a nowhere-zero regular automorphism on $\mathcal{O}_X$ which corresponds to an element in the unit group of the coordinate ring $k[X]^\times$.
    Thus, we have $\omega' = c \omega$ for some $c \in k[X]^\times$. Since $X$ is a homogeneous space of $\mathrm{SL}_n$ and $\Hom(\mathrm{SL}_n, \mathbb{G}_m) = 1$, we have $k[X]^\times=k^\times$ by \cite[Lemma 1.5.1]{BR}.
\end{proof}

On the other hand, we define the following measure on $X(k_v)$ for each $v\in \Omega_k$,
\begin{equation}\label{eq:quotient_measure}
d\mu_v:=\left\{\begin{array}{l l}
     |\omega_{X_{k_v}}^{can}|_v& v\in\infty_k\\
     dg_v/dt_v&  v<\infty_k.
\end{array}
\right.
\end{equation}
\begin{itemize}
    \item For $v\in\infty_k$, $\omega_{X_{k_v}}^{can}$ is a $\mathrm{GL}_{n,k_v}$-invariant form  on $X_{k_v}$ such that $\omega_{\mathrm{GL}_{n,k_v}}^{can}=\omega_{X_{k_v}}^{can}\cdot\omega_{T_{\gamma,k_v}}^{can}$ where $\omega_{\mathrm{GL}_{n,k_v}}^{can}$ (resp. $\omega_{T_{\gamma,k_v}}^{can}$) is the invariant forms on $\mathrm{GL}_{n,k_v}$ (resp. $T_{\gamma,k_v}$) defined in \cite[Section 9]{GG99}.
\item For $v<\infty_k$, $dg_v/dt_v$ is the quotient measure on $X(k_v)\cong T_\gamma(k_v)\backslash \mathrm{GL}_n(k_v)$ where $dg_v$ (resp. $dt_v$) is the measure on $\mathrm{GL}_n(k_v)$ (resp. $T_\gamma(k_v)$) such that $vol(dg_v, \mathrm{GL}_n(\mathcal{O}_{k_v}))=1$ (resp. $vol(dt_v,T_c)=1$ for the maximal compact open subgroup $T_c$ of $T_\gamma(k_v)$).
\end{itemize}
We denote by $\mathcal{O}_{\gamma,d\mu_v}^{\mathrm{GL}_n}(f_v)$ the (stable) orbital integral of $\mathrm{GL}_n$ for $f_v \in \mathcal{C}_c^\infty(\mathrm{GL}_n(k_v))$ at $\gamma$ with respect to $d\mu_v$.
Yun \cite{Yun13} showed that the orbital integral $\mathcal{O}_{\gamma,d\mu_v}^{\mathrm{GL}_n}(\mathbbm{1}_{\mathrm{GL}_n(\mathcal{O}_{k_v})})$ is a $q_v$-polynomial with $\mathbb{Z}$-coefficients (so that it is an integer), whose leading term is expected to be $q_v^{S_v(\gamma)}$. 
Here, $S_v(\gamma)$ denotes the $\mathcal{O}_{k_v}$-module length between $\mathcal{O}_{K_v}$ and $\mathcal{O}_{k_v}[x]/(\chi(x))$, referred to as the Serre invariant.
To state the asymptotic formula for $N(X,T)$ in Theorem \ref{thm:application}, we will use the measure $d\mu_v$ for the local orbital integrals.
Accordingly, it is necessary to compare the measure $\prod_{v\in\Omega_k}|\omega_X^{\mathrm{GL}_n}|_v$ in (\ref{eq:tamagawa_measure_in_SLn}) with $\prod_{v\in\Omega_k}d\mu_v$.
\begin{lemma}\label{lem:C_meas_value}
        We have
    \[
C_{meas}:=\prod_{v\in \Omega_k}\frac{|\omega_X^{\mathrm{GL}_n}|_v}{d\mu_v}= 2^{(n-1)[k:\Q]} \frac{R_K h_K \sqrt{\abs{\Delta_K}}^{-1}}{R_k h_k \sqrt{\abs{\Delta_k}}^{-1}} \left(\prod_{i=2}^n \zeta_k(i)^{-1}\right)\left(\prod_{v<\infty_k}\frac{|\Delta_\chi|_v^{-\frac{1}{2}}}{q_v^{S_v(\gamma)}}\right),
    \]where $R_F$ is the regulator of $F$, $h_F$ is the class number of $\mathcal{O}_F$, and $\zeta_k$ is the Dedekind zeta function of $k$.
\end{lemma}\noindent
For the proof of Lemma \ref{lem:C_meas_value}, see Appendix \ref{app:meas_diff}.
\subsection{Computation}
Firstly, we investigate the structure of $\mathfrak{K}(S_\gamma/k)$ under the assumption that $\deg (K/k)$ is prime. 
This is necessary to describe endoscopic groups of $G$.
\begin{lemma}
We have
\[
\mathfrak{K}(S_\gamma/k)^*\cong \left\{\begin{array}{l l}
     \mathbb{Z}/n\mathbb{Z}&\text{if $K/k$ is a Galois extension};  \\
     1 & \text{otherwise}.
\end{array}\right.
\]
\end{lemma}
\begin{proof}
By Remark \ref{rmk:identification_of_cohomology} and the proof of \cite[Theorem 6.1]{WX}, we have
    \begin{equation}\label{eq:Brgroup_when_SLn}
    \mathfrak{K}(S_\gamma/k)^*\cong \Pic S_{\gamma} \cong \ker (\Hom(\mathrm{Gal}(L/k), \Q/\Z) \to \Hom(\mathrm{Gal}(L/K), \Q/\Z)),\end{equation}
where $L$ is the Galois closure of $K/k$.
\begin{enumerate}
    \item In the case that $K/k$ is a Galois extension, then $L=K$ and so we have \[\mathfrak{K}(S_\gamma/k)^*  \cong \mathrm{Hom}(\mathrm{Gal}(K/k),\mathbb{Q}/\mathbb{Z})\cong \mathbb{Z}/n\mathbb{Z}.\]  
    Indeed, the degree $n$ is a prime number.
    \item In the case that $K/k$ is not a Galois extension, for $\xi\in \mathfrak{K}(S_\gamma/k)^*$, let $\phi\in \mathrm{Hom}(\mathrm{Gal}(L/k),\mathbb{Q}/\mathbb{Z})$ be the image of $\xi$ under the composition of isomorphisms in (\ref{eq:Brgroup_when_SLn}).
    Since $\phi$ vanishes on $\mathrm{Gal}(L/K)\subset \mathrm{Gal}(L/k)$, a normal subgroup $\ker \phi \unlhd \mathrm{Gal}(L/k)$ contains $\mathrm{Gal}(L/K)$.
    By assumption that $[K:k]=[\mathrm{Gal}(L/k):\mathrm{Gal}(L/K)]$ is prime, we have $\ker \phi=\mathrm{Gal}(L/k)$ or $\ker \phi= \mathrm{Gal}(L/K)$.
    If $\ker \phi= \mathrm{Gal}(L/K)$, then it contradicts the assumption that $K/k$ is not Galois.
    Therefore $\ker \phi=\mathrm{Gal}(L/k)$ and this direct yields that $\mathfrak{K}(S_\gamma/k)^*$ is trivial.
\end{enumerate}
\end{proof}
\begin{remark}\label{rmk:kappa_matching}
    Suppose that $K/k$ is a Galois extension.
    We recall that $\mathfrak{K}(S_\gamma/k)^*\cong \pi_0(\hat{S}_\gamma^\Gamma)=\hat{S}_\gamma^{\Gamma}$. 
    Here, via the isomorphism $\hat{S}_\gamma\cong \mathrm{Hom}(X_*(S_\gamma),\mathbb{C}^\times)$, an element of $\hat{S}_\gamma^\Gamma$ is identified with a $\Gamma$-invariant character of $X_*(S_\gamma)$.
    The $\Gamma$-action on $X_*(S_\gamma)$ factors through $\mathrm{Gal}(K/k)\cong \mathbb{Z}/n\mathbb{Z}$, which acts on $X_*(S_\gamma)\cong \mathbb{Z}^n_{\Sigma=0}$ by permuting the factors.
    Here, $\mathbb{Z}^n_{\Sigma=0}$ denotes the set of $(z_1,\cdots,z_n)\in \mathbb{Z}^n$ such that $\sum_i z_i=0$.
    
    Therefore, an element $\kappa\in \mathfrak{K}(S_\gamma/k)^*$ is identified with an character of $X_*(S_\gamma)$ satisfying that $\kappa(\sigma\alpha)=\kappa(\alpha)$ for all $\sigma\in\mathrm{Gal}(K/k)$ and $\alpha \in X_*(S_\gamma)$.
    It then follows that a $\Gamma$-invariant character is completely determined by the value at $(1,-1,0,\cdots,0)\in X_*(S_\gamma)$, which must be an $n$-th root of unity.
\end{remark}

In the case that $\kappa$ is non-trivial and the corresponding endoscopic group is unramified, we use the fundamental lemma stated in Remark \ref{rmk:fundamental_lemma_for_quasi_split}.
The following lemma describes the normalization of the restriction of $\mu_v$ to each $\mathrm{SL}_n(k_v)$-orbit $\gamma'\cdot \mathrm{SL}_n(k_v)$ in $X(k_v)$, which is required for the application of the fundamental lemma.
\begin{lemma}\label{lem:measure_comparison_gln_sln}
    Fix a finite place $v$ of $k$ such that $K_{v,i}/k_v$ is unramified extension for each $i\in B_v(\chi)$.
    For $\gamma'\in X(k_v)$, let $S_{\gamma'}$ be the centralizer of $\gamma'$ in $\mathrm{SL}_n$.
    For an $\mathrm{SL}_n(k_v)$-orbit $\gamma'\cdot \mathrm{SL}_n(k_v)\cong S_{\gamma'}(k_v)\backslash \mathrm{SL}_n(k_v)$ in $X(k_v)$, we have
    \[
    \left.\mu_v\right|_{\gamma'\cdot \mathrm{SL}_n(k_v)}=\frac{dh_v}{ds_v}
    \]
    where $dh_v$ (resp. $ds_v$) denotes a Haar measure on $\mathrm{SL}_n(k_v)$ (resp. $S_{\gamma'}(k_v)$) such that $vol(dh_v,\mathrm{SL}_n(\mathcal{O}_{k_v}))=1$ (resp. $vol(ds_v,S_c')=1$ for the maximal compact subgroup $S_c'$ of $S_{\gamma'}(k_v)$).
\end{lemma}\noindent
For the proof of Lemma \ref{lem:measure_comparison_gln_sln}, see Appendix \ref{app:meas_diff}.

\begin{theorem}\label{thm:application}
Let $\chi(x)\in\mathcal{O}_k[x]$ be an irreducible polynomial such that $\chi(0)=1$.
Let $\mathbf{X}$ be an $\mathcal{O}_k$-scheme representing the set of $n\times n$ integral matrices whose characteristic polynomial is $\chi(x)$ and $X:=\mathbf{X}\otimes_{\mathcal{O}_k}k$.
We define 
\[N(X, T)=|\{x\in \mathbf{X}(\mathcal{O}_k)\mid \norm{x}_\infty\leq T\}|,\]
for $T>0$, where the norm $\norm{\cdot}_\infty$ is defined in (\ref{intro:eq:modified_norm}).
Under the condition (\ref{intro:condition_ours}), we have the following asymptotic formulas.
\begin{enumerate}
        \item If $K/k$ is not Galois or ramified Galois, then
        \[ N(X, T) \sim C_T\prod_{v<\infty_k}\frac{\mathcal{O}_{\gamma,d\mu_v}^{\mathrm{GL}_n}(\mathbbm{1}_{\mathrm{GL}_n(\mathcal{O}_{k_v})})}{q_v^{S_v(\gamma)}}.
        \]
    
        \item If $K/k$ is unramified Galois, then
        \[N(X, T) \sim C_T \left(\prod_{v<\infty_k}\frac{\mathcal{O}_{\gamma,d\mu_v}^{\mathrm{GL}_n}(\mathbbm{1}_{\mathrm{GL}_n(\mathcal{O}_{k_v})})}{q_v^{S_v(\gamma)}}+n-1\right).\]
\end{enumerate}\noindent
Here, $\mathcal{O}_{\gamma,d\mu_v}^{\mathrm{GL}_n}(\mathbbm{1}_{\mathrm{GL}_n(\mathcal{O}_{k_v})})$ denotes the orbital integral of $\mathrm{GL}_n$ for $\mathbbm{1}_{\mathrm{GL}_n(\mathcal{O}_{k_v})}$ at $\gamma$ with respect to the measure $d\mu_v$ defined in (\ref{eq:quotient_measure}), $S_v(\gamma)$ is the $\mathcal{O}_{k_v}$-module length between $\mathcal{O}_{K_v}$ and $\mathcal{O}_{k_v}[x]/(\chi(x))$, 
    \[C_T:= |\Delta_k|^{\frac{-n^2+n}{2}} \frac{R_K h_K \sqrt{\abs{\Delta_K}}^{-1}}{R_k h_k \sqrt{\abs{\Delta_k}}^{-1}} \left(\prod_{i=2}^n \zeta_k(i)^{-1}\right)\left( \frac{2^{n-1}w_n\pi^{\frac{n(n+1)}{4}} }{\prod_{i=1}^n \Gamma(\frac{i}{2})} T^{\frac{n(n-1)}{2}}  \right)^{[k:\Q]},\]
    and we use the following notations:
    \begin{itemize}
        \item $R_F$ is the regulator of $F$, and $h_F$ is the class number of $\mathcal{O}_F$ for $F=k$ or $K$. 
        \item $w_n$ is the volume of the unit ball in $\mathbb{R}^{\frac{n(n-1)}{2}}$, and $\zeta_k$ is the Dedekind zeta function of $k$.
    \end{itemize}
 \end{theorem}
\begin{proof}
Let $\gamma\in\mathrm{SL}_n(\mathcal{O}_k)$ be an element whose characteristic polynomial is $\chi(x)$.
Here, the existence of such an element is guaranteed by the companion matrix of $\chi(x)$.
By the irreducibility of $\chi(x)$, the element $\gamma$ is elliptic regular, and we have $X\cong S_\gamma\backslash \mathrm{SL}_n$.
Note that $\mathbf{X}$ is the schematic closure of $X$ in $\mathrm{SL}_{n,\mathcal{O}_k}$.
To ease the notation, we set $f:=f_{\mathbf{X},T}$ and $\tilde{f}:=f_{\mathrm{SL}_{n,\mathcal{O}_k},T} $ defined in (\ref{eq:choice_of_G_integral}).
Then $\tilde{f}_v=\mathbbm{1}_{\mathrm{SL}_n(\mathcal{O}_{k_v})}$ for $v<\infty_k$.
\begin{enumerate}
    \item 
For $k=\mathbb{Q}$, by \cite[Theorem 6.1]{WX}, it follows that
\[
N(X,T)\sim \int_{X(\mathbb{A}_k)} f\ m^X =\mathcal{SO}_\gamma(\tilde{f}).
\]
Indeed, $X(\mathbb{A}_k)$ is identified with the stable conjugacy class of $\gamma$ in $\mathrm{SL}_n(\mathbb{A}_k)$ by Remark \ref{rmk:conjugacy_and_pointsofX}.
    \item
For a general number field $k$ where $\deg(K/k)$ is prime, using Theorem \ref{thm:main}, we have
    \[
    N(X,T)\sim \sum_{\kappa\in \mathfrak{K}(S_\gamma/k)^*}\mathcal{O}_{\gamma}^\kappa(\tilde{f})
~~~~~~\text{ where }~~ 
    \mathcal{O}_{\gamma}^\kappa(\tilde{f})=
    \sum_{\substack{\gamma'\in \mathrm{SL}_n(\mathbb{A}_k)/\sim\\ \gamma'\sim_{st} \gamma}}\kappa(inv(\gamma',\gamma))\mathcal{O}_{\gamma'}(\tilde{f}).
    \]
We note that, by Remark \ref{rmk:equidistribution_for_SLn}, the equidistribution property for $X$ is already satisfied.
By \cite[Section 5.7]{Kal}, under the isomorphism $H^1(k,S_\gamma(\mathbb{A}_{\bar{k}}))\cong \bigoplus_{v\in\Omega_k}H^1(k_v,S_\gamma)$, the adelic invariant $inv(\gamma',\gamma)$ corresponds to the tuple $(inv_v(\gamma_v',\gamma))_{v\in\Omega_k}$ of local invariants. 
Under the Tate-Nakayama duality, the composite map $\bigoplus_{v\in\Omega_k}H^1(k_v,S_\gamma)\xrightarrow{\cong}H^1(k,S_\gamma(\mathbb{A}_{\bar{k}}))\rightarrow \mathfrak{K}(S_\gamma/k)$ is dual to the diagonal embedding $\hat{S}_\gamma^\Gamma\rightarrow \prod_{v\in\Omega_k}\hat{S}_\gamma^{\Gamma_v}$.
Therefore, we have
\[
\mathcal{O}_{\gamma}^\kappa(\tilde{f})=
\prod_{v\in\Omega_k}\mathcal{O}_{\gamma}^{\kappa_v}(\tilde{f}_v)~~~~~~\text{ where }~~\mathcal{O}_{\gamma}^{\kappa_v}(\tilde{f}_v)=\sum_{\substack{\gamma_v'\in \mathrm{SL}_n(k_v)/\sim \\ \gamma_v'\sim_{st}\gamma}}\kappa_v(inv_v(\gamma_v',\gamma))\mathcal{O}_{\gamma_v'}(\tilde{f}_v),
\]
for the image $(\kappa_v)_{v\in\Omega_k}$ of $\kappa$ under the embedding 
$\hat{S}_\gamma^\Gamma\rightarrow \prod_{v\in\Omega_k}\hat{S}_\gamma^{\Gamma_v}$.

Using the identification $\kappa$ as an element of $\mathrm{Hom}(X_*(S_\gamma),\mathbb{G}_m)$ in Remark \ref{rmk:kappa_matching}, the endoscopic group $H$, corresponding to the pair $(\gamma,\kappa)$ by \cite[Lemma 9.7]{Kot86}, is described as follows:
\begin{itemize}
    \item 
For a trivial $\kappa$, we have $H=\mathrm{SL}_n$.
    \item
For a non-trivial $\kappa$, since $n$ is prime,
we have \[R^\vee(S_\gamma,H)=\{\alpha^\vee\in \{e_i-e_j\}_{1\leq i,j\leq n}\mid \langle \kappa,\alpha^\vee \rangle=1 \}=\emptyset,
\]where $e_i=(\underbrace{0,\cdots,0,1}_{i},  \underbrace{0,\cdots,0}_{n-i}) \in \mathbb{Z}^n$ with the identification $X_*(S_\gamma)\cong \mathbb{Z}_{\Sigma=0}^n$ (see Remark \ref{rmk:root_system_of_endoscopic_group}).  
Its dual $R(S_\gamma, H)$ is therefore empty, and so is the basis $\Delta^H$. 
Hence, the $\Gamma$-action on $X^*(S_\gamma)$ induced from that on $S_\gamma$ already preserves $\Delta^H$.
By Remark \ref{rmk:root_system_of_endoscopic_group}, it follows that the corresponding endoscopic group $H$ is $S_\gamma$.
\end{itemize}
In the case that $K_v/k_v$ is a field extension, we have $S_{\gamma,k_v}\cong \mathrm{R}^{(1)}_{K_v/k_v}\mathbb{G}_{m,K_v}$, where $\mathrm{R}^{(1)}_{K_v/k_v}\mathbb{G}_{m,K_v}$ denotes the kernel of the norm map $\mathrm{Nm}_{K_v/k_v}: \mathrm{R}_{K_v/k_v}\mathbb{G}_{m,K_v}\rightarrow \mathbb{G}_{m,k_v}$.
In particular, the splitting field of $S_{\gamma,k_v}$ is $K_v$.
By \cite[Proposition 7.5]{Kot86}, if the extension $K_v/k_v$ for $v<\infty_k$ is ramified, then for a non-trivial $\kappa$,  
\[\mathcal{O}_{\gamma}^{\kappa_v}(\tilde{f}_v)=\mathcal{O}_{\gamma}^{\kappa_v}(\mathbbm{1}_{\mathrm{SL}_n(\mathcal{O}_{k_v})})=0.\]
\end{enumerate}
In conclusion, we have
\[
N(X,T)\sim \left\{
\begin{array}{l l}
\mathcal{SO}_{\gamma}(\tilde{f})& \text{if $K/k$ is not Galois or ramified Galois};\\
\mathcal{SO_\gamma}(\tilde{f})+\sum_{\kappa\neq1}\mathcal{O}_{\gamma}^\kappa(\tilde{f})&\text{if $K/k$ is unramfied Galois}.
\end{array}
\right.
\]
Note that, when $k = \mathbb{Q}$, the only case (1) is possible since any finite extension of $\mathbb{Q}$ is ramified.

By the $\mathrm{GL}_n$-homogeneous space structure of $X$, the local stable orbital integral $\mathcal{SO}_\gamma(\tilde{f}_v)$ is identified with the  local (stable) orbital integral of $\mathrm{GL}_n$.
By the equality (\ref{eq:tamagawa_measure_in_SLn}) and Lemma \ref{lem:C_meas_value}, we then have
\begin{align*}
\mathcal{SO}_{\gamma}(\tilde{f})=|\Delta_k|^{\frac{-n^2+n}{2}} C_{meas}\cdot 
vol(d\mu_{\infty},X(k_\infty,T))
\prod_{v<\infty_k}\mathcal{O}_{\gamma,d\mu_v}^{\mathrm{GL}_n}(\mathbbm{1}_{\mathrm{GL}_n(\mathcal{O}_{k_v})}),
\end{align*}
since $X(k_v)\cap \mathrm{SL}_n(\mathcal{O}_{k_v})=X(k_v)\cap \mathrm{GL}_n(\mathcal{O}_{k_v})$.
We refer to Lemma \ref{lem:infinite_place_calculation} for the asymptotic formula for $vol(d\mu_\infty,X(k_\infty,T))$.
This yields that
\[
\mathcal{SO}_\gamma(\tilde{f})\sim|\Delta_k|^{\frac{-n^2+n}{2}} \frac{R_K h_K \sqrt{\abs{\Delta_K}}^{-1}}{R_k h_k \sqrt{\abs{\Delta_k}}^{-1}} \left(\prod_{i=2}^n \zeta_k(i)^{-1}\right)\left( \frac{2^{n-1}w_n\pi^{\frac{n(n+1)}{4}} }{\prod_{i=1}^n \Gamma(\frac{i}{2})} T^{\frac{n(n-1)}{2}}  \right)^{[k:\Q]}\frac{\mathcal{O}_{\gamma,d\mu_v}^{\mathrm{GL}_n}(\mathbbm{1}_{\mathrm{GL}_n(\mathcal{O}_{k_v})})}{q_v^{S_v(\gamma)}}.
\]


In the case that $K/k$ is unramified and Galois, for $\kappa\neq 1$, the equality (\ref{eq:tamagawa_measure_in_SLn}) and Lemma \ref{lem:C_meas_value} yield that
\[
\mathcal{O}_\gamma^\kappa(\tilde{f})=|\Delta_k|^{\frac{-n^2+n}{2}} C_{meas}\cdot 
\prod_{v\in\Omega_k}\mathcal{O}_{\gamma,d\mu_v}^{\kappa_v}(\tilde{f}_v)
\]
where $\mathcal{O}_{\gamma,d\mu_v}^{\kappa_v}(\tilde{f}_v)=\sum\limits_{\substack{\gamma_v'\in \mathrm{SL}_n(k_v)/\sim\\ \gamma_v'\sim_{st}\gamma}}\kappa_v(inv_v(\gamma_v',\gamma))\mathcal{O}_{\gamma_v',d\mu_v}(\tilde{f}_v)$.
For each $v\in\Omega_k$, the local $\kappa$-orbital integral  $\mathcal{O}_{\gamma,d\mu_v}^{\kappa_v}(\tilde{f}_v)$ with $\kappa\neq 1$ is computed as follows.
\begin{itemize}
    \item 
For $v\in\infty_k$, by the assumption that $k$ and $K$ are totally real, we have $S_{\gamma,k_v}\cong \mathbb{G}_{m,k_v}^{n-1}$ and thus $\hat{S}_\gamma^{\Gamma_v}\cong H^1(k_v,S_\gamma)=1$ (Hilbert's Theorem 90). This yields that  
\[
\mathcal{O}_{\gamma,d\mu_\infty}^{\kappa_\infty}(\tilde{f}_\infty)=vol(d\mu_\infty,X(k_\infty,T)).
\]
\item
Otherwise, the endoscopic group $S_{\gamma,k_v}$ is unramified.
For $v<\infty_k$, by Remark \ref{rmk:fundamental_lemma_for_quasi_split} and Lemma \ref{lem:measure_comparison_gln_sln}, we obtain
\[
\mathcal{O}_{\gamma,d\mu_v}^{\kappa_v}(\mathbbm{1}_{\mathrm{SL}_n(\mathcal{O}_{k_v})})=
\left(\prod_{\alpha\in R(S_\gamma,\mathrm{SL_n})}|\alpha(\gamma)-1|_v^{1/2}\right)^{-1}\mathbbm{1}_{S_{c}}(\gamma_H)=|\Delta_\chi|_v^{-\frac{1}{2}},
\]where $S_{c}$ denotes the maximal compact subgroup of $S_\gamma(k_v)$. 
The last equality follows from \cite[Example 3.8]{Gor}.
Moreover, as in the proof of \cite[Proposition 2.5]{CKL}, we have $|\Delta_\chi|_v^{-\frac{1}{2}}=q_v^{S_v(\gamma)}$.
\end{itemize}
Collecting the above results, we have
\[
\mathcal{O}_\gamma^\kappa(\tilde{f})\sim|\Delta_k|^{\frac{-n^2+n}{2}} \frac{R_K h_K \sqrt{\abs{\Delta_K}}^{-1}}{R_k h_k \sqrt{\abs{\Delta_k}}^{-1}} \left(\prod_{i=2}^n \zeta_k(i)^{-1}\right)\left( \frac{2^{n-1}w_n\pi^{\frac{n(n+1)}{4}} }{\prod_{i=1}^n \Gamma(\frac{i}{2})} T^{\frac{n(n-1)}{2}}  \right)^{[k:\Q]} ~~~~\text{ for $\kappa\neq 1$}.
\]
\end{proof}
\begin{remark}\label{rmk:about_orbitalintegral_of_GLn}
    For any finite place $v$ of $k$, the value of the orbital integral $\mathcal{O}_{\gamma,d\mu_v}^{\mathrm{GL}_n}(\mathbbm{1}_{\mathrm{GL}_n(\mathcal{O}_{k_v})})$, as noted previously, is an integer given by a $q_v$-polynomial whose leading term is expected to be $q_v^{S_v(\gamma)}$ by \cite{Yun13}.
    Since $\chi(x)\in\mathcal{O}_k[x]$, it follows that  $S_v(\gamma)=0$ for almost all $v<\infty_k$. 
    In fact,
    $\prod\limits_{v<\infty_k}\frac{\mathcal{O}_{\gamma,d\mu_v}^{\mathrm{GL}_n}(\mathbbm{1}_{\mathrm{GL}_n(\mathcal{O}_{k_v})})}{q_v^{S_v(\gamma)}}$ is a finite product.
    Furthermore, the explicit closed formulas for $\mathcal{O}_{\gamma,d\mu_v}^{\mathrm{GL}_n}(\mathbbm{1}_{\mathrm{GL}_n(\mathcal{O}_{k_v})})$ have been established in the following cases:
    \begin{enumerate}
        \item For $n=2$ and $3$, the close formula is given in \cite{CKL}.
        \item In the case that $\mathcal{O}_{k_v}[x]/(\chi(x))$ is a Bass ring, i.e. every ideal of $\mathcal{O}_{k_v}[x]/(\chi(x))$ is generated by 2 elements, the closed formula is given in \cite{CHL}.
    \end{enumerate}
\end{remark}
\appendix
\section{Computations in Section \ref{sec:application}}\label{sec:appendix_A}

In this appendix, we provide the proofs of Lemma \ref{lem:C_meas_value} and Lemma \ref{lem:measure_comparison_gln_sln}, and compute the value of $vol(d\mu_\infty, X(k_\infty,T))$.
We continue to use the notations in Section \ref{sec:application}.

\subsection{Measure comparisons}\label{app:meas_diff}
For the proofs of Lemma \ref{lem:C_meas_value} and Lemma \ref{lem:measure_comparison_gln_sln}, we describe the structure of $T_{\gamma,k_v}$ and its model for each $v\in\Omega_k$.
By \cite[3.12]{Vosk}, one has that $T_{\gamma,k_v}\cong \prod_{i\in B_v(\chi)}\mathrm{R}_{K_{v,i}/k_v}\mathbb{G}_{m,K_{v,i}}$.
For each $i\in B_v(\chi)$, the component $\mathrm{R}_{K_{v, i}/k_v}\mathbb{G}_{m, K_{v, i}}$ has a smooth model over $\mathcal{O}_{k_v}$, given by $\mathrm{R}_{\mathcal{O}_{K_{v, i}}/\mathcal{O}_{k_v}}\mathbb{G}_{m,\mathcal{O}_{K_{v, i}}}$ (see \cite[B.3.(2)]{KP23}). 
Thus, $T_{\gamma,k_v}$ admits a smooth model\[
T_{\gamma,\mathcal{O}_{k_v}}\cong\prod_{i\in B_v(\chi)}\mathrm{R}_{\mathcal{O}_{K_{v,i}}/\mathcal{O}_{k_v}}\mathbb{G}_{m,\mathcal{O}_{K_{v,i}}},\]
which is referred to as the standard integral model of $T_{\gamma,k_v}$.
In particular, the set of $\mathcal{O}_{k_v}$-points $T_{\gamma,\mathcal{O}_{k_v}}(\mathcal{O}_{k_v})= \prod_{i\in B_v(\chi)}\mathcal{O}_{K_{v,i}}^{\times}$ is the maximal compact open subgroup of $T_\gamma(k_v)$.
\begin{proof}[Proof of Lemma \ref{lem:C_meas_value}]
To compare $\prod_{v\in\Omega_k}|\omega_X^{\mathrm{GL}_n}|_v$ and $\prod_{v\in\Omega_k}d\mu_v$, we define a measure which plays a role of a bridge between these two measures.
For $v< \infty_k$, let $\omega_{X_{k_v}}^{can}$ be a $\mathrm{GL}_{n,k_v}$-invariant form on $X_{k_v}$ such that $\omega_{\mathrm{GL}_{n,k_v}}^{can}=\omega_{X_{k_v}}^{can}\cdot\omega_{T_{\gamma,k_v}}^{can}$, where $\omega_{\mathrm{GL}_{n,k_v}}^{can}$ (resp. $\omega_{T_{\gamma,k_v}}^{can}$) is the invariant form on $\mathrm{GL}_{n,k_v}$ (resp. $T_{\gamma,k_v}$) defined in \cite[Section 4]{Gro97}.
For $v\in \infty_k$, recall the volume form $\omega_{X_{k_v}}^{can}$ from the description given before Lemma \ref{lem:C_meas_value}.
We then have
\begin{equation}\label{eq:measrue_com_local}
|\omega_{X_{k_v}}^{can}|_v=\left\{
\begin{array}{c l}
    \frac{|\mathrm{GL}_{n, \mathcal{O}_{k_v}}(\kappa_v)|\cdot q_v^{-n^2}}{| T_{\gamma,\mathcal{O}_{k_v}}(\kappa_v)|\cdot q_v^{-n}}d\mu_v & v<\infty_k; \\
    d\mu_v & v\in\infty_k.
\end{array}
\right.\end{equation}
Here, the terms when $v<\infty_k$ are formulated as follows
    \[
    \left\{
    \begin{array}{l}
        |\mathrm{GL}_{n, \mathcal{O}_{k_v}}(\kappa_v)|\cdot q_v^{-n^2} = (1-\frac{1}{q_v^n}) (1-\frac{1}{q_v^{n-1}}) \cdots (1- \frac{1}{q_v});\\
        |T_{\gamma,\mathcal{O}_{k_v}}(\kappa_v)|\cdot q_v^{-n} = \prod_{i \in B_v(\chi)} ( 1- \frac{1}{|\kappa_{K_{v,i}}|} ),
    \end{array}
    \right.
    \]
    where $|T_{\gamma,\mathcal{O}_{k_v}}(\kappa_v)|=| \mathrm{R}_{\mathcal{O}_{K_{v,i}} / \mathcal{O}_{k_v}}(\mathbb{G}_{m,\mathcal{O}_{K_{v,i}}})(\kappa_v)| =q_v^{[K_{v,i}:k_v]} (1 - \frac{1}{|\kappa_{K_{v,i}}|})$.
    Hence, it remains to compare $\prod_{v\in\Omega_k}|\omega_X^{\mathrm{GL}_n}|_v$ with $\prod_{v\in\Omega_k}|\omega_{X_{k_v}}^{can}|_v$.
    
For a connected reductive group $G$, by \cite[Corollary 7.3 and Proposition 9.3]{GG99}, we have
    \[
    \prod_{v\in \Omega_k}\frac{|\omega_{G_{k_v}}^{can}|_v}{|\omega_G|_v}=f(M_G)^{\frac{1}{2}},
    \]
    where $f(M_G)$ is the global conductor of the motive $M_G$ of $G$ (see \cite[Equation (9.1)]{GG99}). 
    For $G= \mathrm{GL}_n$ or $T_\gamma$, the conductor $f(M_G)^{\frac{1}{2}}$ is given as follows:
    \begin{enumerate}
        \item 
    Since $\mathrm{GL}_n$ is split over $k$, 
    we have $f(M_{\mathrm{GL}_n})^{\frac{1}{2}}=1$.
\item
    On the other hand, in the case of $T_\gamma$, the global conductor of the motive $M_{T_\gamma}$ is given by 
    $f(M_{T_{\gamma}})=\prod_{v<\infty_k} q_v^{a_v(X^*(T_\gamma)\otimes\mathbb{Q})}$ following \cite[Section 2]{Gro97}, 
    where $a_v(X^*(T_\gamma)\otimes\mathbb{Q})$ is the local Artin conductor of $X^*(T_\gamma)\otimes\mathbb{Q}$.
    By \cite[Proposition VII.11.7]{Neu} and the argument in \cite[Section 1.2]{Lee21},
    we have that
    \[f(M_{T_\gamma})=\prod_{v<\infty_k}\prod_{i\in B_v(\chi)}|\Delta_{K_{v,i}/k_v}|^{-1}.\]
    \end{enumerate}
Therefore, by Proposition \ref{prop:alg-top_match}, we have
\[
    \prod_{v\in \Omega_k}|\omega_X^{\mathrm{GL}_n}|_v= \left(
    \prod_{v< \infty_k}\prod_{i\in B_v(\chi)}|\Delta_{K_{v,i}/k_v}|^{-\frac{1}{2}}\right) \prod_{v\in \Omega_k} \abs{\omega_{X_{k_v}}^{can}}_v.
\]
Here, \cite[Proposition 2.5]{CKL} yields that
    \[
        \abs{\Delta_\chi}_v^{\frac{1}{2}} \prod_{i\in B_v(\chi)}|\Delta_{K_{v,i}/k_v}|_v^{-\frac{1}{2}} = \Big(\abs{\Delta_\chi}_v^{\frac{1}{2}} \prod_{i\in B_v(\chi)} 
        \abs{\Delta_{\chi_{v,i}}}_v^{-\frac{1}{2}}\Big)
        \prod_{i\in B_v(\chi)} q_v^{-S_v(\chi_{v,i})} 
        = q_v^{-S_v(\gamma)},
   \] 
where $S_v(\chi_{v,i})$ denotes the $\mathcal{O}_{k_v}$-module length between $\mathcal{O}_{K_{v,i}}$ and $\mathcal{O}_{k_v}[x]/(\chi_{v,i}(x))$.
The last equality follows from \cite[Section 4.1]{Yun13} and \cite[Corollary 1 of Proposition 11 in Chapter 4.6]{Bou}. 

    In conclusion, by using the class number formula for $K$ and $k$, we have
    \begin{align*}
    C_{meas}&=\left. \frac{\zeta_K(s)}{\zeta_k(s)} \right|_{s=1} \left(\prod_{i=2}^n \zeta_k(i)^{-1}\right)\left(\prod_{v<\infty_k}\frac{|\Delta_\chi|_v^{-\frac{1}{2}}}{q_v^{S_v(\gamma)}}\right)\\
    &=2^{(n-1)[k:\Q]} \frac{R_K h_K \sqrt{\abs{\Delta_K}}^{-1}}{R_k h_k \sqrt{\abs{\Delta_k}}^{-1}} \left(\prod_{i=2}^n \zeta_k(i)^{-1}\right)\left(\prod_{v<\infty_k}\frac{|\Delta_\chi|_v^{-\frac{1}{2}}}{q_v^{S_v(\gamma)}}\right).
    \end{align*}
\end{proof}

\begin{proof}[Proof of Lemma \ref{lem:measure_comparison_gln_sln}]
In this proof, to ease the notations, we omit the subscript $k_v$ indicating the base change $(-)\times_k k_v$ for any scheme, whenever no confusion arises.
We compare the measure $|\omega_X^{can}|_v$, defined in the proof of Lemma \ref{lem:C_meas_value}, with $dh_v/ds_v$.
The desired result then follows from the equality given in (\ref{eq:measrue_com_local}).
For the centralizer $T_{\gamma'}$ of $\gamma'$ in $\mathrm{GL}_{n}$, let $\omega_{T_{\gamma'}}^{can}$ be an invariant form in \cite[Section 4]{Gro97}.
Since $T_\gamma\cong T_{\gamma'}$, the quotient measure $|\omega_{\mathrm{GL}_{n}}^{can}|_v/|\omega_{T_{\gamma'}}^{can}|_v$ coincides with  $|\omega_X^{can}|_v=|\omega_{\mathrm{GL}_{n}}^{can}|_v/|\omega_{T_{\gamma}}^{can}|_v$.
Thus we may and do work with $\gamma'=\gamma$.
    
    For a torus $S_{\gamma}\cong \mathrm{R}^{(1)}_{K_v/k_v}\mathbb{G}_{m,K_v}$, there is the standard integral model $S_{\gamma,\mathcal{O}_{k_v}}$ isomorphic to the kernel of the norm map $\Nm_{\mathcal{O}_{K_v}/\mathcal{O}_{k_v}}:T_{\gamma,\mathcal{O}_{k_v}}\rightarrow \mathbb{G}_{m,\mathcal{O}_{k_v}}$.
    Then, the set of its $\mathcal{O}_{k_v}$-points $S_{\gamma,\mathcal{O}_{k_v}}(\mathcal{O}_{k_v})$ coincides with the maximal compact subgroup $S_c$ of $S_\gamma(k_v)$. 
    Moreover, the following commutative diagram over $\mathcal{O}_{k_v}$ holds
    \begin{equation}\label{eq:commdiast}
    \begin{tikzcd}
	{S_{\gamma,\mathcal{O}_{k_v}}}&&{T_{\gamma,\mathcal{O}_{k_v}}} && {\mathbb{G}_{m,\mathcal{O}_{k_v}}} \\
	{\mathrm{SL}_{n,\mathcal{O}_{k_v}}}&&{\mathrm{GL}_{n,\mathcal{O}_{k_v}}} && {\mathbb{G}_{m,\mathcal{O}_{k_v}}.}
    \arrow["{\Nm}", from=1-3, to=1-5]
	\arrow[hook, from=1-1, to=2-1]
    \arrow[hook, from=1-3, to=2-3]
    \arrow[from=1-1, to=1-3]
    \arrow[from=2-1, to=2-3]
    \arrow["id", from=1-5, to=2-5]
	\arrow["{\mathrm{det}}", from=2-3, to=2-5]
    \end{tikzcd}
    \end{equation}
    Here the map $\Nm:=\Nm_{\mathcal{O}_{K_v}/\mathcal{O}_{k_v}}$ is smooth over $\mathcal{O}_{k_v}$ since $K_{v,i}/k_v$ is unramified extension for each $i\in B_v(\chi)$, and the map $\det$ is smooth over $\mathcal{O}_{k_v}$ as well.
    We denote by $\omega_{\mathrm{SL}_n}^{can}$ (resp. $\omega_{S_{\gamma}}^{can}$) an invariant form on $\mathrm{SL}_n$ (resp. $S_{\gamma}$) defined in \cite[Section 4]{Gro97}. 
    On $\gamma\cdot \mathrm{SL}_n(k_v)\cong S_\gamma(k_v)\backslash \mathrm{SL}_n(k_v)$, it follows that
    \begin{equation}\label{eq:volumeform_quotient_SLn}
    \frac{|\omega_{\mathrm{SL}_{n}}^{can}|_v}{|\omega_{S_{\gamma}}^{can}|_v}=\frac{|\mathrm{SL}_{n, \mathcal{O}_{k_v}}(\kappa_v)|\cdot q_v^{-(n^2-1)}}{ |S_{\gamma,\mathcal{O}_{k_v}}(\kappa_v)|\cdot q_v^{-(n-1)}}\cdot \frac{dh_v}{ds_v}.
    \end{equation}
    Thus, it suffices to compare $|\omega_{X}^{can}|_v$ with $|\omega_{\mathrm{SL}_{n}}^{can}|_v/|\omega_{S_{\gamma}}^{can}|_v$.
    We claim that $|\omega_{X}^{can}|_v=|\omega_{\mathrm{SL}_{n}}^{can}|_v/|\omega_{S_{\gamma}}^{can}|_v$.
    
    Let $\omega_{\mathbb{G}_m}^{can}$ denote an invariant form on $\mathbb{G}_m$ which has good reduction (mod $\pi_v$), as defined in \cite[p. 293]{Gro97}.
    Since $\mathbb{G}_m\cong \mathrm{SL}_n\backslash \mathrm{GL}_n$, one can deduce that $\det^*\omega_{\mathbb{G}_m}^{can}\wedge \tilde{\omega}_{\mathrm{SL}_n}^{can}$ is a non-zero element in $\bigwedge^{n^2}\Omega_{\mathrm{GL}_n/k_v}(\mathrm{GL}_n)$, where $\tilde{\omega}_{\mathrm{SL}_n}^{can}$ is a lifting of $\omega_{\mathrm{SL}_n}^{can}$ in Definition \ref{deg:algmatch}.
    Since $\det$ is smooth over $\mathcal{O}_{k_v}$, the invariant form $\det^*\omega_{\mathbb{G}_m}^{can}\wedge \tilde{\omega}_{\mathrm{SL}_n}^{can}$ has good reduction (mod $\pi_v$). Hence, $\omega_{\mathrm{GL}_n}^{can}$ agrees with $\det^*\omega_{\mathbb{G}_m}^{can}\wedge \tilde{\omega}_{\mathrm{SL}_n}^{can}$ up to $\mathcal{O}_{k_v}^\times$.
    By the same argument, 
    $\omega_{T_\gamma}^{can}$ agrees with $\Nm ^*\omega_{\mathbb{G}_m}^{can}\wedge \tilde{\omega}_{S_\gamma}^{can}$ up to $\mathcal{O}_{k_v}^\times$, since $\mathbb{G}_m\cong S_\gamma\backslash T_\gamma$ and $\Nm$ is smooth over $\mathcal{O}_{k_v}$.

We denote  by $\tilde{\omega}^{\mathrm{GL}_n}$ a lifting of a volume form $\omega$ on $S_\gamma,T_\gamma$, and $\mathrm{SL}_n$ along the embedding into $\mathrm{GL}_n$.
For $\varphi_{\mathrm{GL}_n}$ in (\ref{eq:generalGXmap}), recall that
\begin{equation}\label{eq:volumform_equation_1}
\omega_{\mathrm{GL}_n}^{can}= \varphi_{\mathrm{GL}_n}^* \omega^{can}_X \wedge \tilde{\omega}_{T_\gamma}^{can,\mathrm{GL}_n}.
\end{equation}
The left-hand side turns out to be $\det^*\omega_{\mathbb{G}_m}^{can}\wedge \tilde{\omega}_{\mathrm{SL}_n}^{can}$, up to $\mathcal{O}_{k_v}^\times$.
Since the right-hand side is independent of the choice of $\tilde{\omega}_{T_\gamma}^{can,\mathrm{GL}_n}$ (see \cite[p. 24]{Wei82}),
we may choose a lifting  $\tilde{\omega}_{T_\gamma}^{can,\mathrm{GL}_n}=\det^*\omega_{\mathbb{G}_m}^{can}\wedge\tilde{\omega}_{S_\gamma}^{can,\mathrm{GL}_n}$, up to $\mathcal{O}_{k_v}^\times$.
Here, by the commutative diagram (\ref{eq:commdiast}), $\det^*\omega_{\mathbb{G}_m}^{can}$ indeed provides a lifting of $\mathrm{Nm}^*\omega_{\mathbb{G}_m}^{can}$ along $T_\gamma\hookrightarrow \mathrm{GL}_n$.
We then have the following relation, up to $\mathcal{O}_{k_v}^\times$,
\[
\mathrm{det}^* \omega_{\mathbb{G}_m}^{can}\wedge\tilde{\omega}^{can,\mathrm{GL}_n}_{\mathrm{SL}_n}=\varphi_{\mathrm{GL}_n}^*\omega_X^{can} \wedge\mathrm{det}^*\omega_{\mathbb{G}_m}^{can}\wedge \tilde{\omega}_{S_\gamma}^{can,\mathrm{GL}_n}.
\]
Therefore, for some $c\in\mathcal{O}_{k_v}^\times$, the $(n^2-1)$-form $\tilde{\omega}^{can,\mathrm{GL}_n}_{\mathrm{SL}_n}+c\cdot\varphi_{\mathrm{GL}_n}^*\omega_X^{can}\wedge \tilde{\omega}_{S_\gamma}^{can,\mathrm{GL}_n}$ on $\mathrm{GL}_n$ is annihilated by the exterior product with $\mathrm{det}^*\omega_{\mathbb{G}_m}^{can}$.
Since the form $\det^*\omega_{\mathbb{G}_m}^{can}$ is of degree 1, there exists an $(n^2-2)$-form $\omega'$ on $\mathrm{GL}_n$ such that
\[
\tilde{\omega}^{can,\mathrm{GL}_n}_{\mathrm{SL}_n}+c\cdot\varphi_{\mathrm{GL}_n}^*\omega_X^{can}\wedge \tilde{\omega}_{S_\gamma}^{can,\mathrm{GL}_n}=\mathrm{det}^*\omega_{\mathbb{G}_m}^{can}\wedge \omega'.
\]
Restricting this equation to $\mathrm{SL}_n$, by imposing $\det=1$ with $d(\det)=0$, we obtain the following relation, up to $\mathcal{O}_{k_v}^\times$,
\[
\omega_{\mathrm{SL}_{n}}^{can}=\varphi_{\mathrm{SL}_n}^*\omega_{X}^{can}\wedge \tilde{\omega}_{S_\gamma}^{can,\mathrm{SL}_n},
\]
where $\tilde{\omega}_{S_\gamma}^{can,\mathrm{SL}_n}$ denotes a lifting of $\omega_{S_\gamma}^{can}$ along $S_\gamma\hookrightarrow \mathrm{SL}_n$.
By Proposition \ref{prop:alg-top_match}, this yields the claim and we have
\[d\mu_v=\frac{|T_{\gamma,\mathcal{O}_{k_v}}(\kappa_v)|\cdot q_v^{-n}}{| \mathrm{GL}_{n, \mathcal{O}_{k_v}}(\kappa_v)|\cdot q_v^{-n^2}}|\omega_{X_{k_v}}^{can}|_v=\frac{|S_{\gamma,\mathcal{O}_{k_v}}(\kappa_v)|\cdot q_v^{-(n-1)}}{|\mathrm{SL}_{n, \mathcal{O}_{k_v}}(\kappa_v)|\cdot q_v^{-(n^2-1)}}|\omega_{X_{k_v}}^{can}|_v=\frac{dh_v}{ds_v}\]
on $\gamma\cdot \mathrm{SL}_n(k_v)\cong S_\gamma(k_v)\backslash \mathrm{SL}_n(k_v)$ by (\ref{eq:measrue_com_local}) and (\ref{eq:volumeform_quotient_SLn}).
\end{proof}

\subsection{Asymptotic formula for $vol(d\mu_\infty,X(k_\infty,T))$}
\begin{lemma}\label{lem:infinite_place_calculation}
We have
\[vol(d\mu_\infty,X(k_\infty,T))\sim\left( \frac{w_n\pi^{\frac{n(n+1)}{4}} }{\prod_{i=1}^n \Gamma(\frac{i}{2})} T^{\frac{n(n-1)}{2}}  \right)^{[k:\Q]}\prod_{v\in\infty_k}|\Delta_\chi|_v^{-\frac{1}{2}}\]
where $w_n$ is the volume of the unit ball in $\mathbb{R}^{\frac{n(n-1)}{2}}$.
\end{lemma}

\begin{proof}
We fix an infinite place $v \in \infty_k$.
By the assumption that $k$ and $K$ are totally real number fields, the polynomial $\chi(x)$ splits completely over $k_v\cong \mathbb{R}$.
Let $\lambda_1, \ldots, \lambda_n \in k_v$ be the roots of $\chi(x)$, and $\lambda := \diag(\lambda_1, \ldots, \lambda_n) \in X(k_v)$ be the diagonal matrix.
We denote by $T_{\lambda,k_v}$ the centralizer of $\lambda$ in $\mathrm{GL}_{n, k_v}$.
Since $\lambda$ is regular, the roots $\lambda_i$ are distinct and so the stabilizer $T_{\lambda}(k_v)$ is the set of diagonal matrices in $\mathrm{GL}_n(k_v)$. 
By the argument in Section \ref{meas:gln}, there exists $g_0\in \mathrm{GL}_n(k_v)$ such that $g_0^{-1}\gamma g_0=\lambda$.
We define the map
\[
\phi: T_\lambda(k_v)\backslash \mathrm{GL}_n(k_v)\xrightarrow{\cong} T_\gamma(k_v)\backslash \mathrm{GL}_n(k_v),\ T_\lambda(k_v)g\mapsto T_\gamma(k_v) g_0g,
\]which is $\mathrm{GL}_n(k_v)$-invariant and set \[R_{T,v}:= \{T_\lambda(k_v) g\in T_\lambda(k_v)\backslash\mathrm{GL}_n(k_v) \mid \norm{g^{-1}\lambda g}_v\leq T\}.\]
We abuse the notation $d\mu_v$ for the pullback of $d\mu_v=|\omega_{X,k_v}^{can}|_v$ along the 
isomorphism $\phi$.
Then, 
we have \[d\mu_v=\frac{|\omega_{\mathrm{GL}_{n,k_v}}^{can}|_v}{\bigwedge_{i=1}^n\frac{dt_i}{|t_i|}}\] on $T_\lambda(k_v)\backslash \mathrm{GL}_n(k_v)$ where $(t_1,\cdots,t_n)$ are coordinates on $ T_\lambda\cong \mathbb{G}_{m,k_v}^n$, and the following equality holds
\[vol(d\mu_v,X(k_v,T))=vol(d\mu_v, R_{T,v}).\]

We consider the Iwasawa decomposition
\begin{equation}\label{eq:ankdecompo}
\mathrm{GL}_n(k_v)= ANK,
\end{equation}
where
$
\left\{
\begin{array}{l}
\textit{$A$ is the subgroup of $\mathrm{GL}_n(k_v)$ consisting of positive diagonal matrices};\\
\textit{$N$ is the subgroup of $\mathrm{GL}_n(k_v)$ consisting of unipotent upper triangular matrices};\\
\textit{$K$ is the orthogonal group $\mathrm{O}_n(k_v)\subset \mathrm{GL}_n(k_v)$}.\\ 
\end{array}
\right.
$
We fix a measure $da:=\bigwedge_{i=1}^n\frac{dt_i}{t_i}$ on $A\cong \mathbb{R}_{>0}^n$ and $dn := \bigwedge_{1\leq i<j \leq n} dx_{ij}$ on $N\cong \mathbb{R}^{\frac{n(n-1)}{2}}$.
For $f\in \mathcal{C}_c^\infty(T_\lambda(k_v))$, it follows that
\begin{equation}\label{eq:techtech}
\int_{T_\lambda(k_v)}f(t) \bigwedge_{i=1}^n\frac{dt_i}{\abs{t_i}}=\sum_{\sigma \in \Sigma}\int_A f(\sigma a)da,
\end{equation}
where $\Sigma=\{\pm 1\}^n$.
By \cite[Section 8]{Knapp}, with respect to the Iwasawa decomposition in (\ref{eq:ankdecompo}), a Haar measure $dk$ on $K$ can be normalized so that
\begin{equation}\label{eq:decompose_haar}
    dg:=|\omega_{\mathrm{GL}_{n,k_v}}^{can}|_v = \frac{1}{\abs{\det}^n} \bigwedge_{1\leq i,j \leq n} dx_{ij} = dadndk.
\end{equation}
For $f\in\mathcal{C}_c^\infty(G(k_v))$, we have
\begin{align*}
    \int_{\mathrm{GL}_n(k_v)}f(g)dg&= \int_{A\times N \times K} f(ank) da dn dk\\
    &=|\Sigma|^{-1}\sum_{\sigma\in\Sigma}\int_{A\times N\times K}f(a nk\sigma)dadndk\\
    &=\frac{1}{2^n}\int_{T_\lambda(k_v)\times N\times K}f(tnk)dtdndk,
\end{align*}where $dt:= \bigwedge_{i=1}^n\frac{dt_i}{\abs{t_i}}$. 
By Proposition \ref{prop:alg-top_match}, this yields that 
\[
\int_{X(k_v)}f(x) d\mu_v=\frac{1}{2^n}\int_{N\times K}f(nk)dndk.
\]
Here $R_{T,v}$ corresponds to the set $N_{T,v}\times K \subset N\times K$ where $N_{T,v}:= \{n \in N \mid \norm{n^{-1} \lambda n}_v \leq T\}$. 
Thus we have
\[
vol(d\mu_v, R_{T,v}) = \frac{1}{2^n} \cdot vol( N_{T,v}, dn) \cdot vol(K, dk).
\]
The arguments in \cite[p. 281-282]{EMS} (in that paper, the Haar measure on $K$ is normalized, so that its total volume is 1) imply that
\[vol( N_{T,v}, dn) \sim w_n \abs{\Delta_\chi}_v^{-\frac{1}{2}} T^{\frac{n(n-1)}{2}}.\]
Lemma \ref{lem:volomeofK}, provided below, proves that
\[vol(K, dk) = 2^n \pi^{\frac{n(n+1)}{4}} \prod_{i=1}^n \Gamma(\frac{i}{2})^{-1}.\]
Combining these results, we have
\[vol(d\mu_\infty,X(k_\infty,T))= \prod_{v\in \infty_k} vol(d\mu_v,X(k_v,T))\sim \left( \frac{w_n\pi^{\frac{n(n+1)}{4}} }{\prod_{i=1}^n \Gamma(\frac{i}{2})} T^{\frac{n(n-1)}{2}}  \right)^{[k:\Q]}\prod_{v\in\infty_k}|\Delta_\chi|_v^{-\frac{1}{2}}. \]
\end{proof}

\begin{lemma}\label{lem:volomeofK}
    Let $dk$ be the Haar measure on the orthogonal group $K = \mathrm{O}_n(k_v)$ satisfying (\ref{eq:decompose_haar}). We have
    \[vol(K, dk) = 2^n \pi^{\frac{n(n+1)}{4}} \prod_{i=1}^n \Gamma(\frac{i}{2})^{-1}.\]
\end{lemma}

\begin{proof}
    We consider the LDU-decomposition of $\mathrm{GL}_n(k_v)$ into $\bar{N}T_\lambda N$ where $\bar{N}$ is the subgroup of unipotent lower triangular matrices in $\mathrm{GL}_n(k_v)$ and we use $T_\lambda$ to denote $T_\lambda(k_v)$ by abuse of notation.
    Let $d\bar{n} = \bigwedge_{1\leq j<i \leq n} dx_{ij}$ be a Haar measure on $\bar{N}$.
    By the change of variables along the LDU-decomposition, we obtain that
    \begin{align*}
      \frac{1}{\abs{\det(x_{ij})}^n} \bigwedge_{1\leq i,j \leq n} dx_{ij} &= \frac{\abs{t_2^{2n-2} t_2^{2n-4} \cdots t_{n-1}^{2}}}{\abs{t_1 t_2 \cdots t_n}^n} \bigwedge_{1\leq j<i\leq n} dx_{ij} \wedge \bigwedge_{1\leq i \leq n} dt_i \wedge \bigwedge_{1\leq i<j\leq n} dx_{ij} \\ 
       &= \left| t_1^{n-1} t_2^{n-3} \cdots t_n^{-n+1} \right| \bigwedge_{1\leq j<i\leq n} dx_{ij} \wedge \bigwedge_{1\leq i \leq n} \frac{dt_i}{|t_i|} \wedge \bigwedge_{1\leq i<j\leq n} dx_{ij} \\
       &= e^{ 2\rho \log t} d\bar{n}dtdn,
     \end{align*}where $\rho$ denotes the half sum of positive roots of $T_\lambda$ in $\mathrm{GL}_n$, and so $e^{ 2\rho \log t} = \left| \prod_{1\leq i<j \leq n} \frac{t_i}{t_j} \right| = \abs{t_1^{n-1} t_2^{n-3} \cdots t_n^{-n+1}}$ for $t=\diag (t_1,\cdots,t_n)$.
    For $f \in \mathcal{C}_c^\infty(G(k_v))$ such that $f(kg)=f(g)$ for $k\in K$, we have
    \begin{align*}
        \int_{\mathrm{GL}_n(k_v)} f(x) dg &=\int_{\overline{N}\times T_\lambda \times N}f(\bar{n}tn)e^{2\rho \log t}d\bar{n}dtdn\\
        &=\int_{\overline{N}\times T_\lambda \times N}f(a(\bar{n})n(\bar{n})tn)e^{2\rho \log t}d\bar{n}dtdn\\
        &=\int_{\overline{N}\times T_\lambda \times N}f(a(\bar{n})tn)e^{2\rho \log a(\bar{n})t}e^{-2\rho\log a(\bar{n})} d\bar{n}dtdn\\
        &=\int_{\bar{N}}e^{-2\rho\log a(\bar{n})}d\bar{n}\int_{T_\lambda \times N}f(tn)e^{2\rho \log t} dtdn\\
        &=vol(dk, K)^{-1}\cdot2^n\int_{\bar{N}}e^{-2\rho\log a(\bar{n})}d\bar{n}\int_{K\times A\times N} f(kan)e^{2\rho \log t} dkdadn,
    \end{align*}
    where $\bar{n}=k(\bar{n})a(\bar{n})n(\bar{n})$ under the Iwasawa decomposition $\mathrm{GL}_n(k_v)=KAN$ described in (\ref{eq:ankdecompo}).
    Using the identities $dadndk = e^{ 2\rho \log t} dkda dn$ (see \cite[Proposition 8.43]{Knapp}) and (\ref{eq:decompose_haar}), we have
    \[
    vol(dk,K)=2^n\int_{\bar{N}}e^{-2\rho\log a(\bar{n})}d\bar{n}=2^n \pi^{\frac{n(n+1)}{4}} \prod_{i=1}^n \Gamma(\frac{i}{2})^{-1},
    \]
    where the last equality follows from \cite[Theorem 1, 14.10]{Vosk}.
    \end{proof}
\section{Automorphic interpretation via $L^2$-expansion}\label{app:spectral_expansion}
When $ N(X;f_{\mathbf{X},T})$ admits an automorphic (spectral) interpretation as in \cite{DRS}, its main term as $T\rightarrow \infty$ usually corresponds to the trivial representation as discussed in \cite{Get18}. 
In this appendix, we consider this expectation in our context.

We maintain the assumptions from Section \ref{sec:main_result}. Through  Remark \ref{rmk:conjugacy_and_pointsofX}, we identify $X(\mathbb{A}_k)$ with the stable conjugacy class of $\gamma$ in $G(\mathbb{A}_k)$.
For a function $f_{\mathbf{X},T}$ on $X(\mathbb{A}_k)$ defined in (\ref{def:alg_ftn}),
we define a function 
\[
F_{\mathbf{X},T}:[G]\rightarrow \mathbb{C}, g\rightarrow \sum_{x\in X(k)}f_{\mathbf{X},T}(x\cdot g).
\] 
Then one has $ N(X;f_{\mathbf{X},T})=F_{\mathbf{X},T}(1)$.
For a compact subset $\Omega \subset G(\mathbb{A}_k)$, nonzero summands in $F_{\mathbf{X},T}(g)$ for $g\in \Omega$ is finite since they correspond to elements $x \in X(k)\cap supp(f_{\mathbf{X},T})\cdot \Omega$.
The integral of $F_{\mathbf{X},T}$ over $[G]$ can be written as the finite sum of the products of Tamagawa number of $G_{\gamma'}$ and the orbital integral at $\gamma'$ for $f_{\mathbf{X},T}$, where $\gamma'\in G(k)/\sim$  such that $\gamma'\sim_{st}\gamma$.
Hence $F_{\mathbf{X},T}$ is an $L^1$-function. Since it is bounded, $F_{\mathbf{X},T}$ is a $L^2$-function.

Applying the $L^2$-expansion to $F_{\mathbf{X},T}$, we have a spectral interpretation of $ N(X;f_{\mathbf{X},T})$.
The contribution to $F_{\mathbf{X},T}(g)$ of a discrete automorphic representation $\pi\subset L^2([G])$ is formulated as follows
\[
\sum_{\varphi_i\in \mathcal{B}(\pi)}\int_{[G]}F_{\mathbf{X},T}(g')\varphi_i(g') m^{G}(g') \bar{\varphi}_i(g),
\]where $\mathcal{B}(\pi)$ is an orthonormal basis of the $\pi$-isotypic subspace $L^2(\pi)$ of $L^2([G])$.
Let $K:=K_\infty K^\infty < G(\mathbb{A}_k)$ where $K_\infty <G(k_\infty)$ is a maximal compact subgroup and $K^\infty <G(\mathbb{A}_k^\infty)$ is a compact open subgroup. 
Since $K$-finite vectors are dense in $L^2(\pi)$ by \cite[Proposition 4.4.3]{GH19}, we may choose $\mathcal{B}(\pi)$ to consist of $K$-finite vectors.
To ease the notation, we set $f:=f_{\mathbf{X},T}$.
Then we have
\begin{align*}
    \sum_{\varphi_i\in \mathcal{B}(\pi)}\int_{[G]}F_{\mathbf{X},T}(g')\varphi_i(g') m^{G}(g')\bar{\varphi}_i(g)=&\sum_{\varphi_i\in \mathcal{B}(\pi)}\int_{[G]}\sum_{x\in X(k)}f(x\cdot g')\varphi_i(g') m^G(g')\bar{\varphi}_i(g)\\
    =&\sum_{\varphi_i\in \mathcal{B}(\pi)}\int_{G(k)\backslash G(\mathbb{A}_k)}\sum_{\substack{\gamma'\in G(k)/\sim \\ \gamma'\sim_{st} \gamma}}\sum_{h\in G_{\gamma'}(k)\backslash G(k)}f(\gamma'\cdot hg')\varphi_i(g') m^G(g')\bar{\varphi}_i(g)\\
      =&\sum_{\varphi_i\in \mathcal{B}(\pi)}\sum_{\substack{\gamma'\in G(k)/\sim \\ \gamma'\sim_{st} \gamma}}
     \int_{G_{\gamma'}(k)\backslash G(\mathbb{A}_k)}f(g'^{-1}\gamma'g')\varphi_i(g') m^G(g')\bar{\varphi}_i(g)\\
    =&\sum_{\varphi_i\in \mathcal{B}(\pi)}\sum_{\substack{\gamma'\in G(k)/\sim \\ \gamma'\sim_{st} \gamma}}
    \int_{G_{\gamma'}(\mathbb{A}_k)\backslash G(\mathbb{A}_k)}f(g'^{-1}\gamma'g')\varphi_i^{G_{\gamma'}}(g') m^X(g') \bar{\varphi}_i(g),
\end{align*}
where $\varphi_i^{G_{\gamma'}}(g'):=\int_{[G_{\gamma'}]} \varphi_i(hg')m^{G_{\gamma'}}(h)$.
Here, $\varphi_i^{G_{\gamma'}}(g')$ converges absolutely since $\varphi_i$ is of moderate growth by \cite[Theorem 6.6.4]{GH19} and $[G_{\gamma'}]$ is compact.
In particular, for the trivial representation $\pi_{tri}$, the $\pi_{tri}$-isotypic subspace $L^2(\pi_{tri})$ of $L^2([G])$ is the set of constant functions on $[G]$ and $\mathcal{B}(\pi_{tri})=\{1\}$.
Therefore the contribution of $\pi_{tri}$ is
\[
\sum_{\substack{\gamma'\in G(k)/\sim \\ \gamma'\sim_{st} \gamma}}\tau(G_{\gamma'})\mathcal{O}_{\gamma'}(\tilde{f}_{\mathbf{X},T}),
\]
where $\tilde{f}_{\mathbf{X},T}$ is a function on $G(\mathbb{A}_k)$ satisfying (\ref{eq:test_function_on_G}).
As in Remark \ref{rmk:main}, applying the deviation for the pre-stabilization in Section \ref{sec:pre_stabilization} to the right-hand side of (\ref{eq:mainthm}), one can deduce that this contribution is asymptotically equivalent to $ N(X;f_{\mathbf{X},T})$.
This shows that our main theorem supports the expectation, introduced at the beginning of this section, that the main term in the asymptotic behavior of  $ N(X;f_{\mathbf{X},T})$ arises from the trivial representation.

\bibliographystyle{alpha}
\bibliography{References}

@incollection {Art,
    AUTHOR = {Arthur, James},
     TITLE = {An introduction to the trace formula},
 BOOKTITLE = {Harmonic analysis, the trace formula, and {S}himura varieties},
    SERIES = {Clay Math. Proc.},
    VOLUME = {4},
     PAGES = {1--263},
 PUBLISHER = {Amer. Math. Soc., Providence, RI},
      YEAR = {2005},
      ISBN = {0-8218-3844-X},
   MRCLASS = {11F72},
  MRNUMBER = {2192011},
MRREVIEWER = {Erez\ M.\ Lapid},
}

@book {Art13,
    AUTHOR = {Arthur, James},
     TITLE = {The endoscopic classification of representations},
    SERIES = {American Mathematical Society Colloquium Publications},
    VOLUME = {61},
      NOTE = {Orthogonal and symplectic groups},
 PUBLISHER = {American Mathematical Society, Providence, RI},
      YEAR = {2013},
     PAGES = {xviii+590},
      ISBN = {978-0-8218-4990-3},
   MRCLASS = {22E55 (11F66 11F70 11F72 11R37 20G25 22E50)},
  MRNUMBER = {3135650},
MRREVIEWER = {Dihua\ Jiang},
       DOI = {10.1090/coll/061},
       URL = {https://doi.org/10.1090/coll/061},
}

@article {BorHar,
    AUTHOR = {Borel, Armand and Harish-Chandra},
     TITLE = {Arithmetic subgroups of algebraic groups},
   JOURNAL = {Ann. of Math. (2)},
  FJOURNAL = {Annals of Mathematics. Second Series},
    VOLUME = {75},
      YEAR = {1962},
     PAGES = {485--535},
      ISSN = {0003-486X},
   MRCLASS = {20.65 (14.50)},
  MRNUMBER = {147566},
MRREVIEWER = {P.\ Cartier},
       DOI = {10.2307/1970210},
       URL = {https://doi.org/10.2307/1970210},
}

@article {BR,
    AUTHOR = {Borovoi, Mikhail and Rudnick, Ze\'{e}v},
     TITLE = {Hardy-{L}ittlewood varieties and semisimple groups},
   JOURNAL = {Invent. Math.},
  FJOURNAL = {Inventiones Mathematicae},
    VOLUME = {119},
      YEAR = {1995},
    NUMBER = {1},
     PAGES = {37--66},
      ISSN = {0020-9910,1432-1297},
   MRCLASS = {11G35 (14M17)},
  MRNUMBER = {1309971},
MRREVIEWER = {A.\ A.\ Bondarenko},
       DOI = {10.1007/BF01245174},
       URL = {https://doi.org/10.1007/BF01245174},
}

@article {BD,
    AUTHOR = {Borovoi, Mikhail and Demarche, Cyril},
     TITLE = {Manin obstruction to strong approximation for homogeneous
              spaces},
   JOURNAL = {Comment. Math. Helv.},
  FJOURNAL = {Commentarii Mathematici Helvetici. A Journal of the Swiss
              Mathematical Society},
    VOLUME = {88},
      YEAR = {2013},
    NUMBER = {1},
     PAGES = {1--54},
      ISSN = {0010-2571,1420-8946},
   MRCLASS = {14M17 (11G35 14F22 14G25)},
  MRNUMBER = {3008912},
MRREVIEWER = {Stefan\ Schr\"oer},
       DOI = {10.4171/CMH/277},
       URL = {https://doi.org/10.4171/CMH/277},
}

@book {Bou,
    AUTHOR = {Bourbaki, Nicolas},
     TITLE = {Algebra {II}. {C}hapters 4--7},
    SERIES = {Elements of Mathematics (Berlin)},
   EDITION = {English},
 PUBLISHER = {Springer-Verlag, Berlin},
      YEAR = {2003},
     PAGES = {viii+461},
      ISBN = {3-540-00706-7},
   MRCLASS = {00A05 (12-01 13-01)},
  MRNUMBER = {1994218},
       DOI = {10.1007/978-3-642-61698-3},
       URL = {https://doi.org/10.1007/978-3-642-61698-3},
}

@article {BO,
    AUTHOR = {Benoist, Yves and Oh, Hee},
     TITLE = {Effective equidistribution of {$S$}-integral points on
              symmetric varieties},
   JOURNAL = {Ann. Inst. Fourier (Grenoble)},
  FJOURNAL = {Universit\'e{} de Grenoble. Annales de l'Institut Fourier},
    VOLUME = {62},
      YEAR = {2012},
    NUMBER = {5},
     PAGES = {1889--1942},
      ISSN = {0373-0956,1777-5310},
   MRCLASS = {11G35},
  MRNUMBER = {3025156},
MRREVIEWER = {Konstantinos\ Draziotis},
       DOI = {10.5802/aif.2738},
       URL = {https://doi.org/10.5802/aif.2738},
}

@article {CT20,
    AUTHOR = {Chenevier, Ga\"etan and Ta\"ibi, Olivier},
     TITLE = {Discrete series multiplicities for classical groups over {$\bf
              Z$} and level 1 algebraic cusp forms},
   JOURNAL = {Publ. Math. Inst. Hautes \'Etudes Sci.},
  FJOURNAL = {Publications Math\'ematiques. Institut de Hautes \'Etudes
              Scientifiques},
    VOLUME = {131},
      YEAR = {2020},
     PAGES = {261--323},
      ISSN = {0073-8301,1618-1913},
   MRCLASS = {11F46 (11F66 11F72 14G35 20H25)},
  MRNUMBER = {4106796},
MRREVIEWER = {\Dbar\cftil o{} Ng\d oc Di\cfudot ep},
       DOI = {10.1007/s10240-020-00115-z},
       URL = {https://doi.org/10.1007/s10240-020-00115-z},
}

@article{CHL,
  title={Orbital integrals and ideal calss monoids for a Bass order},
  author={Cho, Sungmun and Hong, Jungtaek and Lee, Yuchan},
  journal={arXiv:2408.16199v2},
}

@article{CKL,
  title={Orbital integrals for classical Lie algebras and smooth integral models},
  author={Cho, Sungmun and Kang, Taeyeoup and Lee, Yuchan},
  journal={arXiv:2411.16054},
}

@article {Col,
    AUTHOR = {Colliot-Th\'el\`ene, Jean-Louis},
     TITLE = {R\'esolutions flasques des groupes lin\'eaires connexes},
   JOURNAL = {J. Reine Angew. Math.},
  FJOURNAL = {Journal f\"ur die Reine und Angewandte Mathematik. [Crelle's
              Journal]},
    VOLUME = {618},
      YEAR = {2008},
     PAGES = {77--133},
      ISSN = {0075-4102,1435-5345},
   MRCLASS = {11E72 (12G05 20G15)},
  MRNUMBER = {2404747},
MRREVIEWER = {B.\ Sury},
       DOI = {10.1515/CRELLE.2008.034},
       URL = {https://doi.org/10.1515/CRELLE.2008.034},
}

@article {CtX,
    AUTHOR = {Colliot-Th\'el\`ene, Jean-Louis and Xu, Fei},
     TITLE = {Brauer-{M}anin obstruction for integral points of homogeneous
              spaces and representation by integral quadratic forms},
      NOTE = {With an appendix by Dasheng Wei and Xu},
   JOURNAL = {Compos. Math.},
  FJOURNAL = {Compositio Mathematica},
    VOLUME = {145},
      YEAR = {2009},
    NUMBER = {2},
     PAGES = {309--363},
      ISSN = {0010-437X,1570-5846},
   MRCLASS = {11G35 (11D85 11E12 14F22 14G25 20G30)},
  MRNUMBER = {2501421},
MRREVIEWER = {Tam\'as\ Szamuely},
       DOI = {10.1112/S0010437X0800376X},
       URL = {https://doi.org/10.1112/S0010437X0800376X},
}

@book{CGP15, 
place={Cambridge}, edition={2}, series={New Mathematical Monographs}, title={Pseudo-reductive Groups}, publisher={Cambridge University Press}, author={Conrad, Brian and Gabber, Ofer and Prasad, Gopal}, year={2015}, collection={New Mathematical Monographs}}

@article {DRS,
    AUTHOR = {Duke, W. and Rudnick, Z. and Sarnak, P.},
     TITLE = {Density of integer points on affine homogeneous varieties},
   JOURNAL = {Duke Math. J.},
  FJOURNAL = {Duke Mathematical Journal},
    VOLUME = {71},
      YEAR = {1993},
    NUMBER = {1},
     PAGES = {143--179},
      ISSN = {0012-7094,1547-7398},
   MRCLASS = {11G99 (11P21)},
  MRNUMBER = {1230289},
MRREVIEWER = {D.\ Mili\v ci\'c},
       DOI = {10.1215/S0012-7094-93-07107-4},
       URL = {https://doi.org/10.1215/S0012-7094-93-07107-4},
}

@article {EM,
    AUTHOR = {Eskin, Alex and McMullen, Curt},
     TITLE = {Mixing, counting, and equidistribution in {L}ie groups},
   JOURNAL = {Duke Math. J.},
  FJOURNAL = {Duke Mathematical Journal},
    VOLUME = {71},
      YEAR = {1993},
    NUMBER = {1},
     PAGES = {181--209},
      ISSN = {0012-7094,1547-7398},
   MRCLASS = {22E40 (57S30 58F17)},
  MRNUMBER = {1230290},
MRREVIEWER = {Nimish\ A.\ Shah},
       DOI = {10.1215/S0012-7094-93-07108-6},
       URL = {https://doi.org/10.1215/S0012-7094-93-07108-6},
}

@article {EMS,
    AUTHOR = {Eskin, Alex and Mozes, Shahar and Shah, Nimish},
     TITLE = {Unipotent flows and counting lattice points on homogeneous
              varieties},
   JOURNAL = {Ann. of Math. (2)},
  FJOURNAL = {Annals of Mathematics. Second Series},
    VOLUME = {143},
      YEAR = {1996},
    NUMBER = {2},
     PAGES = {253--299},
      ISSN = {0003-486X,1939-8980},
   MRCLASS = {22E40 (11E57 11P21)},
  MRNUMBER = {1381987},
MRREVIEWER = {Alexander\ Starkov},
       DOI = {10.2307/2118644},
       URL = {https://doi.org/10.2307/2118644},
}

@article {FMT,
    AUTHOR = {Franke, Jens and Manin, Yuri I. and Tschinkel, Yuri},
     TITLE = {Rational points of bounded height on {F}ano varieties},
   JOURNAL = {Invent. Math.},
  FJOURNAL = {Inventiones Mathematicae},
    VOLUME = {95},
      YEAR = {1989},
    NUMBER = {2},
     PAGES = {421--435},
      ISSN = {0020-9910,1432-1297},
   MRCLASS = {11G35 (14G25 14J20)},
  MRNUMBER = {974910},
MRREVIEWER = {Joseph\ H.\ Silverman},
       DOI = {10.1007/BF01393904},
       URL = {https://doi.org/10.1007/BF01393904},
}

@article {Get18,
    AUTHOR = {Getz, Jayce R.},
     TITLE = {Secondary terms in asymptotics for the number of zeros of
              quadratic forms over number fields},
   JOURNAL = {J. Lond. Math. Soc. (2)},
  FJOURNAL = {Journal of the London Mathematical Society. Second Series},
    VOLUME = {98},
      YEAR = {2018},
    NUMBER = {2},
     PAGES = {275--305},
      ISSN = {0024-6107,1469-7750},
   MRCLASS = {11D45 (11E12 11E45 11N45)},
  MRNUMBER = {3873109},
MRREVIEWER = {Julia\ Brandes},
       DOI = {10.1112/jlms.12130},
       URL = {https://doi.org/10.1112/jlms.12130},
}

@article{GH19,
  title={An introduction to automorphic representations with a view towards trace formulae},
  author={Getz, Jayce R and Hahn, Heekyoung},
  journal={Graduate Studies in Mathematics},
  volume={6},
  year={2019},
  publisher={Springer}
}

@incollection {Gor,
    AUTHOR = {Gordon, Julia},
     TITLE = {Orbital integrals and normalizations of measures},
 BOOKTITLE = {On the {L}anglands program---endoscopy and beyond},
    SERIES = {Lect. Notes Ser. Inst. Math. Sci. Natl. Univ. Singap.},
    VOLUME = {43},
     PAGES = {247--314},
 PUBLISHER = {World Sci. Publ., Hackensack, NJ},
      YEAR = {2024},
      ISBN = {[9789811285820]; [9789811285837]; [9789811285813]},
   MRCLASS = {20G25 (17B08)},
  MRNUMBER = {4739679},
}

@article {GG99,
    AUTHOR = {Gross, Benedict H. and Gan, Wee Teck},
     TITLE = {Haar measure and the {A}rtin conductor},
   JOURNAL = {Trans. Amer. Math. Soc.},
  FJOURNAL = {Transactions of the American Mathematical Society},
    VOLUME = {351},
      YEAR = {1999},
    NUMBER = {4},
     PAGES = {1691--1704},
      ISSN = {0002-9947},
   MRCLASS = {20G30 (11G99)},
  MRNUMBER = {1458303},
MRREVIEWER = {Stefan K\"{u}hnlein},
       DOI = {10.1090/S0002-9947-99-02095-4},
       URL = {https://doi.org/10.1090/S0002-9947-99-02095-4},
}

@article {Gro97,
    AUTHOR = {Gross, Benedict H.},
     TITLE = {On the motive of a reductive group},
   JOURNAL = {Invent. Math.},
  FJOURNAL = {Inventiones Mathematicae},
    VOLUME = {130},
      YEAR = {1997},
    NUMBER = {2},
     PAGES = {287--313},
      ISSN = {0020-9910},
   MRCLASS = {20G35 (11G09 11G40 14F99 20G30)},
  MRNUMBER = {1474159},
MRREVIEWER = {Dipendra Prasad},
       DOI = {10.1007/s002220050186},
       URL = {https://doi.org/10.1007/s002220050186},
}

@book {JacLang,
    AUTHOR = {Jacquet, H. and Langlands, R. P.},
     TITLE = {Automorphic forms on {${\rm GL}(2)$}},
    SERIES = {Lecture Notes in Mathematics},
    VOLUME = {Vol. 114},
 PUBLISHER = {Springer-Verlag, Berlin-New York},
      YEAR = {1970},
     PAGES = {vii+548},
   MRCLASS = {10D15 (12A65 12A70 22E55)},
  MRNUMBER = {401654},
MRREVIEWER = {Stephen\ Gelbart},
}

@article{JL,
  title={An asymptotic formula for the number of integral matrices with a fixed characteristic polynomial via orbital integrals},
  author={Seongsu Jeon and Yuchan Lee},
  journal={arXiv preprint arXiv:2501.00284},
  year={}
}

@incollection {Kal,
    AUTHOR = {Kaletha, Tasho},
     TITLE = {Lectures on the stable trace formula with emphasis on {$\rm
              SL_2$}},
 BOOKTITLE = {On the {L}anglands program---endoscopy and beyond},
    SERIES = {Lect. Notes Ser. Inst. Math. Sci. Natl. Univ. Singap.},
    VOLUME = {43},
     PAGES = {45--139},
 PUBLISHER = {World Sci. Publ., Hackensack, NJ},
      YEAR = {2024},
      ISBN = {[9789811285820]; [9789811285837]; [9789811285813]},
   MRCLASS = {11F72},
  MRNUMBER = {4739676},
}

@book {Knapp,
    AUTHOR = {Knapp, Anthony W.},
     TITLE = {Lie groups beyond an introduction},
    SERIES = {Progress in Mathematics},
    VOLUME = {140},
   EDITION = {Second},
 PUBLISHER = {Birkh\"{a}user Boston, Inc., Boston, MA},
      YEAR = {2002},
     PAGES = {xviii+812},
      ISBN = {0-8176-4259-5},
   MRCLASS = {22-01},
  MRNUMBER = {1920389},
}

@article {Kot82,
    AUTHOR = {Kottwitz, Robert E.},
     TITLE = {Rational conjugacy classes in reductive groups},
   JOURNAL = {Duke Math. J.},
  FJOURNAL = {Duke Mathematical Journal},
    VOLUME = {49},
      YEAR = {1982},
    NUMBER = {4},
     PAGES = {785--806},
      ISSN = {0012-7094,1547-7398},
   MRCLASS = {20G15 (22E55)},
  MRNUMBER = {683003},
MRREVIEWER = {Andy\ R.\ Magid},
       URL = {http://projecteuclid.org/euclid.dmj/1077315531},
}

@article {Kot84,
    AUTHOR = {Kottwitz, Robert E.},
     TITLE = {Stable trace formula: cuspidal tempered terms},
   JOURNAL = {Duke Math. J.},
  FJOURNAL = {Duke Mathematical Journal},
    VOLUME = {51},
      YEAR = {1984},
    NUMBER = {3},
     PAGES = {611--650},
      ISSN = {0012-7094,1547-7398},
   MRCLASS = {11R39 (11F70 11F72 22E55)},
  MRNUMBER = {757954},
MRREVIEWER = {Jean-Pierre\ Labesse},
       DOI = {10.1215/S0012-7094-84-05129-9},
       URL = {https://doi.org/10.1215/S0012-7094-84-05129-9},
}

@article {Kot86,
    AUTHOR = {Kottwitz, Robert E.},
     TITLE = {Stable trace formula: elliptic singular terms},
   JOURNAL = {Math. Ann.},
  FJOURNAL = {Mathematische Annalen},
    VOLUME = {275},
      YEAR = {1986},
    NUMBER = {3},
     PAGES = {365--399},
      ISSN = {0025-5831,1432-1807},
   MRCLASS = {22E55 (11F70 11F72)},
  MRNUMBER = {858284},
MRREVIEWER = {J.\ Szmidt},
       DOI = {10.1007/BF01458611},
       URL = {https://doi.org/10.1007/BF01458611},
}

@article {KS,
    AUTHOR = {Kottwitz, Robert E. and Shelstad, Diana},
     TITLE = {Foundations of twisted endoscopy},
   JOURNAL = {Ast\'erisque},
  FJOURNAL = {Ast\'erisque},
    NUMBER = {255},
      YEAR = {1999},
     PAGES = {vi+190},
      ISSN = {0303-1179,2492-5926},
   MRCLASS = {22E55 (11F70 11R34 22-02 22E50)},
  MRNUMBER = {1687096},
MRREVIEWER = {Volker\ J.\ Heiermann},
}

@book {KP23,
    AUTHOR = {Kaletha, Tasho and Prasad, Gopal},
     TITLE = {Bruhat-{T}its theory---a new approach},
    SERIES = {New Mathematical Monographs},
    VOLUME = {44},
 PUBLISHER = {Cambridge University Press, Cambridge},
      YEAR = {2023},
     PAGES = {xxx+718},
      ISBN = {978-1-108-83196-3},
   MRCLASS = {20E42 (11F70 20G25 22E50)},
  MRNUMBER = {4520154},
}

@article {Lang63,
    AUTHOR = {Langlands, R. P.},
     TITLE = {The dimension of spaces of automorphic forms},
   JOURNAL = {Amer. J. Math.},
  FJOURNAL = {American Journal of Mathematics},
    VOLUME = {85},
      YEAR = {1963},
     PAGES = {99--125},
      ISSN = {0002-9327,1080-6377},
   MRCLASS = {22.60 (10.23)},
  MRNUMBER = {156362},
MRREVIEWER = {S.\ Helgason},
       DOI = {10.2307/2373189},
       URL = {https://doi.org/10.2307/2373189},
}

@article {LS,
    AUTHOR = {Langlands, R. P. and Shelstad, D.},
     TITLE = {On the definition of transfer factors},
   JOURNAL = {Math. Ann.},
  FJOURNAL = {Mathematische Annalen},
    VOLUME = {278},
      YEAR = {1987},
    NUMBER = {1-4},
     PAGES = {219--271},
      ISSN = {0025-5831,1432-1807},
   MRCLASS = {11S37 (11F70 11F72 22E35 22E45 22E50)},
  MRNUMBER = {909227},
MRREVIEWER = {Ernst-Wilhelm\ Zink},
       DOI = {10.1007/BF01458070},
       URL = {https://doi.org/10.1007/BF01458070},
}

@article{Lee21,
  title={Counting algebraic tori over $\mathbb{Q}$  by {A}rtin conductor},
  author={Lee, Jungin},
  journal={arXiv preprint arXiv:2104.02855},
  year={}
}

@article {LL,
    AUTHOR = {Labesse, J.-P. and Langlands, R. P.},
     TITLE = {{$L$}-indistinguishability for {${\rm SL}(2)$}},
   JOURNAL = {Canadian J. Math.},
  FJOURNAL = {Canadian Journal of Mathematics. Journal Canadien de
              Math\'ematiques},
    VOLUME = {31},
      YEAR = {1979},
    NUMBER = {4},
     PAGES = {726--785},
      ISSN = {0008-414X,1496-4279},
   MRCLASS = {22E50 (10D40 12A67 22E55)},
  MRNUMBER = {540902},
MRREVIEWER = {Joe\ Repka},
       DOI = {10.4153/CJM-1979-070-3},
       URL = {https://doi.org/10.4153/CJM-1979-070-3},
}

@book {Neu,
    AUTHOR = {Neukirch, J\"{u}rgen},
     TITLE = {Algebraic number theory},
    SERIES = {Grundlehren der mathematischen Wissenschaften [Fundamental
              Principles of Mathematical Sciences]},
    VOLUME = {322},
      NOTE = {Translated from the 1992 German original and with a note by
              Norbert Schappacher,
              With a foreword by G. Harder},
 PUBLISHER = {Springer-Verlag, Berlin},
      YEAR = {1999},
     PAGES = {xviii+571},
      ISBN = {3-540-65399-6},
   MRCLASS = {11Rxx (11-02 11S15 11S31 14C40)},
  MRNUMBER = {1697859},
MRREVIEWER = {Cornelius Greither},
       DOI = {10.1007/978-3-662-03983-0},
       URL = {https://doi.org/10.1007/978-3-662-03983-0},
}

@article {Ono65,
    AUTHOR = {Ono, Takashi},
     TITLE = {On the relative theory of {T}amagawa numbers},
   JOURNAL = {Ann. of Math. (2)},
  FJOURNAL = {Annals of Mathematics. Second Series},
    VOLUME = {82},
      YEAR = {1965},
     PAGES = {88--111},
      ISSN = {0003-486X},
   MRCLASS = {10.25 (14.50)},
}

@book {Poonen,
    AUTHOR = {Poonen, Bjorn},
     TITLE = {Rational points on varieties},
    SERIES = {Graduate Studies in Mathematics},
    VOLUME = {186},
 PUBLISHER = {American Mathematical Society, Providence, RI},
      YEAR = {2017},
     PAGES = {xv+337},
      ISBN = {978-1-4704-3773-2},
   MRCLASS = {14G05 (11G35)},
  MRNUMBER = {3729254},
MRREVIEWER = {Daniel\ Loughran},
       DOI = {10.1090/gsm/186},
       URL = {https://doi.org/10.1090/gsm/186},
}

@book {PR94,
    AUTHOR = {Platonov, Vladimir and Rapinchuk, Andrei},
     TITLE = {Algebraic groups and number theory},
    SERIES = {Pure and Applied Mathematics},
    VOLUME = {139},
      NOTE = {Translated from the 1991 Russian original by Rachel Rowen},
 PUBLISHER = {Academic Press, Inc., Boston, MA},
      YEAR = {1994},
     PAGES = {xii+614},
      ISBN = {0-12-558180-7},
   MRCLASS = {11E57 (11-02 20Gxx)},
  MRNUMBER = {1278263},
}

@article {Sch,
    AUTHOR = {Schmidt, Wolfgang M.},
     TITLE = {The density of integer points on homogeneous varieties},
   JOURNAL = {Acta Math.},
  FJOURNAL = {Acta Mathematica},
    VOLUME = {154},
      YEAR = {1985},
    NUMBER = {3-4},
     PAGES = {243--296},
      ISSN = {0001-5962,1871-2509},
   MRCLASS = {11D72 (11P55 11R09 11R45)},
  MRNUMBER = {781588},
MRREVIEWER = {D.\ J.\ Lewis},
       DOI = {10.1007/BF02392473},
       URL = {https://doi.org/10.1007/BF02392473},
}

@book {Ser2,
    AUTHOR = {Serre, Jean-Pierre},
     TITLE = {Galois cohomology},
      NOTE = {Translated from the French by Patrick Ion and revised by the
              author},
 PUBLISHER = {Springer-Verlag, Berlin},
      YEAR = {1997},
     PAGES = {x+210},
      ISBN = {3-540-61990-9},
   MRCLASS = {12G05 (11R34)},
  MRNUMBER = {1466966},
       DOI = {10.1007/978-3-642-59141-9},
       URL = {https://doi.org/10.1007/978-3-642-59141-9},
}

@book {St,
    AUTHOR = {Steinberg, Robert},
     TITLE = {Endomorphisms of linear algebraic groups},
    SERIES = {Memoirs of the American Mathematical Society},
    VOLUME = {No. 80},
 PUBLISHER = {American Mathematical Society, Providence, RI},
      YEAR = {1968},
     PAGES = {108},
   MRCLASS = {14.50 (22.00)},
  MRNUMBER = {230728},
MRREVIEWER = {E.\ Abe},
}

@article {St2,
    AUTHOR = {Steinberg, Robert},
     TITLE = {Regular elements of semisimple algebraic groups},
   JOURNAL = {Inst. Hautes \'Etudes Sci. Publ. Math.},
  FJOURNAL = {Institut des Hautes \'Etudes Scientifiques. Publications
              Math\'ematiques},
    NUMBER = {25},
      YEAR = {1965},
     PAGES = {49--80},
      ISSN = {0073-8301,1618-1913},
   MRCLASS = {14.50 (20.75)},
  MRNUMBER = {180554},
MRREVIEWER = {J.\ L.\ Koszul},
       URL = {http://www.numdam.org/item?id=PMIHES_1965__25__49_0},
}

@incollection {Tom,
    AUTHOR = {Hales, Thomas C.},
     TITLE = {A statement of the fundamental lemma},
 BOOKTITLE = {Harmonic analysis, the trace formula, and {S}himura varieties},
    SERIES = {Clay Math. Proc.},
    VOLUME = {4},
     PAGES = {643--658},
 PUBLISHER = {Amer. Math. Soc., Providence, RI},
      YEAR = {2005},
      ISBN = {0-8218-3844-X},
   MRCLASS = {22E50 (11F70 22E55)},
  MRNUMBER = {2192018},
MRREVIEWER = {Ernst-Wilhelm\ Zink},
}

@incollection {VP,
    AUTHOR = {Vinberg, \`E. B. and Popov, V. L.},
     TITLE = {Invariant theory},
 BOOKTITLE = {Algebraic geometry, 4 ({R}ussian)},
    SERIES = {Itogi Nauki i Tekhniki},
     PAGES = {137--314, 315},
 PUBLISHER = {Akad. Nauk SSSR, Vsesoyuz. Inst. Nauchn. i Tekhn. Inform.,
              Moscow},
      YEAR = {1989},
   MRCLASS = {14D25 (13A50 14L30 15A72)},
  MRNUMBER = {1100485},
MRREVIEWER = {P.\ E.\ Newstead},
}

@book{Vau, place={Cambridge}, edition={2}, series={Cambridge Tracts in Mathematics}, title={The Hardy-Littlewood Method}, publisher={Cambridge University Press}, author={Vaughan, R. C.}, year={1997}, collection={Cambridge Tracts in Mathematics}}

@book {Vosk,
	AUTHOR = {Voskresenski\u{\i}, V. E.},
	TITLE = {Algebraic groups and their birational invariants},
	SERIES = {Translations of Mathematical Monographs},
	VOLUME = {179},
	NOTE = {Translated from the Russian manuscript by Boris Kunyavski
	[Boris \`E. Kunyavski\u{\i}]},
	PUBLISHER = {American Mathematical Society, Providence, RI},
	YEAR = {1998},
	PAGES = {xiv+218},
	ISBN = {0-8218-0905-9},
	MRCLASS = {20G30 (11E72 14G05 14G25)},
	MRNUMBER = {1634406},
	MRREVIEWER = {Philippe\ Gille},
	DOI = {10.1090/mmono/179},
	URL = {https://doi.org/10.1090/mmono/179},
}

@book {Wei82,
    AUTHOR = {Weil, Andr\'{e}},
     TITLE = {Adeles and algebraic groups},
    SERIES = {Progress in Mathematics},
    VOLUME = {23},
      NOTE = {With appendices by M. Demazure and Takashi Ono},
 PUBLISHER = {Birkh\"{a}user, Boston, MA},
      YEAR = {1982},
     PAGES = {iii+126},
      ISBN = {3-7643-3092-9},
   MRCLASS = {10C30 (12A82 12A85 20G35)},
  MRNUMBER = {670072},
MRREVIEWER = {K.\ F.\ Lai},
}

@article {WX,
    AUTHOR = {Wei, Dasheng and Xu, Fei},
     TITLE = {Counting integral points in certain homogeneous spaces},
   JOURNAL = {J. Algebra},
  FJOURNAL = {Journal of Algebra},
    VOLUME = {448},
      YEAR = {2016},
     PAGES = {350--398},
      ISSN = {0021-8693,1090-266X},
   MRCLASS = {11G35 (11E72 11R37 14F22 14G25 20G30 20G35)},
  MRNUMBER = {3438314},
MRREVIEWER = {J\"{o}rg\ Jahnel},
       DOI = {10.1016/j.jalgebra.2015.09.043},
       URL = {https://doi.org/10.1016/j.jalgebra.2015.09.043},
}

@incollection {Yun13,
    AUTHOR = {Yun, Zhiwei},
     TITLE = {Orbital integrals and {D}edekind zeta functions},
 BOOKTITLE = {The legacy of {S}rinivasa {R}amanujan},
    SERIES = {Ramanujan Math. Soc. Lect. Notes Ser.},
    VOLUME = {20},
     PAGES = {399--420},
 PUBLISHER = {Ramanujan Math. Soc., Mysore},
      YEAR = {2013},
      ISBN = {978-93-80416-13-7},
   MRCLASS = {22E35 (11R54)},
  MRNUMBER = {3221323},
MRREVIEWER = {Ivan\ Mati\'{c}},
}

\end{document}